\documentclass[final,notitlepage,12pt,reqno]{amsart}

\usepackage[table]{xcolor}
\usepackage{etoolbox,mathtools}
\makeatletter
\patchcmd{\@makefnmark}{\fontsize}{\check@mathfonts\fontsize}{}{}
\makeatother

\makeatletter

\DeclareFontFamily{U}{mathx}{}
\DeclareFontShape{U}{mathx}{m}{n}{<-> mathx10}{}
\DeclareSymbolFont{mathx}{U}{mathx}{m}{n}
\DeclareMathAccent{\widecheck}{0}{mathx}{"71}

\def\smallunderbrace#1{\mathop{\vtop{\m@th\ialign{##\crcr
   $\hfil\displaystyle{#1}\hfil$\crcr
   \noalign{\kern3\p@\nointerlineskip}%
   \tiny\upbracefill\crcr\noalign{\kern3\p@}}}}\limits}
\makeatother

\UseRawInputEncoding
\usepackage{graphicx,extarrows}
\usepackage{calc}
\usepackage{url}
\newlength{\depthofsumsign}
\makeatletter
\let\I\@undefined
\makeatother

\usepackage{geometry}
\usepackage{indentfirst}
\usepackage[normalem]{ulem}
\usepackage{float}
\usepackage{amsthm}

\usepackage{enumitem}[2011/09/28]
\setenumerate{align=left, leftmargin=0pt,labelsep=.5em, labelindent=0\parindent,listparindent=\parindent,itemindent=*}
\usepackage{array}
\usepackage[all]{xy}

\usepackage{exscale,relsize}
\usepackage{multirow}

\usepackage{caption}[2012/02/19]

\usepackage{longtable,lscape}

\usepackage{colonequals}

\usepackage{amssymb,bm,amsmath}
\usepackage{amsfonts}

\usepackage[OT2,T1,T2A]{fontenc}
\usepackage[utf8x]{inputenc}
\usepackage[french,german,russian,english]{babel}
\usepackage{appendix}

\DeclareMathOperator{\Span}{span}

\DeclareMathOperator*{\fib}{fib}

\DeclareMathOperator*{\vfib}{\widetilde{fib}}
\DeclareMathOperator{\Lyn}{Lyn}

\DeclareMathOperator{\Shf}{\text{\cyrins{\textsf{Ш}}}}
\DeclareMathOperator{\shf}{\text{\cyrins{\textsf{ш}}}}

\DeclareMathOperator{\Ls}{Ls}

\DeclareMathOperator{\Li}{Li}

\DeclareMathOperator{\D}{d}
\DeclareMathOperator{\I}{Im}
\DeclareMathOperator{\RE}{Re}
\DeclareMathOperator*{\Reg}{Reg}

\def\eor{\hfill$ \square$}

\newcolumntype{L}{>{$}l<{$}}
\newcolumntype{C}{>{$}c<{$}}
\newcolumntype{R}{>{$}r<{$}}

\theoremstyle{plain}
\newtheorem{theorem}{Theorem}[section]
\newtheorem{proposition}[theorem]{Proposition}

\newenvironment{remark}[1][Remark.]{\begin{trivlist}
\item[\hskip \labelsep {\bfseries #1}]}{\end{trivlist}}

\theoremstyle{definition}

\numberwithin{equation}{section}

\usepackage{Baskervaldx}
\usepackage[baskervaldx]{newtxmath}
\usepackage[cal=cm,scr=rsfs,frak=euler]{mathalfa}

\usepackage[scr=rsfs]{mathalfa}
\DeclareMathAlphabet{\mathsf}{OT1}{\sfdefault}{m}{n}
\SetMathAlphabet{\mathsf}{bold}{OT1}{\sfdefault}{m}{n}
\DeclareSymbolFontAlphabet{\mathbb}{AMSb}
\DeclareRobustCommand{\cyrins}[1]{%
  \begingroup\fontfamily{erewhon-TLF}%
  \foreignlanguage{russian}{#1}%
  \endgroup
}

\usepackage{color}

\def\Z{\Bbb Z}

\def\bg{\bigg}
\def\({\bg(}
\def\){\bg)}

\begin{document}

\pagenumbering{roman}
\selectlanguage{english}
\title{Multiple Clausen values and  deformed Ap\'ery-like series}\author{Zhi-Wei Sun}\address[Z.-W.\ Sun]{School of Mathematics, Nanjing
University, Nanjing 210093, People's Republic of China}
\email{{\tt zwsun@nju.edu.cn}
\newline\indent
{\it Homepage}: {\tt http://maths.nju.edu.cn/\lower0.5ex\hbox{\~{}}zwsun}}
 \author{Yajun Zhou}
\address[Y. Zhou]{Program in Applied and Computational Mathematics (PACM), Princeton University, Princeton, NJ 08544} \email{yajunz@math.princeton.edu}\curraddr{\textrm{} \textsc{Academy of Advanced Interdisciplinary Studies (AAIS), Peking University, Beijing 100871, P. R. China}}\email{yajun.zhou.1982@pku.edu.cn}

\date{\today}\thanks{\indent * Z.-W. Sun was supported by the Natural Science Foundation of China (grant no.\ 12371004). Y. Zhou was supported in part  by the Applied Mathematics Program within the Department of Energy
(DOE) Office of Advanced Scientific Computing Research (ASCR) as part of the Collaboratory on
Mathematics for Mesoscopic Modeling of Materials (CM4)}


\begin{abstract}With generalized central binomial coefficients $ \binom{2x}{x}:=\frac{\Gamma(2x+1)}{[\Gamma(x+1)]^2}$ defined through Euler's gamma function, we represent deformed Ap\'ery-like series \[ \mathscr A_{s,n}:=\sum_{k=1}^\infty\left.\!\frac{\partial^n}{\partial x^n}\frac{1}{x^s\binom{2x}{x}}\right|_{x=k} \] by multiple Clausen values (MCVs), which belong to a special class of cyclotomic multiple zeta values (CMZVs) at level $3$. For example, exploiting provable algebraic relations among MCVs, we show that \[\mathscr A_{1,5}=-\frac{9[495L(\chi_{-3},6)-30\pi^{2}L(\chi_{-3},4)-2\pi^{4}L(\chi_{-3},2)]}{4}\]and\[\mathscr A_{4,4}=\frac{352\zeta_{5,3}}{15}+\frac{752537\pi^{8}}{10206000},\]where $ L(\chi_{-3},s):=\sum_{n=0}^\infty\left[(3n+1)^{-s}-(3n+2)^{-s}\right]$ and $ \zeta_{5,3}:=\sum_{m>n>0} m^{-5}n^{-3}$.\end{abstract}

\maketitle
\pagenumbering{arabic}

\section{Introduction\label{sec:intro}}  Let $ \binom{2k}{k}\colonequals \frac{(2k)!}{(k!)^2}$ be the central binomial coefficent. The infinite series\begin{align}
\sum_{k=1}^\infty\frac{1}{k^2\binom{2k}{k}}=\frac{\zeta_2}{3}\text{ and }\sum_{k=1}^\infty\frac{(-1)^{k-1}}{k^3\binom{2k}{k}}=\frac{2\zeta_{3}}{5}
\label{eq:A2A3}\end{align}played crucial parts in Ap\'ery's proof \cite{Apery1978,Poorten1978} for the irrationality of $ \zeta_2\colonequals \sum_{k=1}^\infty\frac{1}{k^2}=\frac{\pi^2}{6}$ and   $ \zeta_3\colonequals \sum_{k=1}^\infty\frac{1}{k^3}$.
 Subsequently, for positive integers $ s$, the Ap\'ery-like series\begin{align}
\mathscr A_s\colonequals \sum_{k=1}^\infty\frac{1}{k^s\binom{2k}{k}}
\end{align}were systematically investigated by Zucker \cite{Zucker1985}  and Borwein--Broadhurst--Kamnitzer \cite{BorweinBroadhurstKamnitzer2001}, culminating in an elegant integral representation (cf.\ \cite[(2.5)]{Zucker1985} or \cite[Lemma 3.1]{BorweinBroadhurstKamnitzer2001}) for $ s\in\mathbb Z_{>1}$:
\begin{align}
\mathscr A_s=\frac{(-2)^{s-2}}{(s-2)!}\int_0^{\pi/3}\theta\log^{s-2}\left( 2\sin\frac{\theta}{2} \right)\D\theta.\label{eq:AsLs}
\end{align}
 The integral on the right-hand side of the equation above can be evaluated symbolically, by the \texttt{logsine} package (current version known as \texttt{LsToLi} v2.0) of Borwein--Straub \cite{BorweinStraub2011ISSAC}. The special value $ \mathscr A_1=\frac{\pi}{3\sqrt{3}}$ does not fit into the aforementioned pattern, which deserves a separate  treatment.

Sun \cite{Sun2026dSum} has  recently performed numerical experiments on deformations of Ap\'ery-like series:\begin{align}
\mathscr A_{s,n}\colonequals \sum_{k=1}^\infty\left.\!\frac{\partial^n}{\partial x^n}\frac{1}{x^s\binom{2x}{x}}\right|_{x=k}\equiv \sum_{k=1}^\infty\left.\!\frac{\partial^n}{\partial x^n}\frac{[\Gamma(x)]^{2}}{2x^{s-1}\Gamma(2x)}\right|_{x=k},
\end{align}where $ \Gamma(x)\colonequals \int_0^\infty t^{x-1}e^{-t}\D t$ is Euler's gamma function for $ x>0$. Empirically, many instances of these deformed Ap\'ery-like series can be represented via members in the $ \mathbb Q$-vector space spanned by cyclotomic multiple zeta values (CMZVs) of weight $ w\in\mathbb Z_{>0}$ and level $N\in\mathbb Z_{>0}$:\begin{align}
 \mathfrak Z_{w}(N)\colonequals\Span_{\mathbb Q}\left\{\Li_{a_1,\dots,a_n}(z_1,\dots,z_n)\colonequals \smash[b]{\sum\limits_{\ell_{1}>\dots>\ell_{n}>0}\prod\limits_{j=1}^n\frac{\smash[t]{z_{j}^{\ell_{j}}}}{\ell_j^{a_j}}}\middle|\begin{smallmatrix}a_1,\dots,a_n\in\mathbb Z_{>0}\\z_{1}^{N}=\dots=z_n^N=1\\\sum _{j=1}^{n}a_{j}=w\\(a_1,z_1)\neq(1,1)\end{smallmatrix} \right\}.\label{eq:Zk(N)}
\end{align} In \S\S\ref{sec:CMZVstruct}--\ref{sec:evalClGl} of this work, we will give theoretical justifications for these CMZV characterizations, as summarized by the three theorems below.\begin{theorem}
\label{thm:A1n}For all $ n\in\mathbb Z_{\geq0}$, we have \begin{align}
\mathscr A_{1,n}\in i\sqrt{3}\mathfrak Z_{n+1}(3),\label{eq:A1nCMZV3}
\end{align}with the retroactive definition that $ \mathscr A_{1,0}\colonequals \mathscr A_1$.\end{theorem}\begin{theorem}\label{thm:Asn}For all $ s\in\mathbb Z_{>1}$ and $ n\in\mathbb Z_{\geq0}$, we have 
\begin{align}
\mathscr A_{s,n}\in \mathfrak Z_{n+s}(3),\label{eq:AsnCMZV3}
\end{align}with the retroactive definition that $ \mathscr A_{s,0}\colonequals \mathscr A_s$.\end{theorem}\begin{theorem}\label{thm:A2nA4n}For all  $ n\in\mathbb Z_{\geq0}$, we have \begin{align}\mathscr A_{4,n}\in{}&\mathfrak Z_{n+4}(1).\label{eq:A4nMZV}
\end{align}
\end{theorem}\begin{remark}Using Wilf--Zeilberger (WZ) pairs \cite{WZ1990},  Hou--Sun proved that (cf.\ \cite[Theorem 1.3]{HouSun2026})\begin{align}
\mathscr A_{2,n}\in{}&\mathfrak Z_{n+2}(1),\label{eq:A2nMZV}
\end{align} and that (cf.\ \cite[Theorem 1.4]{HouSun2026})  \begin{align}
\bar{\mathscr A}_{3,n}\colonequals \sum_{k=1}^\infty\left.\!\frac{\partial^n}{\partial x^n}\frac{(-1)^{x}}{x^3\binom{2x}{x}}\right|_{x=k}\equiv \sum_{k=1}^\infty\left.\!\frac{\partial^n}{\partial x^n}\frac{e^{\pi i x}[\Gamma(x)]^{2}}{2x^{2}\Gamma(2x)}\right|_{x=k}\in\mathfrak Z_{n+3}(1),
\end{align} the latter of which generalized Ap\'ery's formula  $ \bar{\mathscr A}_3\equiv \bar{\mathscr A}_{3,0}=-\frac{2\zeta_{3}}{5}$ [cf.\ \eqref{eq:A2A3}]. On Mar. 5, 2026, the MZV expressions for $ \mathscr A_{4,2},\dots,\mathscr A_{4,5}$ were announced as conjectures by the first-named author  in Question 508768 of  MathOverflow \cite{Sun2024}, which subsequently triggered Deyi Chen's empirical evaluation of $ \mathscr A_{4,6}$. However, there seem to be no ``naturally-occurring'' WZ pairs associated to  $ \mathscr A_{4,n}$, which partly explained the lack of  their proofs before the current work. In \S\ref{subsec:descClausen} of this article, we will treat \eqref{eq:A4nMZV} and  \eqref{eq:A2nMZV} in a unified framework that does not draw on WZ pairs.
\eor\end{remark}\begin{remark}Our demonstrations of Theorems \ref{thm:A1n}--\ref{thm:A2nA4n} will be both constructive and algorithmic, providing the readers with explicit  $ \mathbb Q$-linear combinations of multiple polylogarithms (MPLs) $\Li_{a_1,\dots,a_n}(z_1,\dots,z_n) $ whose arguments $ z_1,\dots,z_n$ are roots of unity.  With the aid of several software packages that  process   $ \mathbb Q$-linear relations among MPLs (to be elaborated in \S\ref{sec:evalClGl}), we may further reduce the MPL representations of $ \mathscr A_{1,n}$ and $ \mathscr A_{s,n}$ to the entries of Tables   \ref{tab:A1n}--\ref{tab:A2nA4n} (cf.\ \cite[Conjectures 4.1 and 4.3]{Sun2026dSum}). To save horizontal space in these tables, we employ the standard notation for multiple zeta values (MZVs)\begin{align}
  \zeta_{a_1,\dots,a_n}\colonequals {}&{\sum\limits_{\ell_{1}>\dots>\ell_{n}>0}\prod\limits_{j=1}^n\frac{{1}}{\ell_j^{a_j}}}\in\mathfrak Z_{a_1+\dots+a_n}(1)
\label{eq:MZV}\end{align} of weight $ a_1+\dots+a_n$ and depth $n$, as well as the following abbreviations for multiple Clausen values (MCVs)  \cite[\S2]{BorweinBroadhurstKamnitzer2001}{\allowdisplaybreaks
\begin{align}
\mu_{a_1,\dots,a_n}\colonequals {}&\sum_{\ell_{1}>\dots>\ell_{n+1}>0}\frac{e^{\ell_{1}\pi i/3}}{\ell_{1}^{a_{1}}}\prod\limits_{j=2}^{n}\frac{\smash[t]{1}}{\ell_j^{a_j}}\in \mathfrak Z_{a_{1}+\dots+a_n}(6)\label{eq:MCV_defn}
\end{align}}of weight $ a_1+\dots+a_n$ and depth $n$.   In \S\ref{subsec:descClausen}, we will sharpen \eqref{eq:MCV_defn} into $ \mu_{a_1,\dots,a_n}\in \mathfrak Z_{a_{1}+\dots+a_n}(3)$, so as to align all  the entries of Tables   \ref{tab:A1n} and \ref{tab:Asn} with the statements in   Theorems \ref{thm:A1n} and \ref{thm:Asn}.\eor\end{remark} \begin{table}[h]\caption{Closed forms for selected $ \mathscr A_{1,n}$\label{tab:A1n}}

{\footnotesize\begin{align*}\begin{array}{c|l}\hline\hline  n+1&\mathscr A_{1,n}\\\hline 1&\frac{\pi}{3\sqrt{3}}\\[5pt]2&-\frac{2\I\mu_2}{\sqrt{3}}\\[5pt]3&\frac{2^{2} \pi ^3}{3^{3} \sqrt{3}}\\[5pt]4&-\frac{2\I(2\cdot3^{2}\mu_4-\pi^{2}\mu_2)}{\sqrt{3}}\\[5pt]5&-\frac{2^{5}\cdot3\I(\mu_{4,1}+\zeta_3\mu_2)}{\sqrt{3}}+\frac{2\cdot67\pi^{5}}{3^{5}\sqrt{3}}\\[5pt]6&-\frac{2\cdot3\I(2^{3}\cdot3^{2}\cdot5\mu_6-2^{2}\cdot5\pi^{2}\mu_4-\pi^{4}\mu_{2})}{\sqrt{3}}\\[5pt]7&-\frac{2^{6}\cdot3^{3}\cdot5\I(\mu_{6,1}+\zeta_3\mu_4+\zeta_5\mu_2)}{\sqrt{3}}+\frac{2^{5}\cdot3\cdot5\pi^{2}\I(\mu_{4,1}+\zeta_3\mu_2)}{\sqrt{3}}+\frac{2^{2}\cdot5\cdot13^{2}\pi ^7}{3^{6}\sqrt{3}}\\[5pt]8&-\frac{2\I[2^{4}\cdot3^{4}\cdot5\cdot7\mu_{8}-2^{3}\cdot3\cdot5\cdot7\cdot23 \pi ^2\mu_{6}-2\cdot3^{4}\cdot7 \pi ^4 \mu_{4}+(2^{5}\cdot3^{3}\cdot5\cdot7 \zeta_3^2{}-79 \pi ^2\cdot3)\mu_{2}] }{3 \sqrt{3}}-\frac{2^{7}\cdot3^{2}\cdot5\cdot7\I(\mu_{6,1,1}+\zeta_{3}\mu_{4,1})}{\sqrt{3}}\\&{}-\frac{2^{2}\cdot5\cdot7 \pi  \left(3\cdot5\cdot659 \zeta _{7}+2^{3}\cdot3^{3}\cdot7 \pi ^2\zeta_{5}+2^{2}\cdot17 \pi ^4 \zeta_{3}\right)}{3^{4} \sqrt{3}}\\[4pt]\hline\hline\end{array}\end{align*}}
\end{table}

\begin{table}[h]\caption{\label{tab:Asn} Closed forms for selected $ \mathscr A_{s,n}$ }
{\scriptsize\begin{align*}\begin{array}{c|llll}\hline\hline  w&\mathscr A_{3,w-3}&\mathscr A_{5,w-5}&\mathscr A_{6,w-6}&\mathscr A_{7,w-7}
\\\hline 3&-\frac{2^{2}\zeta_{3}}{3} +\frac{2 \pi \I\mu_2}{3}&-\!\!-&-\!\!-&-\!\!-\\[5pt]4&-2( \I\mu_2)^{2}&-\!\!-&-\!\!-&-\!\!-\\[5pt]5&-\frac{2^{2}\cdot11 \zeta_5}{3^{2}}+\frac{2^{2}\pi ^2 \zeta _{3}}{3^{2}}+\frac{2^{3} \pi ^3 \I\mu_2}{3^{3}}&-\frac{19\zeta _{5}}{3}+2\pi  \I\mu _{4}+\frac{\pi ^{2} \zeta _{3}}{3^{2}}&-\!\!-&-\!\!- \\[5pt]6&-2^{3}\cdot3^{2}(\RE\mu _{4,2}+\I\mu_4 \I\mu_2)&2^{2}\cdot3\RE\mu _{4,2}&-\frac{2^{2} \pi  \I\mu _{4,1}}{3}&-\!\!-\\&{}+2^{4} \zeta _3^2+2 \pi ^2 ( \I\mu_2)^{2}-\frac{6691 \pi ^6}{2^{3}\cdot3^{6}\cdot5\cdot7}&{}-\frac{2^{4} \zeta _{3}^2}{3}+\frac{6691\pi ^{6}}{2^{4}\cdot3^{7}\cdot5\cdot7}&{}-\frac{2^{2} \zeta _{3}^2}{3}+\frac{13\cdot257\pi ^{6}}{2^{4}\cdot3^{7}\cdot5\cdot7}\\[5pt]7& -\frac{6037\zeta _{7}}{3^{3}}-2^{6}\cdot3\RE\mu _{4,1,2}&-\frac{1423\zeta _{7}}{2^{2}\cdot3^{2}}&-\frac{6037\zeta _{7}}{2^{3}\cdot3^{4}}-2^{3}\RE\mu _{4,1,2}&-\frac{17\cdot29\zeta _{7}}{2^{3}\cdot3}\\&{}+2^{5}\cdot5\pi  \I\mu _{6}+\frac{2^{2}\cdot127\pi ^2\zeta _{5}}{3^{3}}&{}+\frac{19\pi ^{2} \zeta _{5}}{3^{2}}&{}+\frac{2^{2}\cdot5\pi  \I\mu _{6}}{3}+\frac{61\pi ^{2} \zeta _{5}}{2\cdot3^{4}}&{}+2\cdot3\pi  \I\mu _{6}\\&{}-2^{6}\cdot3\I\mu _{2} \I\mu _{4,1} -\frac{2^{4} \pi ^3\I\mu _{4}}{3^{3}}&{}+\frac{2^{3} \pi ^{3} \I\mu _{4}}{3^{2}}&{}-\frac{2\pi ^{3} \I\mu _{4}}{3^{4}}&{}+\frac{\pi ^{2} \zeta _{5}}{3}\\&{}-2^{5}\cdot3\zeta _{3} (\I\mu _{2})^2-\frac{2\cdot491\pi ^4\zeta _{3}}{3^{4}\cdot5}&{}+\frac{31\pi ^{4} \zeta _{3}}{2\cdot3^{3}\cdot5} &{}-\frac{653\pi ^{4} \zeta _{3}}{2^{2}\cdot3^{5}\cdot5}&{}+\frac{17\pi ^{4} \zeta _{3}}{2^{2}\cdot3^{4}\cdot5}\\&{}+\frac{2^{2}\cdot67\pi ^5\I\mu _{2}}{3^{5}}\\[5pt]8&-2^{5}\cdot3^{3}\cdot5(\RE\mu _{6,2}+\I\mu _{6} \I\mu _{2})&-2^{2}\cdot3^{3}(\I\mu _{4})^2&-\frac{2^{4} \pi ^{3} \I\mu _{4,1}}{3^{3}}&2^{2}\cdot3^{2}\RE\mu _{6,2}\\&{}-\frac{2^{8}\cdot5\zeta _{5,3}}{3}+2^{5}\cdot5\cdot19\zeta_{5} \zeta_{3}&{}-2^{2}\cdot19\zeta _{5}\zeta _{3}&{}-\frac{2^{3}\cdot11\zeta _{5,3}}{3^{2}\cdot5}&{}+\frac{2^{5} \zeta _{5,3}}{3^{2}}\\&{}+2^{4}\cdot3\cdot5\pi ^2(\RE\mu _{4,2}+\I\mu _{4} \I\mu _{2})&{}-2^{2}\cdot3\pi ^{2} \RE\mu _{4,2}&{}-\frac{2^{3}\cdot11\zeta _{5} \zeta _{3}}{3^{2}}&{}-\frac{2^{3}\cdot43\zeta _{5}\zeta _{3} }{3^{2}}\\&{}-\frac{2^{5}\cdot5\pi ^2\zeta _{3}^2}{3}+2\cdot3\pi ^4(\I\mu _{2})^2&{}+\frac{2^{4} \pi ^{2} \zeta _{3}^2}{3}&{}+\frac{2^{2} \pi ^{2} \zeta _{3}^2}{3^{2}}&{}+\frac{679153\pi ^{8}}{2^{7}\cdot3^{8}\cdot5^{2}\cdot7}\\&{}-\frac{443\cdot1231\pi ^8}{2^{4}\cdot3^{7}\cdot5\cdot7}&{}-\frac{6691\pi ^{8}}{2^{4}\cdot3^{7}\cdot5\cdot7}&{}+\frac{3970097\pi ^{8}}{2^{6}\cdot3^{9}\cdot5^{3}\cdot7}\\[2pt]\hline\hline\end{array}\end{align*}}\end{table}\begin{table}[h]\caption{\label{tab:A2nA4n} Closed forms for selected $ \mathscr A_{2,n}$ and $ \mathscr A_{4,n}$ }{\footnotesize\begin{align*}\begin{array}{c|ll}\hline\hline  w&\mathscr A_{2,w-2}&\mathscr A_{4,w-4}\\\hline2&\frac{\pi^2}{2\cdot3^{2}}&-\!\!-\\[3pt]3&-\frac{2^{2}\zeta_{3}}{3}&-\!\!-\\[3pt]4&\frac{31 \pi ^4}{2^{2}\cdot3^{3}\cdot5}&\frac{17 \pi ^4}{2^{3}\cdot3^{4}\cdot5}\\[3pt]5&-2\cdot19\zeta _{5}+\frac{2^{2} \pi ^2\zeta_{3}}{3}&-\frac{2\cdot11\zeta _{5}}{3^{2}}\\[3pt]6&-2^{5}\zeta _{3}^2+\frac{11\cdot89\pi ^6}{2\cdot3^{4}\cdot5\cdot7}&\frac{439\pi^{6}}{2^{3}\cdot3^{6}\cdot5}\\[3pt]7&-5\cdot17\cdot29\zeta _{7}+\frac{2^{2}\cdot5\cdot19\pi ^2\zeta_{5}}{3}+2^{2}\pi ^4\zeta _{3}&-\frac{1423\zeta_{7}}{2^{2}\cdot3}+\frac{2\cdot11\pi ^{2} \zeta_{5}}{3^{2}}\\[3pt]8&2^{5}\cdot3\cdot7\zeta_{5,3}-2^{7}\cdot3^{2}\cdot5\zeta _{5}\zeta _{3}+2^{5}\cdot5\pi^2\zeta _{3}^{2}+\frac{11\cdot13\cdot829\pi ^{8}}{2^{3}\cdot3^{4}\cdot5^{2}\cdot7}&\frac{2^{5}\cdot11\zeta_{5,3}}{3\cdot5}+\frac{23\cdot32719\pi^{8}}{2^{4}\cdot3^{6}\cdot5^{3}\cdot7}\\ [3pt]9&-\frac{2\cdot5\cdot7\cdot13921\zeta _{9}}{3}+5\cdot7\cdot17\cdot29\pi ^{2} \zeta _{7}+2\cdot3\cdot7\cdot19\pi ^{4} \zeta _{5}&{-\frac{2\cdot5\cdot7\cdot13\cdot311\zeta _{9}}{3^{3}}+\frac{5\cdot1423\pi ^{2} \zeta _{7}}{2\cdot3^{2}}+\frac{2\cdot11\pi ^{4} \zeta _{5}}{3}}\\&{}-2^{7}\cdot5\cdot7\zeta _{3}^3+\frac{2^{2}\cdot79\pi ^{6} \zeta _{3}}{3^{2}}\\[3pt]10&2^{4}\cdot3\cdot5^{2}\cdot41\zeta _{7,3}-2^{7}\cdot7^{2}\pi ^{2} \zeta _{5,3}-2^{10}\cdot3^{3}\cdot5\cdot7\zeta _{7} \zeta _{3}&{\frac{2\cdot5\cdot1423\zeta _{7,3}}{7}-\frac{2^{5}\cdot11\pi ^{2} \zeta _{5,3}}{3}}\\&{}-2^{4}\cdot3^{2}\cdot5\cdot859\zeta _{5}^2+2^{9}\cdot3\cdot5\cdot7\pi ^{2} \zeta _{5} \zeta _{3}&{}+\frac{2\cdot3\cdot5\cdot191\zeta _{5}^2}{7}\\&{}+2^{6}\cdot3\cdot7\pi ^{4} \zeta _{3}^2+\frac{17^{2}\cdot15583\pi ^{10}}{2\cdot3^{4}\cdot5^{2}\cdot7\cdot11}&{}+\frac{937\cdot10567\pi ^{10}}{2^{4}\cdot3^{4}\cdot5^{2}\cdot7^{2}\cdot11}\\[3pt]11&-3^{3}\cdot7\cdot13\cdot79\cdot653\zeta _{11}+2^{9}\cdot3^{3}\cdot7^{2}\zeta _{5,3,3}&-\frac{7\cdot1567\cdot1693\zeta _{11}}{2^{3}}+2^{7}\cdot7\cdot11\zeta _{5,3,3}\\&{}+2^{3}\cdot5^{2}\cdot7\cdot11^{2}\cdot53\pi ^{2} \zeta _{9}+\frac{2\cdot3^{2}\cdot7\cdot8291\pi ^{4} \zeta _{7}}{5}&{}+\frac{2\cdot5\cdot7\cdot56813\pi ^{2} \zeta _{9}}{3^{3}}+\frac{7\cdot37\cdot653\pi ^{4} \zeta _{7}}{2^{2}\cdot3\cdot5}\\&{}-2^{10}\cdot3^{4}\cdot5\cdot7\zeta _{5} \zeta _{3}^2+2^{3}\cdot1277\pi ^{6} \zeta _{5}&{}+\frac{2\cdot11\cdot47\pi ^{6} \zeta _{5}}{3^{3}}\\&{}+2^{9}\cdot3\cdot5\cdot7\pi ^{2} \zeta _{3}^3+\frac{2^{2}\cdot2339\pi ^{8} \zeta _{3}}{3\cdot5}\\[3pt]12&2^{6}\cdot3^{3}\cdot5^{2}\cdot7\cdot23\zeta _{9,3}+2^{9}\cdot3^{3}\cdot5^{2}\cdot7\zeta _{4,4,2,2}&2^{5}\cdot3\cdot5\cdot7\cdot79\zeta _{9,3}-2^{6}\cdot5\cdot7\cdot23\zeta _{4,4,2,2}\\&{}-2^{4}\cdot3^{4}\cdot5^{2}\cdot37\pi ^{2} \zeta _{7,3}-2^{8}\cdot3\cdot5^{2}\cdot7\cdot1987\zeta _{9} \zeta _{3}&{}-2^{3}\cdot5\cdot229\pi ^{2} \zeta _{7,3}-\frac{2^{5}\cdot5\cdot7\cdot23\cdot53\zeta _{9} \zeta _{3}}{3^{2}}\\&{}+2^{6}\cdot3\cdot7^{2}\cdot11\pi ^{4} \zeta _{5,3}-2^{7}\cdot3^{3}\cdot5^{2}\cdot7\cdot479\zeta_{7} \zeta_{5}&{}-\frac{2^{5}\cdot7\cdot17\cdot59\pi ^{4} \zeta _{5,3}}{3^{2}\cdot5}-2^{6}\cdot5\cdot7\cdot863\zeta_{7} \zeta_{5}\\&{}+2^{11}\cdot3^{2}\cdot5^{2}\cdot7\pi ^{2} \zeta _{7} \zeta _{3}+2^{4}\cdot3^{2}\cdot5^{2}\cdot29\cdot37\pi ^{2} \zeta _{5}^2&{}+\frac{2^{7}\cdot5\cdot7^{2}\cdot23\pi ^{2} \zeta _{7} \zeta _{3}}{3}+\frac{2^{3}\cdot5^{3}\cdot17\cdot19\pi ^{2} \zeta _{5}^2}{3}\\&{}+2^{9}\cdot3\cdot5\cdot7\cdot31\pi ^{4} \zeta _{5} \zeta _{3}-2^{9}\cdot3^{3}\cdot5^{2}\cdot7\zeta _{3}^4&{}-\frac{2^{5}\cdot5\cdot7^{2}\cdot23\pi ^{4} \zeta _{5} \zeta _{3}}{3^{2}}+\frac{2^{6}\cdot5\cdot7\cdot23\zeta _{3}^4}{3}\\&{}+\frac{2^{6}\cdot5\cdot7\cdot43\pi ^{6} \zeta _{3}^2}{3}+\frac{1002151729\pi ^{12}}{2^{2}\cdot3^{3}\cdot5^{2}\cdot7\cdot11\cdot13}&{}-\frac{2^{9}\cdot23\pi ^{6} \zeta _{3}^2}{3^{4}}+\frac{16476625657\pi ^{12}}{2^{3}\cdot3^{7}\cdot5^{3}\cdot7\cdot11\cdot13}\\\hline\hline\end{array}\end{align*}}\end{table}\begin{remark}Some entries of  Tables   \ref{tab:A1n}--\ref{tab:A2nA4n} were reported in the existing literature. While  $ \mathscr A_2=\frac{\zeta_2}{3}$ was found in Ap\'ery's work  \cite{Apery1978,Poorten1978}, we  point out that both $ \mathscr A_1=\frac{\pi}{3\sqrt{3}}$ and $ \mathscr A_2=\frac{\pi^2}{18}$ are special cases of the following classical hypergeometric summations for $ x\in[0,4)$:\begin{align}
\sum_{k=1}^\infty\frac{x^k}{k\binom{2k}{k}}={}&\frac{2 \sqrt{x} \arcsin\frac{\sqrt{x}}{2}}{\sqrt{4-x}},& \sum_{k=1}^\infty\frac{x^k}{k^{2}\binom{2k}{k}}={}&2\left(\arcsin\frac{\sqrt{x}}{2}\right)^{2}.
\end{align}
As noted by Zucker \cite{Zucker1985}, the formula $\mathscr A_4=\frac{17\zeta_4}{36}=\frac{17\pi^{4}}{3240} $ was derived by Cohen \cite[Corollaire 5.3]{Cohen1981} using a  method that did not depend on    \eqref{eq:AsLs}. The computation of $ \mathscr A_4$ via     \eqref{eq:AsLs} was highlighted in van der Poorten's foreword for Lewin's book \cite{Lewin1981}.      Sun \cite{Sun2015} had conjectured (equivalent forms of) $ \mathscr A_{2,1}=-\frac{4\zeta_{3}}{3}$ and  $ \mathscr A_{4,1}=-\frac{22\zeta_5}{9}$ before they  were proved by Ablinger \cite{Ablinger2017}. In the statement of  \cite[Theorem 1.3]{HouSun2026}, one may also find explicit MZV representations\  for $ \mathscr A_{2,0},\dots,\mathscr A_{2,8}$. \eor\end{remark}\begin{remark}Our algorithmic approach to Table  \ref{tab:Asn}  also enables us to make partial progress towards some open questions surrounding MCVs. For example, K. C. Au \cite[Conjecture 1.5(b)]{Au2022a} proposed an empirical sum rule\begin{align}
\mathsf A_{w,n}\colonequals \sum_{\substack{a_1,\dots,a_n\in\mathbb Z_{>0}\\a_1+\dots+a_n=w}}\RE\mu_{a_1,\dots,a_n}\overset?\in\mathfrak Z_w(1)\label{eq:Au1.5b}
\end{align} for each fixed pair of weight $ w\in\mathbb Z_{>0}$ and depth $ n\in\mathbb Z\cap[1,w]$. While Au was able to verify particular cases of his conjecture with provable algebraic relations for $ \mathfrak Z_w(6),w\in\{1,2,3,4,5\}$ \cite{Au2022a}, our analysis in  \S\S\ref{subsec:fibLyn}--\ref{subsec:algLyndonMCV} will push this up to weight  $8$. We summarize  the corresponding results in Table \ref{tab:Au1.5b}.     
\eor\end{remark}

\begin{table}[h]\caption{Provable cases of Au's sum rule for MCVs\label{tab:Au1.5b}}
\begin{minipage}{.3\textwidth}{\footnotesize \begin{align*}\begin{array}{cc|l}\hline\hline w&n& \mathsf A_{w,n}  \\\hline1&1&0\\\hline2&1&\frac{\pi ^{2}}{2^{2}\cdot3^{2}}\\&2&-\frac{\pi ^{2}}{2\cdot3^{2}}\\[2pt]\hline3&1&\frac{\zeta _{3}}{3}\\&2&-\frac{2\zeta _{3}}{3}\\&3&0\\\hline 4&1&\frac{7\cdot13\pi ^{4}}{2^{4}\cdot3^{5}\cdot5}\\&2&-\frac{59\pi ^{4}}{2^{4}\cdot3^{4}\cdot5}\\&3&-\frac{\pi ^{4}}{2^{4}\cdot3^{4}}\\&4&\frac{\pi ^{4}}{2^{3}\cdot3^{5}}\\[2pt]\hline5&1&\frac{5^{2} \zeta _{5}}{2\cdot3^{3}}\\&2&-\frac{43\zeta _{5}}{2^{2}\cdot3^{2}}+\frac{\pi ^{2} \zeta \
_{3}}{2^{2}\cdot3^{2}}\\&3&\frac{29\zeta _{5}}{2^{2}\cdot3^{2}}-\frac{\pi ^{2} \zeta _{3}}{2^{2}\
\cdot3}\\&4&-\frac{29\zeta _{5}}{2\cdot3^{3}}+\frac{\pi ^{2} \zeta 
_{3}}{2\cdot3^{2}}\\&5&0\\\hline6&1&\frac{11^{2}\cdot31\pi ^{6}}{2^{5}\cdot3^{8}\cdot5\cdot7}\\&2&\frac{\zeta _{3}^2}{2\cdot3}-\frac{2203\pi 
^{6}}{2^{3}\cdot3^{8}\cdot5\cdot7}\\&3&-\frac{\zeta _{3}^2}{2}+\frac{31\cdot127\pi \
^{6}}{2^{5}\cdot3^{8}\cdot5\cdot7}\\&4&\frac{\zeta _{3}^2}{3}-\frac{331\pi ^{6}}{2^{2}\cdot3^{8}\cdot5\cdot7}\\&5&\frac{\pi ^{6}}{2^{5}\cdot3^{8}}\\&6&-\frac{\pi ^{6}}{2^{4}\cdot3^{8}\cdot5}\\[2pt]\hline\hline\end{array}\end{align*}}\end{minipage}\begin{minipage}{.5\textwidth}{\footnotesize \begin{align*}\begin{array}{cc|l}\hline\hline w&n& \mathsf A_{w,n}  \\\hline 7&1&\frac{7^{2}\cdot13\zeta _{7}}{2^{4}\cdot3^{4}}\\&2&-\frac{4481\zeta _{7}}{2^{5}\cdot3^{4}}+\frac{\pi ^{2} \zeta 
_{5}}{2^{2}\cdot3^{2}}+\frac{7\cdot13\pi ^{4} \zeta 
_{3}}{2^{4}\cdot3^{5}\cdot5}\\&3&\frac{3229\zeta _{7}}{2^{4}\cdot3^{4}}-\frac{\pi ^{2} \zeta 
_{5}}{3^{2}}-\frac{67\pi ^{4} \zeta _{3}}{2^{2}\cdot3^{5}\cdot5}\\&4&-\frac{3251\zeta _{7}}{2^{4}\cdot3^{4}}+\frac{\pi ^{2} \zeta 
_{5}}{2\cdot3}+\frac{\pi ^{4} \zeta _{3}}{2^{3}\cdot3\cdot5}\\&5&\frac{5\cdot659\zeta _{7}}{2^{5}\cdot3^{4}}-\frac{5\pi ^{2} \zeta 
_{5}}{2^{2}\cdot3^{2}}+\frac{5\pi ^{4} \zeta _{3}}{2^{4}\cdot3^{5}}\\&6&-\frac{659\zeta _{7}}{2^{4}\cdot3^{4}}+\frac{\pi ^{2} \zeta 
_{5}}{2\cdot3^{2}}-\frac{\pi ^{4} \zeta _{3}}{2^{3}\cdot3^{5}}\\&7&0\\\hline8&1&\frac{127\cdot1093\pi ^{8}}{2^{8}\cdot3^{10}\cdot5^{2}\cdot7}\\&2&\frac{7\zeta _{5,3}}{2\cdot3^{3}\cdot5}+\frac{5^{2} \zeta _{5} \zeta 
_{3}}{2\cdot3^{3}}-\frac{13\cdot43\cdot3709\pi ^{8}}{2^{8}\cdot3^{10}\
\cdot5^{3}\cdot7}\\&3&-\frac{7\zeta _{5,3}}{2^{2}\cdot3^{3}}-\frac{179\zeta _{5} \zeta 
_{3}}{2^{2}\cdot3^{3}}+\frac{ \pi ^{2} \zeta 
_{3}^2}{2^{3}\cdot3^{2}}+\frac{37501\pi 
^{8}}{2^{6}\cdot3^{9}\cdot5^{2}\cdot7}\\&4&2\zeta _{5} \zeta _{3}-\frac{ \pi ^{2} \zeta 
_{3}^2}{2\cdot3^{2}}-\frac{353\cdot1217\pi 
^{8}}{2^{8}\cdot3^{10}\cdot5^{2}\cdot7}\\&5&\frac{7\zeta _{5,3}}{2^{2}\cdot3^{3}}-\frac{5\cdot29\zeta _{5} \zeta 
_{3}}{2^{2}\cdot3^{3}}+\frac{5\pi ^{2} \zeta 
_{3}^2}{2^{3}\cdot3^{2}}+\frac{48593\pi 
^{8}}{2^{6}\cdot3^{10}\cdot5^{2}\cdot7}\\&6&-\frac{7\zeta _{5,3}}{2\cdot3^{3}\cdot5}+\frac{29\zeta _{5} \zeta 
_{3}}{2\cdot3^{3}}-\frac{ \pi ^{2} \zeta 
_{3}^2}{2^{2}\cdot3^{2}}-\frac{17\cdot19\cdot401\pi 
^{8}}{2^{8}\cdot3^{9}\cdot5^{3}\cdot7}\\&7&-\frac{\pi ^{8}}{2^{8}\cdot3^{10}\cdot5}\\&8&\frac{\pi ^{8}}{2^{7}\cdot3^{10}\cdot5\cdot7}\\[2pt]\hline\hline\\[68pt]\end{array}\end{align*}}\end{minipage}
\end{table}
 
  In \S\ref{sec:discussion}, we will discuss some integrals and series related to Theorems \ref{thm:A1n}--\ref{thm:A2nA4n}. Among other things, we will offer a new derivation of   \eqref{eq:AsLs}, verify a few special cases of Broadhurst's conjectures \cite[\S3]{Broadhurst2014MDV}, and  produce closed-form evaluations of  \begin{align}
\sum_{k=0}^\infty\left.\!\frac{\partial^n}{\partial x^n}\frac{10x-1}{\binom{4x}{2x}}\right|_{x=k}\label{eq:4k2k}
\end{align}for $ n\in\{0,1,2,3,4\}$, extending the $ n=0$ case treated by Sun  \cite[(1.28)]{Sun2024}.  \section{CMZV structures of deformed Ap\'ery-like series\label{sec:CMZVstruct}}
In \S\ref{subsec:GPLrecap}, we recapitulate some fundamental properties of generalized polylogarithms (GPLs), to pave the way for a constructive approach to (slightly weaker forms of)   Theorems \ref{thm:A1n}--\ref{thm:Asn} in \S\ref{subsec:Z(6)}.  Here, for ``weaker forms'', we mean to relax the right-hand sides of \eqref{eq:A1nCMZV3} and \eqref{eq:AsnCMZV3} to larger $\mathbb Q $-vector spaces $ i\sqrt{3}\mathfrak Z_{n+1}(6)$ and $\mathfrak Z_{n+s}(6)$, respectively. \subsection{Recursions and shuffles of GPLs\label{subsec:GPLrecap}}The GPLs are   defined through the recursion \cite[(2.1)]{Frellesvig2016} \begin{align}\notag\\[-12pt]
G(\alpha_{1},\dots,\alpha_n;z)\colonequals{\displaystyle\int_0^z\frac{G(\alpha_2,\dots,\alpha_n;x)\D x}{x-\alpha_1}},\quad\text{for }\displaystyle\smash[t]{\sum_{k=1}^n|}\alpha_k|\neq0,\label{eq:GPL_rec}
\end{align}equipped with the initial condition \cite[(2.2)]{Frellesvig2016}\begin{align}
{G(\boldsymbol0_n;z)\equiv G(\underset{n }{\underbrace{0,\dots,0 }};z)}\colonequals{\dfrac{\log^nz}{n!}},\quad G(-\!\!-;z)\colonequals1.\label{eq:GPL0}
\end{align}From the recursive definition in  \eqref{eq:GPL_rec}, one may verify the following scaling property \cite[(2.3)]{Frellesvig2016}
\begin{align}
G(\alpha_{1},\dots,\alpha_n;z)=G(\sigma\alpha_{1},\dots,\sigma\alpha_n;\sigma z)\label{eq:GPLscaling}
\end{align}where $ \alpha_n\neq0$ and $\sigma\neq0 $.

 The MPLs, being the analytic continuations of the convergent series\begin{align}
\Li_{a_1,\dots,a_n}(z_1,\dots,z_n)\colonequals \sum_{\ell_{1}>\dots>\ell_{n}>0}\prod_{j=1}^n\frac{z_{j}^{\ell_{j}}}{\ell_j^{a_j}},
\label{eq:Mpl_defn}\end{align}
are related to GPLs by (see \cite[(2.5)]{Frellesvig2016} or \cite[(1.3)]{Panzer2015}){\begin{align} \begin{split}\Li_{a_1,\dots,a_n}(z_1,\dots,z_n)={}&(-1)^{n}G\left(\smash[b]{\underset{a_1-1 }{\underbrace{0,\dots,0 }}},\frac{1}{z_{1}},\smash[b]{\underset{a_2-1 }{\underbrace{0,\dots,0 }}},\frac{1}{z_{1}z_2},\dots,\smash[b]{\underset{a_n-1 }{\underbrace{0,\dots,0 }}},\frac{1}{\prod_{j=1}^nz_j};1\right)\\[12pt]\equiv{}&(-1)^nG\left(  \boldsymbol0_{a_{1}-1},\frac{1}{z_{1}},\boldsymbol0_{a_{2}-1},\frac{1}{z_{1}z_2},\dots,\boldsymbol0_{a_{n}-1},\frac{1}{\prod_{j=1}^nz_j};1\right),
\end{split}\label{eq:MPL_GPL}\\[-6pt]\notag\end{align}}so long as  $ \prod_{j=1}^nz_j\neq0$.  In particular, we have\begin{align}
G(\alpha;t)=-\Li_1\left( \frac{t}{\alpha} \right)=\log\left( 1-\frac{t}{\alpha} \right)\label{eq:G(a;t)}
\end{align}for $ \alpha\neq 0$. By convention, an MPL $ \Li_{a_1,\dots,a_n}(z_1,\dots,z_n)$ is said to have weight $ a_1+\cdots+a_n$ and depth $n$, while a GPL $G(\alpha_1,\dots,\alpha_w;z)$ is said to have weight $w$.
In view of the conversions between GPLs and MPLs, one can check that (cf.\ \cite[(2.24)]{Zhou2022mkMpl}) \begin{align}
\mathfrak Z_{w}(N)\colonequals\Span_{\mathbb Q}\left\{G(\alpha_1,\dots,\alpha_w;1)\middle|\begin{smallmatrix}\alpha_{1}^{N},\dots,\alpha_{w}^N\in\{0,1\}\\\alpha_1\neq1,\alpha_{w}\neq0\end{smallmatrix} \right\}\label{eq:Zk(N)'}\tag{\ref{eq:Zk(N)}$'$}
\end{align} is equivalent to \eqref{eq:Zk(N)}. 

The GPLs satisfy  \cite[(2.4)]{Frellesvig2016} \begin{align}\label{eq:GPL_shuffle}G(\alpha_{1},\dots,\alpha_{j};z)G(\beta_{1},\dots,\beta_{k};z)=\sum_{\bm\gamma\in\bm \alpha\Shf\bm \beta }G(\gamma_{1},\dots,\gamma_{j+k}; z)\end{align}where the set $ \bm \alpha\Shf\bm \beta$ exhausts the shuffles of the vector components in  $ \bm \alpha$ and $\bm \beta$   that preserve the internal orders within $ \bm \alpha$ and $\bm \beta$. In particular, the following formula (cf.\ \cite[(3.4)]{Zhou2023SunCMZV})\begin{align}\begin{split}
G(\alpha;t)G(\beta_1,\dots,\beta_k;t)={}&G(\alpha,\beta_1,\dots,\beta_k;t)+\sum_{j=1}^{k-1}G(\beta_{1},\dots,\beta_j,\alpha,\beta_{j+1},\dots ,\beta_k;t)\\{}&+G(\beta_1,\dots,\beta_k,\alpha;t)\label{eq:ShGPL}
\end{split} \end{align}can be abbreviated into a statement about formal words:\begin{align}
\begin{split}\alpha\shf\beta_1\cdots\beta_k={}&\alpha\beta_1\cdots\beta_k+\sum_{j=1}^{k-1}\beta_1\cdots\beta_j\alpha\beta_{j+1}\cdots\beta_k\\&{}+\beta_1\cdots\beta_k\alpha.
\end{split}\tag{\ref{eq:ShGPL}$'$}
\end{align}  Applying the GPL shuffles to \eqref{eq:Zk(N)'}, one can prove Goncharov's filtration $ \mathfrak Z_j(N)\mathfrak Z_k(N)\subseteq\mathfrak Z_{j+k}(N)$ \cite[\S1.2]{Goncharov1998}, namely, $ z_j\in\mathfrak Z_j(N)$ and $ z_k\in\mathfrak Z_k(N)$ imply that $z_jz_k\in\mathfrak Z_{j+k}(N) $. 
By unshuffling (cf.\ \cite[(4.7)]{DuhrDulat2019}), one can always convert a GPL $ G(\alpha_{1},\dots,\alpha_{n};z)$ (which potentially contain trailing zeros in the parameters $ \alpha_{1},\dots,\alpha_{n}$) to a polynomial of $ \log z$ and GPLs whose parameters do not contain trailing zeros. Combining the GPL (un)shuffles with  the well-known fact that  $ \pi i\in \mathfrak Z_1(N)$ for  $ N\in\mathbb Z_{\geq3}$ \cite[Remark 1.3]{SingerZhao2020}, we have \begin{align}
\mathfrak Z_{w}(N)\colonequals\Span_{\mathbb Q}\left\{G(\alpha_1,\dots,\alpha_w;z)\middle|\begin{smallmatrix}\alpha_{1}^{N},\dots,\alpha_{w}^N,z^{N}\in\{0,1\}\\\alpha_1\neq z\end{smallmatrix} \right\}\tag{\ref{eq:Zk(N)}$''$}\label{eq:Zk(N)''}
\end{align}for all  $ N\in\mathbb Z_{\geq3}$.  \subsection{CMZV characterizations\label{subsec:Z(6)}}Starting from the beta integral\begin{align}
\frac12\int_0^1[t(1-t)]^{x-1}\D t=\frac{1}{x\binom{2x}{x}}\equiv\frac{[\Gamma(x)]^{2}}{2\Gamma(2x)}
\end{align}for $ x>0$, we may establish the following  representations:
\begin{align}
\mathscr A_{1,n}={}&\frac{1}{2}\int_0^1\frac{\log^n[t(1-t)]}{1-t(1-t)}\D t,\label{eq:A1n_int}\\\mathscr A_{s,n}={}&\frac{1}{2}\int_0^1\sum_{k=1}^\infty\left.\!\frac{\partial^{n}}{\partial x^{n}}\frac{[t(1-t)]^{x-1}}{x^{s-1}}\right|_{x=k}\D t,\label{eq:Asn_int}
\end{align}where $ s\in\mathbb Z_{>1}$ and $ n\in\mathbb Z_{\geq0}$.

The recursions and shuffles of GPLs  lead us to the next two propositions, which in turn, refine (the weaker forms of) Theorems \ref{thm:A1n}--\ref{thm:Asn} with quantitative details.\begin{proposition}For all $ n\in\mathbb Z_{>0}$, we have \begin{align}
&\log^n(t(1-t))=[\log t+\log(1-t)]^{n}=n!\sum_{\alpha_i\in\{0,1\}}G(\alpha_1,\dots,\alpha_n;t)
\label{eq:log_pow}\end{align}for $ t\in(0,1)$, and\begin{align}\begin{split}
\mathscr A_{1,n}={}&\frac{n!}{\sqrt{3}}\I\sum_{\alpha_i\in\{0,1\}}G(\varrho,\alpha_1,\dots,\alpha_n;1)\in  i\sqrt{3}\mathfrak Z_{n+1}(6),
\end{split}\label{eq:A1n_Z(6)}
\end{align}where $ \varrho\colonequals e^{\pi i/3}$.
The last equation applies retroactively to \begin{align}
\mathscr A_1=\frac{0!}{2\sqrt{3}i}\left[G(\varrho;1)-G\left(\frac1\varrho;1\right)\right]=\frac{\I G(\varrho;1)}{\sqrt{3}}=\frac{\pi}{3\sqrt{3}},\label{eq:A10}
\end{align}whereas the formal word $ \alpha_1\cdots\alpha_0$ becomes the empty string. \end{proposition}\begin{proof}We prove \eqref{eq:log_pow} by induction on $n$. For $n=1$, we have\begin{align}
\log t+\log(1-t)=G(0;t)+G(1;t),
\end{align}according to \eqref{eq:GPL0} and \eqref{eq:G(a;t)}. Suppose that \eqref{eq:log_pow} is true for a certain positive integer  $n$, then we have\begin{align}
[\log t+\log(1-t)]^{n+1}=n!\sum_{\alpha_i\in\{0,1\}}[G(0;t)+G(1;t)]G(\alpha_1,\dots,\alpha_n;t),
\end{align}where the products  $ G(0;t)G(\alpha_1,\dots,\alpha_n;t)$ and $ G(1;t)G(\alpha_1,\dots,\alpha_n;t)$ can be expanded through shuffles of GPLs, as in \eqref{eq:ShGPL}. Here, the new GPL parameter  (either $0$ or $1$)  can be inserted into $n+1$ different slots, so that  \begin{align}
\sum_{\alpha_i\in\{0,1\}}[G(0;t)+G(1;t)]G(\alpha_1,\dots,\alpha_n;t)
\end{align}  runs over\begin{align}
\sum_{\alpha_i\in\{0,1\}}G(\alpha_1,\dots,\alpha_{n+1};t)
\end{align}for a total of $n+1=\frac{(n+1)!}{n!}$ times. Thus, we see that \eqref{eq:log_pow} is also  true for $n+1$.
  
With the partial fraction expansion\begin{align}
\frac{1}{1-t(1-t)}=\frac{1}{\sqrt{3}i}\left( \frac{1}{t-\varrho}-\frac{1}{t-\frac{1}{\varrho}} \right)
\end{align}and the GPL recursion in \eqref{eq:GPL_rec}, we may turn \eqref{eq:A1n_int} and \eqref{eq:log_pow} into\begin{align}
\mathscr A_{1,n}={}&\frac{n!}{2\sqrt{3}i}\sum_{\alpha_i\in\{0,1\}}\left[G(\varrho,\alpha_1,\dots,\alpha_n;1)-G\left(\frac1\varrho,\alpha_1,\dots,\alpha_n;1\right)\right],\tag{\ref{eq:A1n_Z(6)}$'$}
\end{align} which is equivalent to the  ``$=$''  part of    \eqref{eq:A1n_Z(6)}. For the ``$\in$'' part of \eqref{eq:A1n_Z(6)}, invoke \eqref{eq:Zk(N)''}.

To verify \eqref{eq:A10}, simply evaluate an elementary integral\begin{align}
\mathscr A_1\equiv\mathscr A_{1,0}\xlongequal{\text{\eqref{eq:A1n_int}}}\frac{1}{2}\int_0^1\frac{\D t}{1-t(1-t)}=\frac{\pi}{3\sqrt{3}},
\end{align} and check that it is compatible with special values of weight-$1$ GPLs [see \eqref{eq:G(a;t)}].  \end{proof}

\begin{proposition}\label{prop:Asn}With $ \varrho\colonequals e^{\pi i/3} $, we have the following   sandwich identities \begin{align}
\begin{split}&\sum_{k=1}^\infty\left.\!\frac{\partial^{n}}{\partial x^{n}}\frac{[t(1-t)]^{x}}{x^{s-1}}\right|_{x=k}\\={}&-n!\sum_{\alpha_i,\beta_j\in\{0,1\}}\left[G(\alpha_{1},\dots,\alpha_{s-2},\varrho,\beta_1,\dots,\beta_n;t)+G\left(\alpha_{1},\dots,\alpha_{s-2},\frac{1}{\varrho},\beta_1,\dots,\beta_n;t\right)\right]
\end{split}\label{eq:sandwich}
\end{align}and \begin{align}
\mathscr A_{s,n}={}&-2n!\RE\sum_{\alpha_i,\beta_j\in\{0,1\}}G(0,\alpha_{1},\dots,\alpha_{s-2},\varrho,\beta_1,\dots,\beta_n;1)\in  \mathfrak Z_{n+s}(6)
\label{eq:Asn_Z(6)}
\end{align}for all $s\in\mathbb Z_{>1}, n\in\mathbb Z_{\geq0}$. \end{proposition}\begin{proof}First, we study the cases where $n=0$, so that the left-hand side of \eqref{eq:sandwich} becomes   the classical polylogarithm\begin{align}
\sum_{k=1}^\infty\frac{[t(1-t)]^{k}}{k^{s-1}}=\Li_{s-1}(t(1-t))
\end{align}for $s\in\mathbb Z_{>1} $,
 satisfying the  recursion\begin{align}
\frac{\D }{\D t}\Li_{s}(t(1-t))= \left( \frac{1}{t} +\frac{1}{t-1}\right)\Li_{s-1}(t(1-t)),
\end{align}with an initial condition [cf.\ \eqref{eq:G(a;t)}]\begin{align}
\Li_1(t(1-t))=-\log(1-t(1-t))=-G(\varrho;t)-G\left( \frac{1}{\varrho} ;t\right).
\end{align}Invoking the GPL recursion \eqref{eq:GPL_rec} for\begin{align}
\Li_{s}(t(1-t))=\int_0^t\left( \frac{1}{\tau} +\frac{1}{\tau-1}\right)\Li_{s-1}(\tau(1-\tau))\D\tau,
\end{align} we may verify  \eqref{eq:sandwich}$_{n=0}$\ by induction on $s$.
(This is a variation on the proof of \cite[Theorem 1.3(a)]{Zhou2023SunCMZV}.)

Next, we move on to the scenarios involving positive integers $n$.
Like what we encountered in the last paragraph, the function \begin{align}
a_{s,n}(t)\colonequals 
\sum_{k=1}^\infty\left.\!\frac{\partial^{n}}{\partial x^{n}}\frac{[t(1-t)]^{x}}{x^{s-1}}\right|_{x=k}
\end{align}also satisfies a recursion\begin{align}
a_{s,n}(t)=\int_0^t\left( \frac{1}{\tau} +\frac{1}{\tau-1}\right)a_{s-1,n}(\tau)\D\tau\label{eq:asn_rec}
\end{align}
for all $ s\in\mathbb Z_{>1}$, thanks to the following elementary integral formula: \begin{align} \int_0^t\left( \frac{1}{\tau} +\frac{1}{\tau-1}\right)[\tau(1-\tau)]^{x}\D \tau=\frac{[t(1-t)]^{x}}{x},\end{align} which holds for all $x\in(0,\infty)$. Therefore, for each fixed positive integer $n$, the validity of  \eqref{eq:sandwich}$ _{s\in\mathbb Z_{>1}}$ builds inductively on that of the $s=2$ case, by virtue of \eqref{eq:GPL_rec} and \eqref{eq:asn_rec}. As for $ s=2$, we may compute directly from \eqref{eq:log_pow} and  \eqref{eq:asn_rec}    that \begin{align}
\begin{split}a_{2,n}(t)={}&\int_0^t\left( \frac{1}{\tau} +\frac{1}{\tau-1}\right)\frac{\tau(1-\tau)[\log \tau+\log(1-\tau)]^{n}}{1-\tau(1-\tau)}\D\tau\\={}&-\int_0^t\left( \frac{1}{\tau-\varrho} +\frac{1}{\tau-\frac{1}{\varrho}}\right)[\log \tau+\log(1-\tau)]^{n}\D\tau\\={}&-n!\int_0^t\left( \frac{1}{\tau-\varrho} +\frac{1}{\tau-\frac{1}{\varrho}}\right)\sum_{\alpha_i\in\{0,1\}}G(\alpha_1,\dots,\alpha_n;\tau)\D\tau,
\end{split}\label{eq:a2n}
\end{align}so   \eqref{eq:sandwich}$_{s=2}$ is true by the GPL recursion in \eqref{eq:GPL_rec}.  

Finally, with the partial fraction expansion\begin{align}
\frac{1}{t(1-t)}=\frac{1}{t}+\frac{1}{1-t},
\end{align}the symmetry $ a_{s,n}(t)=a_{s,n}(1-t)$, and the GPL recursion in \eqref{eq:GPL_rec}, we may turn \eqref{eq:Asn_int} and \eqref{eq:sandwich} into \begin{align}
\begin{split}\mathscr A_{s,n}={}&\frac{1}{2}\int_0^1\frac{a_{s,n}(t)\D t}{t(1-t)}=\int_0^1\frac{a_{s,n}(t)\D t}{t}\\={}&-2n!\RE\int_0^1\sum_{\alpha_i,\beta_j\in\{0,1\}}G(\alpha_{1},\dots,\alpha_{s-2},\varrho,\beta_1,\dots,\beta_n;t)\frac{\D t}{t}\\={}&-2n!\RE\sum_{\alpha_i,\beta_j\in\{0,1\}}G(0,\alpha_{1},\dots,\alpha_{s-2},\varrho,\beta_1,\dots,\beta_n;1),\label{eq:Asn_int_sum}
\end{split}
\end{align}as claimed in the   ``$=$''  part of   \eqref{eq:Asn_Z(6)}. Again,  for the ``$\in$'' part of \eqref{eq:Asn_Z(6)}, invoke \eqref{eq:Zk(N)''}.   
\end{proof}\begin{remark}We may rewrite \eqref{eq:a2n} as follows:\begin{align}
\begin{split}a_{2,n}(t)={}&\int_0^t\frac{(1-2\tau)[\log \tau+\log(1-\tau)]^{n}}{1-\tau(1-\tau)}\D\tau=\int_0^{t(1-t)}\frac{\log^ny}{1-y}\D y\\={}&n!\int_0^{t(1-t)}{G(\boldsymbol0_{n};y)}\frac{\D y}{1-y}=-n!G(1,\boldsymbol0_{n};t(1-t)),
\end{split}\tag{\ref{eq:a2n}$'$}
\end{align}which reveals a close connection to the Bloch--Wigner--Ramakrishnan polylogarithm function (cf.\  \cite[A.2.7(8)]{Lewin1981} and \cite[\S1]{Zagier1990}). As a consequence, the recursion \eqref{eq:asn_rec} brings us \begin{align}
  a_{s,n}(t)=-n!G(\boldsymbol0_{s-2},1,\boldsymbol0_{n};t(1-t)).
\end{align}This formulation will become useful  during our proof of Theorem \ref{thm:A2nA4n} in \S\ref{subsubsec:desc1}.
\eor\end{remark}

\begin{remark}We note that our   \eqref{eq:A1n_Z(6)} and      \eqref{eq:Asn_Z(6)} are similar to the explicit formulae of Wang--Xu \cite{WangXu2021}, which expressed certain {E}uler--{A}p\'{e}ry-type series in terms of alternating multiple zeta values (AMZVs), namely, CMZVs of level $2$ that are extensively tabulated \cite{MZVdatamine2010}. Like Wang--Xu \cite{WangXu2021}, we may evaluate $ \mathscr A_{1,n}$ via \eqref{eq:A1n_Z(6)}    [resp.\ $ \mathscr A_{s,n}$  via \eqref{eq:Asn_Z(6)}] for $ 1\leq n+1\leq 5$ (resp.\ $ 2\leq n+s\leq 5$)
by consulting  Au's  \texttt{MultipleZetaValues} package \cite{Au2025a,Au2022a}, which represents $ G(\alpha_1,\dots,\alpha_n;1)$ by \texttt{IterInt[\{$ \alpha_1,\dots,\alpha_n$\}]}, and supports reductions of CMZVs through the function \texttt{MZExpand}.
For want of  provable  look-up tables for $ \mathfrak Z_w(6)$ with  $w>5$, we need some workarounds to be developed in \S\ref{sec:evalClGl}.
\eor\end{remark}\begin{remark}In both \eqref{eq:A1n_Z(6)} and      \eqref{eq:Asn_Z(6)}, the GPLs assume the forms of \begin{align}
G(\alpha_1,\dots,\alpha_n;1),\quad\text{where }\alpha_1,\dots,\alpha_n\in\{0,1,\varrho\},\alpha_1\neq1,\alpha_n\neq0,
\end{align} which qualify them as the complex conjugates of multiple Deligne values (MDVs)  \cite[\S2]{Broadhurst2014MDV}. According to the fibration procedures (to be described in \S\ref{subsec:fibLyn}), the complex conjugate of every MDV (which allows multiple occurrences of $ \varrho$ among the GPL parameters $ \alpha_1,\dots,\alpha_n$) can be rewritten as  a polynomial of MCVs. Therefore, we will not use the terminology MDV later in this article.  
\eor\end{remark}
\section{MCV and MZV evaluations of deformed Ap\'ery-like series\label{sec:evalClGl}}
In \S\ref{subsec:fibLyn}, we describe partial reductions of \eqref{eq:A1n_Z(6)} and
     \eqref{eq:Asn_Z(6)} via automated manipulations of GPLs. In \S\ref{subsec:algLyndonMCV}, we establish algebraic relations among MCVs, so as to fully reduce \eqref{eq:A1n_Z(6)} and
     \eqref{eq:Asn_Z(6)} for the entries of Tables \ref{tab:A1n} and \ref{tab:Asn}. In \S\ref{subsec:descClausen}, we verify Theorems \ref{thm:A1n}--\ref{thm:A2nA4n} in their entirety, and explain the algorithmic construction of  Table \ref{tab:A2nA4n}. \subsection{GPL fibrations and Lyndon word decompositions\label{subsec:fibLyn}}We have  the following differential relation \cite[(8.8)]{Weinzierl2022Book} \begin{align}
\begin{split}{}&\D G(\alpha_1,\dots,\alpha_n;z)\\={}&\sum_{j=1}^n G(\alpha_1,\dots,\widehat{\alpha_j},\dots,\alpha_n;z)[\D\log(\alpha_{j-1}-\alpha_j)-\D\log(\alpha_{j+1}-\alpha_j)]\label{eq:GPL_diff_form}\end{split}
\end{align}for $\alpha_0\colonequals z $, $ \alpha_{n+1}\colonequals0$,  
dropping each element under the caret, while treating ``$ \D\log 0$'' as $0$ (a convention that is compatible with logarithmic regularization  \cite[\S2.3]{Panzer2015}). Through repeated invocations of this differential relation and GPL recursions, one can establish the following fibration of GPLs:{\allowdisplaybreaks
\begin{align}
\sum_{\alpha_i\in\{0,1\}}G(\varrho,\alpha_1,\dots,\alpha_n;1)\in{}&\mathfrak F_{n+1},
\label{eq:A1n_fib}\\\sum_{\alpha_i,\beta_j\in\{0,1\}} G(0,\alpha_{1},\dots,\alpha_{s-2},\varrho,\beta_1,\dots,\beta_n;1)\in{}&\mathfrak F_{n+s},\label{eq:Asn_fib}
\end{align}
}where \begin{align}
\mathfrak F_m\colonequals \Span_{\mathbb Q}\left\{ Z_{k}(\pi i)^{\ell}G(\alpha_{1},\ldots,\alpha_{m-k-\ell};\varrho)\middle | \begin{smallmatrix}k,\ell,m-k-\ell\in\mathbb Z_{\geq0}\\Z_k\in\mathfrak Z_k(1)\\\alpha_1,\ldots,\alpha_{m-k-\ell}\in\{0,1\}\end{smallmatrix}\right\},
\end{align}and $ \mathfrak Z_0(1)\colonequals \mathbb Q$. In both \eqref{eq:A1n_fib} and \eqref{eq:Asn_fib}, explicit fibrations are computable from the function \texttt{fibrationBasis} in Panzer's \texttt{HyperInt}  package \cite{Panzer2015}.  After these fibrations, we may reduce the computations of  $ \mathscr A_{1,n}$ via \eqref{eq:A1n_Z(6)}    [resp.\ $ \mathscr A_{s,n}$  via \eqref{eq:Asn_Z(6)}]
to the imaginary (resp.\ real) part  of the GPLs appearing on the right-hand side of \eqref{eq:A1n_fib} [resp.\ \eqref{eq:Asn_fib}].
 
The shuffling procedure in \eqref{eq:GPL_shuffle} expresses the product of two low-weight GPLs  as a finite sum of high-weight GPLs. It is sometimes desirable to operate in the reverse direction, namely, to decompose a high-weight GPL into a polynomial of GPLs that come in  equal or smaller weights, such as\footnote{One can check the relation \eqref{eq:Lwd} in two ways:  (1) in Panzer's \texttt{HyperInt}  package \cite{Panzer2015} for \texttt{Maple}, type\begin{quote} \texttt{fibrationBasis(Hlog(z, [A, B])*Hlog(z, [A, A, B, B]) - 9*Hlog(z, [A, A, A, B, B, B]) - 4*Hlog(z, [A, A, B, A, B, B]) - Hlog(z, [A, A, B, B, A, B]), [z])}\end{quote} and receive \texttt{Hlog(z, [A, B, A, A, B, B])}; (2) in Ma\^itre's \texttt{HPL} package \cite{Maitre2005,Maitre2012} for \texttt{Mathematica}, type \begin{quote}\texttt{HPLReduceToMinimalSet[HPL[\{0, -1, 0, 0, -1, -1\}, z]] //.HPL[a\_\_] :> G[a] //.G[a\_, z] :> G[HPLMtoA[a], z]}\end{quote} and receive \texttt{G[\{0, -1\}, z] G[\{0, 0, -1, -1\}, z] - G[\{0, 0, -1, -1, 0, -1\}, z] - 4 G[\{0, 0, -1, 0, -1, -1\}, z] - 9 G[\{0, 0, 0, -1, -1, -1\},z]}. \label{fn:Lyn1} } \begin{align}
\begin{split}&G(A,B,A,A,B,B;z)\\={}&G(A,B;z)G(A,A,B,B;z)-9G(A,A,A,B,B,B;z)-4G(A,A,B,A,B,B;z)\\{}&-G(A,A,B,B,A,B;z),
\end{split}\label{eq:Lwd}
\end{align}  or in terms of formal words\begin{align}
ABAABB=AB\shf AABB-9AAABBB-4AABABB-AABBAB.
\end{align}The relation above decomposes the GPL associated with a  non-Lyndon word $ ABAABB$ into a  polynomial of several GPLs associated with Lyndon words $ AB$, $ AABB$, $ AAABBB$, $ AABABB$,  and $ AABBAB$. Here, a binary Lyndon word of length $w$ is an aperiodic sequence of $w$ letters sorted from the alphabet $ \{A,B\}$ that\ precedes (lexicographically) every cyclic shift of itself. In principle, each GPL associated with a binary word admits a Lyndon word decomposition, as guaranteed by Radford's theorem \cite{Radford1979} for the ring basis of shuffle algebra. In practice,  if a binary word is spelt with $8$ letters or fewer, then its Lyndon word decomposition can be generated by  Ma\^itre's  \texttt{HPL} package \cite{Maitre2005,Maitre2012}.\footnote{Thanks to the built-in substitution rules for GPLs up to weight $8$, the function \texttt{HPLReduceToMinimalSet} (following the definition of ``minimal set'' by Remiddi--Vermaseren \cite{RemiddiVermaseren2000}) executes binary  Lyndon word decomposition, so long as one maps the letter  $A$ to $0$ and the letter $B$ to $-1$. Ma\^itre's function \texttt{HPL[\{$ \alpha_1,\dots,\alpha_n$\}, z]} is the same as  $G(\alpha_1,\dots,\alpha_n;z)$ if $ \alpha_1,\dots,\alpha_n\in\{0,-1\}$. Note however that these statements are no longer true if one trades $-1$ for $1$. \label{fn:Lyn2}} 

When $\min\{a,b\}>0$, a monotone word 
\begin{align} \underset{a}{\underbrace{A\cdots A}}\underset{b}{\underbrace{B\cdots B}}\end{align} qualifies (trivially) as a Lyndon word, and its corresponding multiple Clausen value (MCV)\begin{align}
(-1)^{b}G(\boldsymbol0_{a},\boldsymbol1_{b};e^{\pi i/3})\equiv(-1)^{b}G(\underset{a}{\underbrace{0,\dots ,0}},\underset{b}{\underbrace{1,\dots, 1}};e^{\pi i/3})=\begin{cases} \mu_{a+1}\equiv\mu_{a+1,\mathbf 1_0}& \text{if }b=1 \\
\mu_{a+1,\smash[b]{\underset{b-1}{\smallunderbrace{\scriptstyle 1,\dots,1}}}}\equiv\mu_{a+1,\boldsymbol1_{b-1}} & \text{if }b>1 \\
\end{cases}
\end{align}is either a classical polylogarithm (if $b=1$) or a polylogarithm of Nielsen's type (if $ b>1$).
 Some MCVs in the form of $ \mu_{a,\boldsymbol1_n}$  can be further simplified by the function \texttt{LiReduce} in the  \texttt{logsine} package of Borwein--Straub \cite{BorweinStraub2011ISSAC}.\footnote{In  \cite[Theorems 3.2, 3.8, 4.4, and 4.6]{BorweinStraub2015Snp}, Borwein--Straub presented provable sum rules that simplified certain MCVs of Nielsen's type. Meanwhile, in   \cite[Example 4.8]{BorweinStraub2015Snp},  Borwein--Straub listed the empirical reductions of $ \I \mu_{4,1,1}$ and $ \RE\mu_{5,1,1}$ without proving them. In \S\ref{subsec:algLyndonMCV} of this work, we will not require any input from the   \texttt{logsine} package as the prior knowledge for our constructions of algebraic relations among MCVs.} The MCV corresponding to a non-monotone binary Lyndon word\begin{align}
(-1)^{n}G\big(\boldsymbol0_{a_{1}-1},1,\boldsymbol0_{a_{2}-1},1,\dots,\boldsymbol0_{a_{n}-1},1;e^{\pi i/3}\big)
\end{align}takes the form of $ \mu_{a_1,\dots,a_n}$ (where $ a_j>1$ for a certain $ j\in\{2,\dots,n\}$), whose real and imaginary parts may or may not be  reducible to classical/Nielsen polylogarithms  (see Propositions \ref{prop:Lyn56}--\ref{prop:Lyn78} below).

Thus far,  fibrations and Lyndon word decompositions together offer (partial) simplifications for  the right-hand sides of \eqref{eq:A1n_Z(6)} and \eqref{eq:Asn_Z(6)}, narrowing them down to polynomials of certain MCVs. In the next subsection, we will work on Lyndon MCVs $ \mu_{a_1,\dots,a_n}$ (where $ \boldsymbol0_{a_{1}-1}1\boldsymbol0_{a_{2}-1}1\cdots\boldsymbol0_{a_{n}-1}1$ constitute binary Lyndon words) and their automated reductions.

\subsection{Provable reductions for  certain Lyndon MCVs\label{subsec:algLyndonMCV}}We say that a Lyndon MCV $ \mu_{a_1,\dots,a_n}$ (of weight $ a_1+\dots+a_n$ and depth $n$) is ``reducible'' if it can be written as a polynomial (with rational coefficients) of $ \pi i$, Lyndon MCVs of lower weights/depths, and CMZVs of levels $ \leq2$.\footnote{Being ``reducible'' in our context is almost synonymous to being ``non-primitive'' in Broadhurst's terminology \cite[\S3]{Broadhurst2014MDV}, except that Broadhurst did not count the reducibility to  CMZVs of levels $ \leq2$ as a criterion for non-primitivity. Therefore, the MCVs $ \mu_3$, $ \mu_5$, $\mu_7$, and $ \mu_{7,1}$ (cf.\ Tables \ref{tab:MCV2345}, \ref{tab:MCV7}, and \ref{tab:MCV8}) are ``reducible''  but still ``primitive''. In Table \ref{tab:MCV2345}, we also consider the representation of $ \mu_{3,2}$ in terms of $ \mu_{4,1}$ as a ``reduction'', even though both MCVs have the same weight/depth. We pick  $ \mu_{4,1}$ as an  ``irreducible'' MCV, because its corresponding Lyndon word $ 00011$ precedes (lexicographically) that of  $ \mu_{3,2}$, which is $ 00101$. In Tables \ref{tab:MCV2345}--\ref{tab:MCV8}, we always prioritize our choices of  ``irreducible'' MCVs by their lexicographic orders. Presumably, the entries of Tables \ref{tab:MCV2345}--\ref{tab:MCV8} exhaust all the possible algebraic relations among MCVs and MZVs up to weight $8$, since our choices of ``irreducible'' MCVs saturate Broadhurst's conjectural bounds (cf.\ Table  \ref{tab:primMCVs}  in \S\ref{subsec:conjMDV}).\label{fn:irred_prim}} Naturally, this notion of ``reducibility'' can be extended to the real and imaginary parts of  $ \mu_{a_1,\dots,a_n}$, as well.

  \begin{proposition}\label{prop:Lyn56}\begin{enumerate}[leftmargin=*, 
label=\emph{(\alph*)},ref=(\alph*),
widest=d, align=left] \item Certain Lyndon MCVs of weights $2 $, $3$, $4$, and $5$ satisfy the algebraic relations listed in Table \ref{tab:MCV2345}. 
 
 \begin{table}[h]\caption{Provable reductions for  certain Lyndon MCVs of weights 2, 3, 4, and 5\label{tab:MCV2345}}{\scriptsize\begin{tabular}{r@{{\,}={\;}}l}\hline\hline
$\RE\mu _2$&$\frac{\pi ^2}{2^{2}\cdot3^{2}}$\\\hline$\mu _3$&$\frac{\zeta _{3}}{3}+\frac{5 \pi ^3 i }{2\cdot3^{4}}$\\
$\mu _{2,1}$&$\frac{2 \zeta _{3}}{3}+\frac{ \pi  i  \mu _{2}}{3}-\frac{\pi ^3 i }{2\cdot3^{4}}$\\\hline$\mu _{3,1}$&$\mu _{4}-\frac{2 \pi  i  \zeta _{3}}{3^{2}}-\frac{19 \pi ^4}{2^{3}\cdot3^{4}\cdot5}$\\
$\mu _{2,1,1}$&$\mu _{4}-\frac{ \pi  i  \zeta _{3}}{3^{2}}-\frac{\pi ^2 \mu _{2}}{2\cdot3^{2}}-\frac{11 \pi ^4}{2^{3}\cdot3^{4}\cdot5}$ \\ $\RE\mu _4$&$\frac{7\cdot13 \pi ^4}{2^{4}\cdot3^{5}\cdot5}$\\\hline $\mu _5$&$\frac{5^{2} \zeta _{5}}{2\cdot3^{3}}+\frac{17 \pi ^5 i }{2^{3}\cdot3^{6}}$\\
$\mu _{3,2}$&$-\frac{5 \zeta _{5}}{2\cdot3}-\frac{7 \mu _{4,1}}{3}-\frac{2 \pi  i  \mu _{4}}{3^{2}}+\frac{73 \pi ^5 i }{2^{2}\cdot3^{7}\cdot5}$\\
$\mu _{3,1,1}$&$-\frac{7 \zeta _{5}}{2\cdot3^{2}}+\mu _{4,1}+\frac{\pi ^2 \zeta _{3}}{3^{3}}-\frac{\pi ^5 i }{2^{2}\cdot3^{5}\cdot5}$\\
$\mu _{2,2,1}$&$\frac{7^{2} \zeta _{5}}{2\cdot3^{3}}-\frac{7 \mu _{4,1}}{3}+\frac{ \pi  i  \mu _{4}}{3^{2}}+\frac{2 \zeta _{3} \mu _{2}}{3}-\frac{2 \pi ^2 \zeta _{3}}{3^{3}}+\frac{ \pi  i  \mu _{2}^2}{2\cdot3}-\frac{ \pi ^3 i  \mu _{2}}{2\cdot3^{4}}-\frac{97 \pi ^5 i }{2^{3}\cdot3^{7}\cdot5}$\\
$\mu _{2,1,1,1}$&$\frac{29 \zeta _{5}}{2\cdot3^{3}}+\frac{ \pi  i  \mu _{4}}{3}+\frac{\pi ^2 \zeta _{3}}{2\cdot3^{3}}-\frac{ \pi ^3 i  \mu _{2}}{2\cdot3^{4}}-\frac{41 \pi ^5 i }{2^{3}\cdot3^{6}\cdot5}$\\$\RE\mu _{4,1}$&$\frac{43 \zeta _{5}}{2^{2}\cdot3^{2}}-\frac{ \pi  \I\mu _{4}}{3}-\frac{\pi ^2 \zeta _{3}}{2^{2}\cdot3^{2}}$\\\hline\hline
\end{tabular}}\end{table}

\item Certain Lyndon MCVs of weight  $6$ satisfy the algebraic relations listed in Table \ref{tab:MCV6}.\begin{table}[h]\caption{Provable reductions for  certain  Lyndon MCVs of weight 6\label{tab:MCV6}}{\scriptsize\begin{tabular}{r@{{\,}={\;}}l}\hline\hline$\mu _{5,1}$&$2 \mu _{6}-\frac{29 \pi  i  \zeta _{5}}{2\cdot3^{4}}-\frac{\zeta _{3}^2}{2\cdot3}-\frac{5 \pi ^3 i  \zeta _{3}}{2\cdot3^{4}}-\frac{11\cdot53 \pi ^6}{2^{4}\cdot3^{8}\cdot7}$\\
$\mu _{4,1,1}$&$3 \mu _{6}-\frac{2 \pi  i  \zeta _{5}}{3}+\frac{ \pi  i  \mu _{4,1}}{3}+\frac{\pi ^2 \mu _{4}}{2\cdot3^{2}}-\frac{\zeta _{3}^2}{2}-\frac{2 \pi ^3 i  \zeta _{3}}{3^{4}}-\frac{31\cdot59 \pi ^6}{2^{3}\cdot3^{8}\cdot5\cdot7}$\\
$\mu _{3,2,1}$&$-\frac{7\cdot13 \mu _{6}}{11}+\frac{2\cdot7 \pi  i  \zeta _{5}}{3^{2}}-\frac{7 \pi  i  \mu _{4,1}}{3^{2}}-\frac{5 \pi ^2 \mu _{4}}{2\cdot3^{3}}+\frac{23 \zeta _{3}^2}{2\cdot3^{2}}+\frac{19 \pi ^3 i  \zeta _{3}}{3^{5}}+\frac{137 \pi ^6}{2^{3}\cdot3^{8}}$\\
$\mu _{3,1,2}$&$\frac{2\cdot47 \mu _{6}}{11}+\mu _{4,2}-\frac{2\cdot11 \pi  i  \zeta _{5}}{3^{4}}-\frac{2 \pi  i  \mu _{4,1}}{3^{2}}-\frac{2 \pi ^2 \mu _{4}}{3^{3}}-\frac{7 \zeta _{3}^2}{2\cdot3^{2}}-\frac{79 \pi ^3 i  \zeta _{3}}{2\cdot3^{5}}-\frac{11\cdot191 \pi ^6}{2^{3}\cdot3^{7}\cdot5\cdot7}$\\
$\mu _{3,1,1,1}$&$2 \mu _{6}-\frac{43 \pi  i  \zeta _{5}}{2\cdot3^{3}}+\frac{ \pi  i  \mu _{4,1}}{3}+\frac{\pi ^2 \mu _{4}}{2\cdot3^{2}}-\frac{\zeta _{3}^2}{3}+\frac{5 \pi ^3 i  \zeta _{3}}{2\cdot3^{5}}-\frac{167 \pi ^6}{3^{8}\cdot5\cdot7}$\\
$\mu _{2,2,1,1}$&$-11 \mu _{6}-\mu _{4,2}+\frac{7\cdot11 \pi  i  \zeta _{5}}{3^{3}}-\frac{7 \pi  i  \mu _{4,1}}{3^{2}}+\mu _{2} \mu _{4}-\frac{2^{2} \pi ^2 \mu _{4}}{3^{3}}+\zeta _{3}^2-\frac{ \pi  i  \zeta _{3} \mu _{2}}{3^{2}}+\frac{2\cdot5 \pi ^3 i  \zeta _{3}}{3^{5}}-\frac{ \pi ^2 \mu _{2}^2}{2^{2}\cdot3^{2}}-\frac{11 \pi ^4 \mu _{2}}{2^{3}\cdot3^{4}\cdot5}+\frac{17851 \pi ^6}{2^{4}\cdot3^{8}\cdot5\cdot7}$\\
$\mu _{2,1,1,1,1}$&$\mu _{6}-\frac{5^{2} \pi  i  \zeta _{5}}{2\cdot3^{4}}-\frac{\pi ^2 \mu _{4}}{2\cdot3^{2}}+\frac{ \pi ^3 i  \zeta _{3}}{2\cdot3^{5}}+\frac{\pi ^4 \mu _{2}}{2^{3}\cdot3^{5}}-\frac{17\cdot19 \pi ^6}{2^{4}\cdot3^{7}\cdot5\cdot7}$\\$\RE\mu _6$&$\frac{11^{2}\cdot31 \pi ^6}{2^{5}\cdot3^{8}\cdot5\cdot7}$\\$\I\mu _{4,2}$&$-\frac{11 \I\mu _{6}}{2}+\frac{2\cdot29 \pi  \zeta _{5}}{3^{4}}+\frac{ \pi ^2 \I\mu _{4}}{2^{2}\cdot3^{2}}+\frac{5 \pi ^3 \zeta _{3}}{3^{4}}$\\\hline\hline\end{tabular}}\end{table}
\end{enumerate}\end{proposition}\begin{proof}\begin{enumerate}[leftmargin=*,  label=(\alph*),ref=(\alph*),
widest=d, align=left] \item
At weights $\leq5$, all the entries of Table \ref{tab:MCV2345} can be verified in  Au's   \texttt{MultipleZetaValues} package \cite{Au2025a,Au2022a}. Nevertheless, we will describe an alternative approach that generalizes well to higher weights.

Let $ \bm\alpha$ be a Lyndon word on the alphabet $ \{0,1\}$.  Consider the following functions:{\allowdisplaybreaks
\begin{align}
g_{\bm \alpha}(z)\colonequals {}&G\left(\bm \alpha;\frac{z}{z-1}\right)-G(\bm \alpha;1-z),\\g^{R}_{\bm\alpha}(z)\colonequals {}&\RE\left[G(\bm \alpha;z)-G\left(\bm \alpha;1-z\right)\right],\\g^{I}_{\bm\alpha}(z)\colonequals {}&\I\left[G(\bm \alpha;z)+G\left(\bm \alpha;1-z\right)\right].
\end{align}
}It is clear that $ g_{\bm \alpha}(\varrho)=g_{\bm \alpha}^{R}(\varrho)=g_{\bm \alpha}^{I}(\varrho)=0$ for $ \varrho=e^{\pi i/3}$.  We generate provable relations for MCVs from a collection of functions $\{g_{\bm \alpha}\}_{\bm\alpha\in \mathfrak L}$ (where each  $ \bm\alpha\in \mathfrak L$ has the same length) as follows:\begin{itemize}
\item 
Fibrate $g_{\bm \alpha}(z)$ with respect to the variable $z$  for each $ \bm\alpha\in \mathfrak L$,  so that $ g_{\bm \alpha}(z)$ is reexpressed through MZVs and GPLs $ G(\beta_1,\dots,\beta_n;z)$, where $ \beta_j\in\{0,1\}$ and $n$ does not exceed the length of $ \bm \alpha$.\footnote{This can be done by the function \texttt{HPLConvertToSimplerArgument} in  Ma\^itre's \texttt{HPL} package \cite{Maitre2005,Maitre2012}. However, note that Ma\^itre's \texttt{HPL([$  \alpha_1,\dots,\alpha_n$], z)} (where $ \alpha_1,\dots,\alpha_n\in\{0,1\}$) is equal to $ (-1)^{\alpha_1+\dots+\alpha_n}G( \alpha_1,\dots,\alpha_n;z)$. } Collect  these expressions  as $ \{\fib g_{\bm \alpha}(z)\}_{\bm \alpha\in\mathfrak L}$.\item Perform Lyndon word decomposition on the output from the last step,  resulting in a new collection    $ \{\Lyn\fib g_{\bm \alpha}(z)\}_{\bm \alpha\in\mathfrak L}$.\footnote{When the length of $ \bm \alpha$ does not exceed $8$, this can be done in  Ma\^itre's \texttt{HPL} package \cite{Maitre2005,Maitre2012}, so long as one converts back and forth between the alphabets $ \{0,1\}$ and $ \{0,-1\}$. (See Footnotes \ref{fn:Lyn1}--\ref{fn:Lyn2}.)} 
\item
Build a collection of simultaneous equations    $ \{\Lyn\fib_{z=\varrho} g_{\bm \alpha}(z)=0\}_{\bm \alpha\in\mathfrak L}$.  In each equation thus obtained, reduce GPLs to known expressions as far as possible.\footnote{This means that we need to build on the knowledge of $ G(0;\varrho)=\frac{\pi i}{3}$, $ G(1;\varrho)=-\frac{\pi i}{3}$, as well as all the  reductions of MCVs (along with their respective real/imaginary parts) available at lower weights. } Solve the remaining items  from the simultaneous equations. 
\end{itemize}
 When certain  $ g^{R}_{\bm\alpha}$ and $g^{I}_{\bm\alpha}$ join the collection, one can modify the procedures above by separating real and imaginary parts of each GPL expression, while treating $ \RE G(\bm \alpha;\varrho)$ and  $ \I G(\bm \alpha;\varrho)$
 as independent objects in the final step.

The procedures laid out in the last paragraph (which are fully programmable in Ma\^itre's \texttt{HPL} package \cite{Maitre2005,Maitre2012}  for \texttt{Mathematica}) account for  the entries of Table \ref{tab:MCV2345} at weights $\leq4$. An extra tweak will be required at weight $5$.
In the next few passages, we will reveal what exactly happens for the totality of Lyndon words $ \mathfrak L_n$ with lengths $ n\in\{2,3,4,5\}$.

At weight   $2$, we have $ \mathfrak L_{2}=\{01\}$ and $\Lyn\fib_{z=\varrho} g_{01}(z)=\frac{\pi^{2}}{2\cdot3}-G(0;\varrho)G(1;\varrho)+\frac{1}{2}[G(1;\varrho)]^2$ simplifies to $0$, so we cannot produce  non-trivial algebraic relations for MCVs when  not reading off their real/imaginary parts.
In the meantime,
we have $ \Lyn\fib_{z=\varrho} g^R_{01}(z)=\frac{\pi^{2}}{2\cdot3^{2}}-2\RE \mu_2$, which explains the tabulated entry for $ \RE \mu_2$.

At weight $3$, we have the simplifications \begin{align}
\Lyn\fib_{z=\varrho} g_{001}(z)={}&\mu _{3}+\zeta _{3}-2 \mu _{2,1}+\frac{2 \pi  i  \mu _{2}}{3}-\frac{7 \pi ^3 i }{2\cdot3^{4}},\\\Lyn\fib_{z=\varrho} g_{011}(z)={}&\mu _{3}-\zeta _{3}+\mu _{2,1}-\frac{ \pi  i  \mu _{2}}{3}-\frac{2 \pi ^3 i }{3^{4}},
\end{align}while taking into account the explicit evaluations of  $ G(0;\varrho)$ and $ G(1;\varrho)$. Thus, one may solve the two equations in     $ \{\Lyn\fib_{z=\varrho} g_{\bm \alpha}(z)=0\}_{\bm \alpha\in\mathfrak L_{3}}$  to verify the tabulated entries for $ \mu_3$ and $\mu_{2,1}$.

At weight $4$, the three members in      $ \{\Lyn\fib_{z=\varrho} g_{\bm \alpha}(z)=0\}_{\bm \alpha\in\mathfrak L_{4}}$  form a set of inhomogeneous linear equations in the formal variables $ \mu_4$, $\mu_{3,1}$, and $ \mu_{2,1,1}$. This is a  linear system of rank $2$, from which one may express  $\mu_{3,1}$, and $ \mu_{2,1,1}$ in terms of $\mu_4$ and MCVs of lower weights.
After this, one may solve $\RE\mu_4$ from       $ \{\Lyn\fib_{z=\varrho} g_{\bm \alpha}^{R}(z)=0\}_{\bm \alpha\in\mathfrak L_{4}}$.

At weight $5$, we have a linear system       $ \{\Lyn\fib_{z=\varrho} g_{\bm \alpha}(z)=0\}_{\bm \alpha\in\mathfrak L_{5}}$ of rank $4$, from which  one cannot deduce all the tabulated entries for $ \mu_5$, $ \mu_{3,2}$, $ \mu_{3,1,1}$, $ \mu_{2,2,1}$, and $ \mu_{2,1,1,1}$. Fortunately, we can make up for the rank deficit by throwing in an extra relation:\begin{align}
\vfib_{z=\varrho}\left[ G(0,0,z,0,z;1)-G\left( 0,0, z,0,1-\frac{1}{z};1\right) \right]=0.\label{eq:MCV5adj}
\end{align}     Here, when the operator  $ \vfib_{z=\varrho}$ hits a suitably designed expression $ \widetilde f(z)$, the workflow\footnote{This workflow can be completed in Panzer's  \texttt{HyperInt}  package \cite{Panzer2015} for \texttt{Maple}.} runs as follows:\begin{itemize}
\item 
Over the splitting field $ \mathbb Q\big(i,\sqrt{3}\big)$, fibrate $\widetilde f(z)$ with respect to the variable $z$, so that\begin{align}
\widetilde f(z)\in\Span_{\mathbb Q}\left\{ \left(\frac{\pi i\I z}{|\I z|}\right)^\ell Z_mG(\beta_{1},\dots,\beta_n;z)\middle|\begin{smallmatrix}\ell,m\in\mathbb Z_{\geq0},Z_m\in\mathfrak Z_m(1)\\\beta_{1},\dots,\beta_n\in\left\{ 0,1,\frac{1+i\sqrt{3}}{2} ,\frac{1-i\sqrt{3}}{2}\right\}\end{smallmatrix}\right\}.
\end{align}\item In the GPL parameters, replace $ \frac{1+i\sqrt{3}}{2}$ \Big(resp.\ $ \frac{1-i\sqrt{3}}{2}$\Big)
by $ z$ \big(resp.\ $ \frac{z}{z-1}$\big).\item Fibrate again with respect to the variable $z$,  before substituting $ z=\varrho$. \end{itemize}With the bonus relation \eqref{eq:MCV5adj} in place, we get a linear system of rank $5$, which meets our goal.  Consequently,    $ \RE\mu_{4,1}$ can be solved from        $ \{\Lyn\fib_{z=\varrho} g_{\bm \alpha}^{R}(z)=0\}_{\bm \alpha\in\mathfrak L_{5}}$.

\item At weight $6$, the linear system       $ \{\Lyn\fib_{z=\varrho} g_{\bm \alpha}(z)=0\}_{\bm \alpha\in\mathfrak L_{6}}$ has rank $6$. An extra relation
\begin{align}
\vfib_{z=\varrho}\left[ G(0,0,z,0,0,z;1)-G\left( 0,0, z,0,0,1-\frac{1}{z};1\right) \right]=0
\end{align}boosts the rank to $7$, which in turn, gives rise to the top $7$ entries of Table \ref{tab:MCV6}.
The rest of Table \ref{tab:MCV6} descend from      $ \{\Lyn\fib_{z=\varrho} g_{\bm \alpha}^{R}(z)=0\}_{\bm \alpha\in\mathfrak L_6}$    and      $ \{\Lyn\fib_{z=\varrho} g_{\bm \alpha}^{I}(z)=0\}_{\bm \alpha\in\mathfrak L_6}$.   
\qedhere\end{enumerate}\end{proof}\begin{remark}Henn--Smirnov--Smirnov \cite{HennSmirnovSmirnov2017} built an empirical database for the reductions of MPL/GPL expressions in $ \mathfrak Z_w(6)$, where $ 1\leq w\leq 6$. Regrettably,  the design for the empirical basis at weight $6$ in this database  was not optimal enough\footnote{Perhaps this defect was due to some erstwhile insurmountable challenges.  As of 2026 \cite[\S5.2]{Au2022a}, it appears that some algebraic relations in $ \mathfrak Z_6(6)$ are still not understood analytically. } 
to recover all the entries of Table \ref{tab:MCV6}. 
\eor\end{remark}

We may extend the service in the last proposition to  Lyndon words of lengths $7$ and $8$,  as shown below. \begin{proposition}\label{prop:Lyn78}\begin{enumerate}[leftmargin=*,  label=\emph{(\alph*)},ref=(\alph*),
widest=d, align=left] \item Certain Lyndon MCVs of weight  $7$ satisfy the algebraic relations listed in Table \ref{tab:MCV7}.\begin{table}[h]\caption{Provable reductions for  certain  Lyndon MCVs of weight 7\label{tab:MCV7}}{\scriptsize\begin{tabular}{r@{{\,}={\;}}p{.8\textwidth}}\hline\hline$\mu _7$&$\frac{7^{2}\cdot13 \zeta _{7}}{2^{4}\cdot3^{4}}+\frac{7\cdot13 \pi ^7 i }{2^{4}\cdot3^{9}}$\\
$\mu _{5,1,1}$&$-\frac{313 \zeta _{7}}{2^{2}\cdot3^{4}}+2 \mu _{6,1}+\frac{29 \pi ^2 \zeta _{5}}{2^{2}\cdot3^{5}}+\frac{ \pi  i  \zeta _{3}^2}{3^{2}}+\frac{19 \pi ^4 \zeta _{3}}{2^{3}\cdot3^{4}\cdot5}-\frac{97 \pi ^7 i }{2^{2}\cdot3^{9}\cdot7}$\\
$\mu _{4,3}$&$\frac{7\cdot137 \zeta _{7}}{2^{2}\cdot3^{5}}-\frac{7\cdot47 \mu _{6,1}}{3\cdot19}-\frac{2\cdot29 \mu _{5,2}}{19}+\frac{29 \pi  i  \mu _{6}}{3^{2}\cdot19}+\frac{\zeta _{3} \mu _{4}}{3}+\frac{5 \pi ^3 i  \mu _{4}}{2\cdot3^{4}}-\frac{19927 \pi ^7 i }{2^{4}\cdot3^{8}\cdot5\cdot7\cdot19}$\\
$\mu _{4,2,1}$&$\frac{9733 \zeta _{7}}{2^{4}\cdot3^{5}}-\frac{193 \mu _{6,1}}{3\cdot19}+\frac{3^{2} \mu _{5,2}}{19}+\frac{2^{3}\cdot41 \pi  i  \mu _{6}}{3^{2}\cdot19}+\frac{ \pi  i  \mu _{4,2}}{3}+\frac{29 \pi ^2 \zeta _{5}}{3^{5}}+\frac{2 \zeta _{3} \mu _{4}}{3}-\frac{ \pi ^3 i  \mu _{4}}{2\cdot3^{4}}-\frac{ \pi  i  \zeta _{3}^2}{3}+\frac{43 \pi ^4 \zeta _{3}}{2^{2}\cdot3^{5}\cdot5}-\frac{7\cdot1423 \pi ^7 i }{2^{3}\cdot3^{9}\cdot5\cdot19}$\\
$\mu _{4,1,1,1}$&$-\frac{5\cdot41 \zeta _{7}}{2^{3}\cdot3^{3}}+2 \mu _{6,1}+\frac{ \pi  i  \mu _{6}}{3}+\frac{\pi ^2 \zeta _{5}}{3^{2}}-\frac{ \pi ^2 \mu _{4,1}}{2\cdot3^{2}}+\frac{ \pi ^3 i  \mu _{4}}{3^{4}}+\frac{ \pi  i  \zeta _{3}^2}{2\cdot3}+\frac{\pi ^4 \zeta _{3}}{2^{2}\cdot3^{3}}-\frac{5279 \pi ^7 i }{2^{4}\cdot3^{9}\cdot5\cdot7}$\\
$\mu _{3,3,1}$&$-\frac{17\cdot311 \zeta _{7}}{2^{3}\cdot3^{5}}-\frac{2\cdot173 \mu _{6,1}}{3\cdot19}-\frac{3^{2}\cdot5 \mu _{5,2}}{19}-\frac{1013 \pi  i  \mu _{6}}{3^{2}\cdot19}-\pi  i  \mu _{4,2}-\frac{29 \pi ^2 \zeta _{5}}{2\cdot3^{3}}-2 \zeta _{3} \mu _{4}+\frac{ \pi ^3 i  \mu _{4}}{2\cdot3^{3}}+\frac{11 \pi  i  \zeta _{3}^2}{3^{3}}-\frac{281 \pi ^4 \zeta _{3}}{2^{2}\cdot3^{5}\cdot5}+\frac{159697 \pi ^7 i }{2^{2}\cdot3^{9}\cdot5\cdot7\cdot19}$\\
$\mu _{3,2,2}$&$\frac{97\cdot103 \zeta _{7}}{2^{4}\cdot3^{6}}+\frac{2213 \mu _{6,1}}{3^{2}\cdot19}+\frac{2^{4}\cdot3 \mu _{5,2}}{19}-\frac{7 \mu _{4,1,2}}{3}+\frac{859 \pi  i  \mu _{6}}{3^{3}\cdot19}+\frac{2 \pi  i  \mu _{4,2}}{3^{2}}+\frac{2^{2}\cdot29 \pi ^2 \zeta _{5}}{3^{6}}+\frac{2^{3} \zeta _{3} \mu _{4}}{3^{2}}-\frac{2 \pi ^3 i  \mu _{4}}{3^{5}}-\frac{2 \pi  i  \zeta _{3}^2}{3^{3}}+\frac{2\cdot5 \pi ^4 \zeta _{3}}{3^{6}}-\frac{103\cdot373 \pi ^7 i }{2^{3}\cdot3^{8}\cdot5\cdot7\cdot19}$\\
$\mu _{3,2,1,1}$&$\frac{5\cdot1523 \zeta _{7}}{2^{4}\cdot3^{5}}-\frac{193 \mu _{6,1}}{3\cdot19}+\frac{3^{2} \mu _{5,2}}{19}-\frac{1579 \pi  i  \mu _{6}}{3^{2}\cdot11\cdot19}-\frac{7 \pi ^2 \zeta _{5}}{3^{3}}+\frac{7 \pi ^2 \mu _{4,1}}{2\cdot3^{3}}-\frac{2 \pi ^3 i  \mu _{4}}{3^{4}}-\frac{19 \pi  i  \zeta _{3}^2}{2\cdot3^{3}}-\frac{191 \pi ^4 \zeta _{3}}{2^{3}\cdot3^{5}\cdot5}+\frac{137\cdot271 \pi ^7 i }{2^{3}\cdot3^{8}\cdot5\cdot7\cdot19}$\\
$\mu _{3,1,2,1}$&$-\frac{37\cdot79 \zeta _{7}}{2^{3}\cdot3^{5}}-\frac{2\cdot173 \mu _{6,1}}{3\cdot19}-\frac{3^{2}\cdot5 \mu _{5,2}}{19}+\frac{37\cdot47 \pi  i  \mu _{6}}{3^{2}\cdot11\cdot19}+\frac{ \pi  i  \mu _{4,2}}{3}+\frac{2^{3}\cdot5 \pi ^2 \zeta _{5}}{3^{5}}+\frac{ \pi ^2 \mu _{4,1}}{3^{3}}+\frac{2 \zeta _{3} \mu _{4}}{3}-\frac{ \pi ^3 i  \mu _{4}}{2\cdot3^{3}}-\frac{2^{3} \pi  i  \zeta _{3}^2}{3^{3}}+\frac{181 \pi ^4 \zeta _{3}}{2^{2}\cdot3^{6}\cdot5}-\frac{661 \pi ^7 i }{2\cdot3^{9}\cdot7\cdot19}$\\
$\mu _{3,1,1,2}$&$-\frac{5\cdot263 \zeta _{7}}{2^{4}\cdot3^{4}}+\frac{67 \mu _{6,1}}{19}+\frac{2^{3}\cdot3 \mu _{5,2}}{19}+\mu _{4,1,2}+\frac{317 \pi  i  \mu _{6}}{3\cdot11\cdot19}+\frac{13 \pi  i  \zeta _{3}^2}{3^{3}}+\frac{31 \pi ^4 \zeta _{3}}{2\cdot3^{6}}-\frac{163\cdot647 \pi ^7 i }{2^{3}\cdot3^{9}\cdot5\cdot7\cdot19}$\\
$\mu _{3,1,1,1,1}$&$-\frac{593 \zeta _{7}}{2^{4}\cdot3^{4}}+\mu _{6,1}+\frac{ \pi  i  \mu _{6}}{3}+\frac{43 \pi ^2 \zeta _{5}}{2^{2}\cdot3^{4}}-\frac{ \pi ^2 \mu _{4,1}}{2\cdot3^{2}}+\frac{ \pi ^3 i  \mu _{4}}{3^{4}}+\frac{ \pi  i  \zeta _{3}^2}{2\cdot3^{2}}+\frac{89 \pi ^4 \zeta _{3}}{2^{3}\cdot3^{6}\cdot5}-\frac{1129 \pi ^7 i }{2^{4}\cdot3^{8}\cdot5\cdot7}$\\
$\mu _{2,2,2,1}$&$-\frac{5\cdot28961 \zeta _{7}}{2^{4}\cdot3^{6}}+\frac{1597 \mu _{6,1}}{3^{2}\cdot19}-\frac{\mu _{5,2}}{19}+\frac{7 \mu _{4,1,2}}{3}+\frac{2\cdot11261 \pi  i  \mu _{6}}{3^{3}\cdot11\cdot19}+\frac{7 \pi  i  \mu _{4,2}}{3^{2}}+\frac{7^{2} \zeta _{5} \mu _{2}}{2\cdot3^{3}}+\frac{2\cdot311 \pi ^2 \zeta _{5}}{3^{6}}-\frac{7 \mu _{2} \mu _{4,1}}{3}-\frac{2^{2} \pi ^2 \mu _{4,1}}{3^{3}}+\frac{2^{4} \zeta _{3} \mu _{4}}{3^{2}}+\frac{ \pi  i  \mu _{4} \mu _{2}}{3^{2}}+\frac{2^{3} \pi ^3 i  \mu _{4}}{3^{5}}+\frac{ \zeta _{3} \mu _{2}^2}{3}+\frac{5\cdot7 \pi  i  \zeta _{3}^2}{3^{3}}-\frac{2 \pi ^2 \zeta _{3} \mu _{2}}{3^{3}}+\frac{5\cdot37 \pi ^4 \zeta _{3}}{2\cdot3^{6}}+\frac{ \pi  i  \mu _{2}^3}{2\cdot3^{2}}-\frac{ \pi ^3 i  \mu _{2}^2}{2^{2}\cdot3^{4}}-\frac{97 \pi ^5 i  \mu _{2}}{2^{3}\cdot3^{7}\cdot5}-\frac{61\cdot1811 \pi ^7 i }{2^{4}\cdot3^{7}\cdot5\cdot7\cdot19}$\\
$\mu _{2,2,1,1,1}$&$\frac{2^{3}\cdot29 \zeta _{7}}{3^{4}}+\mu _{5,2}-2 \pi  i  \mu _{6}-\frac{ \pi  i  \mu _{4,2}}{3}+\frac{29 \zeta _{5} \mu _{2}}{2\cdot3^{3}}-\frac{17^{2} \pi ^2 \zeta _{5}}{2\cdot3^{5}}+\frac{7 \pi ^2 \mu _{4,1}}{2\cdot3^{3}}+\frac{ \pi  i  \mu _{4} \mu _{2}}{3}-\frac{5 \pi ^3 i  \mu _{4}}{2\cdot3^{4}}+\frac{ \pi ^2 \zeta _{3} \mu _{2}}{2\cdot3^{3}}-\frac{389 \pi ^4 \zeta _{3}}{2^{2}\cdot3^{6}\cdot5}-\frac{ \pi ^3 i  \mu _{2}^2}{2^{2}\cdot3^{4}}-\frac{41 \pi ^5 i  \mu _{2}}{2^{3}\cdot3^{6}\cdot5}+\frac{12653 \pi ^7 i }{2^{4}\cdot3^{9}\cdot5\cdot7}$\\
$\mu _{2,1,2,1,1}$&$-\frac{24697 \zeta _{7}}{2^{4}\cdot3^{5}}-\frac{7\cdot47 \mu _{6,1}}{3\cdot19}-\frac{2\cdot29 \mu _{5,2}}{19}+\frac{2^{4}\cdot41 \pi  i  \mu _{6}}{3^{2}\cdot19}+\pi  i  \mu _{4,2}-\frac{2\cdot29 \zeta _{5} \mu _{2}}{3^{3}}+\frac{2^{4}\cdot5 \pi ^2 \zeta _{5}}{3^{4}}-\frac{ \pi ^2 \mu _{4,1}}{2\cdot3^{2}}-\pi  i  \mu _{2} \mu _{4}+\frac{ \pi ^3 i  \mu _{4}}{2\cdot3^{3}}-\frac{17 \pi  i  \zeta _{3}^2}{2\cdot3^{3}}-\frac{ \pi ^2 \zeta _{3} \mu _{2}}{3^{3}}+\frac{13\cdot47 \pi ^4 \zeta _{3}}{2^{3}\cdot3^{5}\cdot5}+\frac{ \pi ^3 i  \mu _{2}^2}{2^{2}\cdot3^{4}}+\frac{131 \pi ^5 i  \mu _{2}}{2^{3}\cdot3^{6}\cdot5}-\frac{11\cdot317 \pi ^7 i }{2^{3}\cdot3^{7}\cdot7\cdot19}$\\
$\mu _{2,1,1,1,1,1}$&$\frac{659 \zeta _{7}}{2^{4}\cdot3^{4}}+\frac{ \pi  i  \mu _{6}}{3}+\frac{5^{2} \pi ^2 \zeta _{5}}{2^{2}\cdot3^{5}}-\frac{ \pi ^3 i  \mu _{4}}{2\cdot3^{4}}-\frac{\pi ^4 \zeta _{3}}{2^{3}\cdot3^{6}}+\frac{\pi ^5 i  \mu _{2}}{2^{3}\cdot3^{6}\cdot5}-\frac{1567 \pi ^7 i }{2^{4}\cdot3^{9}\cdot5\cdot7}$\\$\RE\mu _{6,1}$&$\frac{4481 \zeta _{7}}{2^{5}\cdot3^{4}}-\frac{ \pi  \I\mu _{6}}{3}-\frac{\pi ^2 \zeta _{5}}{2^{2}\cdot3^{2}}-\frac{7\cdot13 \pi ^4 \zeta _{3}}{2^{4}\cdot3^{5}\cdot5}$\\
$\RE\mu _{5,2}$&$-\frac{1123 \zeta _{7}}{2\cdot3^{4}}+\frac{5 \pi  \I\mu _{6}}{3}+\frac{241 \pi ^2 \zeta _{5}}{2^{3}\cdot3^{5}}+\frac{7\cdot13 \pi ^4 \zeta _{3}}{2^{3}\cdot3^{5}\cdot5}$ \\$\I\mu _{4,1,2}$&$\frac{2^{2}\cdot11 \I\mu _{6,1}}{3\cdot19}+\frac{3\cdot7 \I\mu _{5,2}}{2\cdot19}-\frac{ \pi  \RE\mu _{4,2}}{3}+\frac{ \pi ^2 \I\mu _{4,1}}{2^{2}\cdot3^{2}}-\frac{2^{2} \zeta _{3} \I\mu _{4}}{3}-\frac{ \pi  \zeta _{3}^2}{3^{2}}+\frac{7\cdot29^{2} \pi ^7}{2^{5}\cdot3^{10}\cdot5}$\\\hline\hline\end{tabular}}\end{table}
\item Certain Lyndon MCVs of weight  $8$ satisfy the algebraic relations listed in Table \ref{tab:MCV8}.\begin{table}[h]\caption{Provable reductions for  certain  Lyndon MCVs of weight 8\label{tab:MCV8}}{\scriptsize\begin{tabular}{r@{{\,}={\;}}p{.8\textwidth}}\hline\hline $\mu _{7,1}$&$3\mu _{8}-\frac{7\zeta _{5,3}}{2\cdot3^{3}\cdot5}-\frac{659\pi  i  \zeta _{7}}{2^{4}\cdot3^{5}}-\frac{5^{2} \zeta _{5} \zeta _{3}}{2\cdot3^{3}}-\frac{5\pi ^{3} i  \zeta _{5}}{2\cdot3^{4}}-\frac{17\pi ^{5} i  \zeta _{3}}{2^{3}\cdot3^{6}}-\frac{91631\pi ^{8}}{2^{7}\cdot3^{10}\cdot5^{3}}$\\
$\mu _{5,3}$&$-2\cdot19\mu _{8}-5\mu _{6,2}-\frac{31\zeta _{5,3}}{2\cdot3^{4}\cdot5}+\frac{5\cdot659\pi  i  \zeta _{7}}{2^{4}\cdot3^{4}}+\frac{5^{3} \zeta _{5} \zeta _{3}}{3^{3}}+\frac{5\cdot11\cdot71\pi ^{3} i  \zeta _{5}}{2^{2}\cdot3^{7}}+\frac{5\cdot17\pi ^{5} i  \zeta _{3}}{2^{2}\cdot3^{6}}+\frac{103\cdot44263\pi ^{8}}{2^{6}\cdot3^{10}\cdot5^{3}\cdot7}$\\
$\mu _{5,2,1}$&$-\frac{2\cdot5\cdot29\mu _{8}}{43}-\frac{2^{3}\cdot17\zeta _{5,3}}{3^{4}\cdot5}-2\mu _{6,1,1}+\frac{5\pi  i  \zeta _{7}}{3}+\frac{2\pi  i  \mu _{6,1}}{3}+\frac{ \pi  i  \mu _{5,2}}{3}-\frac{\pi ^{2} \mu _{6}}{2\cdot3}+\frac{5^{2} \zeta _{5} \zeta _{3}}{2\cdot3^{3}}+\frac{2^{2}\cdot17\pi ^{3} i  \zeta _{5}}{3^{7}}+\frac{17\pi ^{5} i  \zeta _{3}}{2^{3}\cdot3^{6}}+\frac{397\cdot275729\pi ^{8}}{2^{7}\cdot3^{10}\cdot5^{3}\cdot7\cdot43}$\\
$\mu _{5,1,2}$&$\frac{2\cdot7\cdot109\mu _{8}}{43}+2\mu _{6,2}+\frac{2\cdot19\zeta _{5,3}}{3^{3}\cdot5}-\frac{7\mu _{6,1,1}}{3}-\frac{983\pi  i  \zeta _{7}}{2^{2}\cdot3^{4}}-\frac{2^{3} \pi  i  \mu _{6,1}}{3^{2}}-\frac{ \pi  i  \mu _{5,2}}{3}+\frac{7\pi ^{2} \mu _{6}}{2\cdot3^{3}}-\frac{5\cdot113\zeta _{5} \zeta _{3}}{2\cdot3^{4}}-\frac{3217\pi ^{3} i  \zeta _{5}}{2^{2}\cdot3^{7}}-\frac{421\pi ^{5} i  \zeta _{3}}{2^{2}\cdot3^{6}\cdot5}-\frac{8087\cdot12979\pi ^{8}}{2^{5}\cdot3^{10}\cdot5^{3}\cdot7\cdot43}$\\
$\mu _{5,1,1,1}$&$-5\mu _{8}+\frac{7\zeta _{5,3}}{2\cdot3^{3}}+2\mu _{6,1,1}+\frac{11\pi  i  \zeta _{7}}{2^{2}\cdot3^{5}}+\frac{71\zeta _{5} \zeta _{3}}{2\cdot3^{3}}+\frac{569\pi ^{3} i  \zeta _{5}}{2^{2}\cdot3^{7}}-\frac{ \pi ^{2} \zeta _{3}^2}{2\cdot3^{3}}+\frac{2\cdot11\pi ^{5} i  \zeta _{3}}{3^{6}\cdot5}+\frac{38729\pi ^{8}}{2^{5}\cdot3^{10}\cdot5^{2}\cdot7}$\\
$\mu _{4,3,1}$&$\frac{509\mu _{8}}{43}+\frac{5\cdot19\zeta _{5,3}}{3^{4}}+\mu _{6,1,1}-\frac{7\cdot13\cdot139\pi  i  \zeta _{7}}{2^{2}\cdot3^{6}}-\frac{2\cdot193\pi  i  \mu _{6,1}}{3^{2}\cdot19}-\frac{2\cdot29\pi  i  \mu _{5,2}}{3\cdot19}+\frac{5\cdot7\cdot13\pi ^{2} \mu _{6}}{2\cdot3^{3}\cdot19}+\frac{5^{2} \zeta _{3} \zeta _{5}}{2\cdot3^{3}}-\frac{5\cdot61\pi ^{3} i  \zeta _{5}}{2^{2}\cdot3^{7}}+\frac{\mu _{4}^2}{2}-\frac{2\pi  i  \zeta _{3} \mu _{4}}{3^{2}}-\frac{19\pi ^{4} \mu _{4}}{2^{3}\cdot3^{4}\cdot5}+\frac{17\pi ^{5} i  \zeta _{3}}{2^{3}\cdot3^{6}}-\frac{271\cdot9371\pi ^{8}}{2^{7}\cdot3^{6}\cdot5\cdot7\cdot19\cdot43}$\\
$\mu _{4,2,1,1}$&$\frac{2^{2}\cdot53\mu _{8}}{43}-\frac{151\zeta _{5,3}}{2\cdot3^{3}\cdot5}-2^{2}\mu _{6,1,1}+\frac{5\cdot839\pi  i  \zeta _{7}}{2^{2}\cdot3^{6}}+\frac{5\cdot7\pi  i  \mu _{6,1}}{3^{2}\cdot19}+\frac{3\pi  i  \mu _{5,2}}{19}-\frac{2\cdot107\pi ^{2} \mu _{6}}{3^{3}\cdot19}-\frac{ \pi ^{2} \mu _{4,2}}{2\cdot3^{2}}-\frac{71\zeta _{5} \zeta _{3}}{3^{3}}-\frac{131\pi ^{3} i  \zeta _{5}}{3^{7}}+\frac{\mu _{4}^2}{2}-\frac{ \pi  i  \zeta _{3} \mu _{4}}{3^{2}}-\frac{11\pi ^{4} \mu _{4}}{2^{3}\cdot3^{4}\cdot5}+\frac{ \pi ^{2} \zeta _{3}^2}{2\cdot3^{2}}-\frac{7\pi ^{5} i  \zeta _{3}}{2\cdot3^{4}\cdot5}+\frac{103\cdot11868151\pi ^{8}}{2^{7}\cdot3^{10}\cdot5^{3}\cdot7\cdot19\cdot43}$\\
$\mu _{4,1,3}$&$-\frac{2^{5}\cdot3\cdot47\mu _{8}}{43}-2^{3}\mu _{6,2}-\frac{29\cdot47\zeta _{5,3}}{2\cdot3^{4}\cdot5}+5\mu _{6,1,1}-\mu _{4,2,2}+\frac{24709\pi  i  \zeta _{7}}{2^{2}\cdot3^{6}}+\frac{2\cdot307\pi  i  \mu _{6,1}}{3^{2}\cdot19}+\frac{2\cdot29\pi  i  \mu _{5,2}}{3\cdot19}-\frac{227\pi ^{2} \mu _{6}}{2\cdot3^{3}\cdot19}+\frac{2\cdot5^{2} \zeta _{3} \zeta _{5}}{3^{2}}+\frac{17\cdot641\pi ^{3} i  \zeta _{5}}{2^{2}\cdot3^{7}}+\frac{ \zeta _{3} \mu _{4,1}}{3}+\frac{5\pi ^{3} i  \mu _{4,1}}{2\cdot3^{4}}-\mu _{4}^2+\frac{2^{2} \pi  i  \zeta _{3} \mu _{4}}{3^{2}}+\frac{19\pi ^{4} \mu _{4}}{2^{2}\cdot3^{4}\cdot5}+\frac{1217\pi ^{5} i  \zeta _{3}}{2^{3}\cdot3^{7}}+\frac{47\cdot487\cdot167221\pi ^{8}}{2^{7}\cdot3^{10}\cdot5^{3}\cdot19\cdot43}$\\
$\mu _{4,1,2,1}$&$\frac{2\cdot311\mu _{8}}{43}+\frac{277\zeta _{5,3}}{2\cdot3^{4}\cdot5}-\frac{5\mu _{6,1,1}}{3}-\frac{3\cdot5\pi  i  \zeta _{7}}{2^{3}}+\frac{5\pi  i  \mu _{6,1}}{3^{2}}+\frac{ \pi  i  \mu _{4,1,2}}{3}-\frac{19\pi ^{2} \mu _{6}}{2\cdot3^{3}}-\frac{ \pi ^{2} \mu _{4,2}}{2\cdot3^{2}}-\frac{2^{3}\cdot17\zeta _{5} \zeta _{3}}{3^{4}}-\frac{2\cdot71\pi ^{3} i  \zeta _{5}}{3^{6}}+\frac{2\zeta _{3} \mu _{4,1}}{3}-\frac{ \pi ^{3} i  \mu _{4,1}}{2\cdot3^{4}}-\frac{3\mu _{4}^2}{2}+\frac{7\pi  i  \zeta _{3} \mu _{4}}{3^{2}}+\frac{139\pi ^{4} \mu _{4}}{2^{3}\cdot3^{5}\cdot5}-\frac{ \pi ^{2} \zeta _{3}^2}{2\cdot3^{3}}+\frac{157\pi ^{5} i  \zeta _{3}}{2^{2}\cdot3^{7}\cdot5}-\frac{31\cdot488797\pi ^{8}}{2^{5}\cdot3^{10}\cdot5^{3}\cdot7\cdot43}$\\
$\mu _{4,1,1,2}$&$\frac{1259\mu _{8}}{43}+3\mu _{6,2}-\frac{2^{2}\cdot13\zeta _{5,3}}{3^{4}\cdot5}-\frac{2\mu _{6,1,1}}{3}-\frac{2^{2}\cdot181\pi  i  \zeta _{7}}{3^{6}}-\frac{5^{2} \pi  i  \mu _{6,1}}{19}-\frac{3\pi  i  \mu _{5,2}}{19}-\frac{ \pi  i  \mu _{4,1,2}}{3}+\frac{347\pi ^{2} \mu _{6}}{3^{3}\cdot19}+\frac{ \pi ^{2} \mu _{4,2}}{3^{2}}-\frac{151\zeta _{5} \zeta _{3}}{3^{4}}-\frac{13\cdot47\pi ^{3} i  \zeta _{5}}{2\cdot3^{6}}-\frac{2^{2} \zeta _{3} \mu _{4,1}}{3}+\frac{ \pi ^{3} i  \mu _{4,1}}{3^{4}}+\frac{3\mu _{4}^2}{2}-\frac{7\pi  i  \zeta _{3} \mu _{4}}{3^{2}}-\frac{139\pi ^{4} \mu _{4}}{2^{3}\cdot3^{5}\cdot5}+\frac{ \pi ^{2} \zeta _{3}^2}{3^{3}}-\frac{3187\pi ^{5} i  \zeta _{3}}{2^{3}\cdot3^{7}\cdot5}-\frac{157\cdot7801231\pi ^{8}}{2^{4}\cdot3^{10}\cdot5^{3}\cdot7\cdot19\cdot43}$\\
$\mu _{4,1,1,1,1}$&$\frac{7\zeta _{5,3}}{2\cdot3^{3}}+\mu _{6,1,1}-\frac{7^{2}\cdot13\pi  i  \zeta _{7}}{2^{3}\cdot3^{4}}+\frac{ \pi  i  \mu _{6,1}}{3}+\frac{17\zeta _{5} \zeta _{3}}{2\cdot3^{3}}+\frac{2^{2} \pi ^{3} i  \zeta _{5}}{3^{4}}-\frac{ \pi ^{3} i  \mu _{4,1}}{2\cdot3^{4}}-\frac{\pi ^{4} \mu _{4}}{2^{3}\cdot3^{4}}-\frac{ \pi ^{2} \zeta _{3}^2}{2^{2}\cdot3^{2}}+\frac{11\pi ^{5} i  \zeta _{3}}{2\cdot3^{5}\cdot5}+\frac{509\pi ^{8}}{2^{5}\cdot3^{10}\cdot5\cdot7}$\\
$\mu _{3,3,2}$&$-\frac{2\cdot13099\mu _{8}}{3\cdot43}-2^{2}\cdot3\mu _{6,2}-\frac{2^{3}\cdot7\cdot83\zeta _{5,3}}{3^{5}\cdot5}+\frac{5^{2}\cdot7\mu _{6,1,1}}{3^{2}}-3\mu _{4,2,2}+\frac{17\cdot1237\pi  i  \zeta _{7}}{2^{2}\cdot3^{5}}+\frac{2\cdot1987\pi  i  \mu _{6,1}}{3^{3}\cdot19}+\frac{2^{3}\cdot7\pi  i  \mu _{5,2}}{19}-\frac{3847\pi ^{2} \mu _{6}}{2\cdot3^{4}\cdot19}+\frac{2^{7}\cdot5^{2} \zeta _{5} \zeta _{3}}{3^{5}}+\frac{7\cdot439\pi ^{3} i  \zeta _{5}}{2\cdot3^{6}}+\frac{2459\pi ^{5} i  \zeta _{3}}{2^{3}\cdot3^{7}}+\frac{11\cdot13\cdot2593\cdot420809\pi ^{8}}{2^{7}\cdot3^{11}\cdot5^{3}\cdot7\cdot19\cdot43}$\\
$\mu _{3,3,1,1}$&$\frac{2\cdot311\mu _{8}}{43}+\frac{7\cdot109\zeta _{5,3}}{2\cdot3^{4}\cdot5}-\frac{5\mu _{6,1,1}}{3}-\frac{5\cdot19\cdot157\pi  i  \zeta _{7}}{2^{2}\cdot3^{6}}-\frac{251\pi  i  \mu _{6,1}}{3^{2}\cdot19}-\frac{3\cdot5\pi  i  \mu _{5,2}}{19}+\frac{479\pi ^{2} \mu _{6}}{2\cdot3^{2}\cdot19}+\frac{ \pi ^{2} \mu _{4,2}}{2\cdot3}+\frac{2\cdot13\zeta _{5} \zeta _{3}}{3^{4}}-\frac{127\pi ^{3} i  \zeta _{5}}{2\cdot3^{6}}-\frac{3\mu _{4}^2}{2}+\frac{ \pi  i  \zeta _{3} \mu _{4}}{3}+\frac{11\pi ^{4} \mu _{4}}{2^{3}\cdot3^{3}\cdot5}-\frac{11\pi ^{2} \zeta _{3}^2}{2\cdot3^{4}}-\frac{1459\pi ^{5} i  \zeta _{3}}{2^{3}\cdot3^{7}\cdot5}-\frac{1831\cdot152681\pi ^{8}}{2^{5}\cdot3^{10}\cdot5^{3}\cdot19\cdot43}$\\
$\mu _{3,2,2,1}$&$-\frac{2^{4}\cdot373\mu _{8}}{3\cdot43}-\frac{2659\zeta _{5,3}}{2\cdot3^{5}\cdot5}+\frac{2\cdot7^{2} \mu _{6,1,1}}{3^{2}}+\frac{11\cdot13\cdot967\pi  i  \zeta _{7}}{2^{3}\cdot3^{7}}+\frac{13\pi  i  \mu _{6,1}}{19}+\frac{2^{4} \pi  i  \mu _{5,2}}{19}-\frac{7\pi  i  \mu _{4,1,2}}{3^{2}}+\frac{2\cdot131\pi ^{2} \mu _{6}}{3^{4}\cdot19}+\frac{5\pi ^{2} \mu _{4,2}}{2\cdot3^{3}}+\frac{2^{2}\cdot283\zeta _{5} \zeta _{3}}{3^{5}}+\frac{2\cdot5\cdot71\pi ^{3} i  \zeta _{5}}{3^{7}}-\frac{2\cdot7\zeta _{3} \mu _{4,1}}{3^{2}}+\frac{7\pi ^{3} i  \mu _{4,1}}{2\cdot3^{5}}+\frac{3^{2} \mu _{4}^2}{2}-\frac{59\pi  i  \zeta _{3} \mu _{4}}{3^{3}}-\frac{7\cdot173\pi ^{4} \mu _{4}}{2^{3}\cdot3^{6}\cdot5}+\frac{ \pi ^{2} \zeta _{3}^2}{2\cdot3^{2}}+\frac{11\cdot461\pi ^{5} i  \zeta _{3}}{2^{3}\cdot3^{8}\cdot5}+\frac{59\cdot71\cdot1748407\pi ^{8}}{2^{7}\cdot3^{10}\cdot5^{3}\cdot7\cdot19\cdot43}$\\
$\mu _{3,2,1,2}$&$-\frac{73\cdot149\mu _{8}}{3\cdot43}-\frac{7\cdot13\mu _{6,2}}{11}+\frac{2^{2}\cdot7\cdot11\zeta _{5,3}}{3^{5}\cdot5}+\frac{7\mu _{6,1,1}}{3^{2}}+\frac{5\cdot17\cdot3121\pi  i  \zeta _{7}}{2^{2}\cdot3^{7}\cdot11}+\frac{163\pi  i  \mu _{6,1}}{3\cdot19}+\frac{5\pi  i  \mu _{5,2}}{19}+\frac{7\pi  i  \mu _{4,1,2}}{3^{2}}-\frac{13\cdot359\pi ^{2} \mu _{6}}{2\cdot3^{4}\cdot19}-\frac{7\pi ^{2} \mu _{4,2}}{3^{3}}+\frac{2^{2}\cdot7^{2}\cdot97\zeta _{5} \zeta _{3}}{3^{5}\cdot11}+\frac{7\cdot89\cdot101\pi ^{3} i  \zeta _{5}}{2\cdot3^{7}\cdot11}+\frac{2^{2}\cdot7\zeta _{3} \mu _{4,1}}{3^{2}}-\frac{7\pi ^{3} i  \mu _{4,1}}{3^{5}}-\frac{3^{2} \mu _{4}^2}{2}+\frac{7\pi  i  \zeta _{3} \mu _{4}}{3}+\frac{139\pi ^{4} \mu _{4}}{2^{3}\cdot3^{4}\cdot5}-\frac{7\pi ^{2} \zeta _{3}^2}{3^{4}}+\frac{7\cdot4447\pi ^{5} i  \zeta _{3}}{3^{8}\cdot5\cdot11}+\frac{7691\cdot110154269\pi ^{8}}{2^{7}\cdot3^{11}\cdot5^{3}\cdot7\cdot11\cdot19\cdot43}$\\
$\mu _{3,2,1,1,1}$&$-\frac{2\cdot5\cdot29\mu _{8}}{43}-\frac{101\zeta _{5,3}}{2\cdot3^{4}}-2\mu _{6,1,1}+\frac{5\cdot29\cdot131\pi  i  \zeta _{7}}{2^{3}\cdot3^{6}}-\frac{79\pi  i  \mu _{6,1}}{3^{2}\cdot19}+\frac{3\pi  i  \mu _{5,2}}{19}-\frac{5^{2}\cdot31\pi ^{2} \mu _{6}}{2\cdot3^{3}\cdot11\cdot19}-\frac{73\zeta _{5} \zeta _{3}}{2\cdot3^{4}}-\frac{19\pi ^{3} i  \zeta _{5}}{2^{2}\cdot3^{5}}+\frac{7\pi ^{3} i  \mu _{4,1}}{2\cdot3^{5}}+\frac{19\pi ^{4} \mu _{4}}{2^{3}\cdot3^{6}}+\frac{19\pi ^{2} \zeta _{3}^2}{2^{2}\cdot3^{4}}-\frac{7^{2}\cdot17\pi ^{5} i  \zeta _{3}}{2^{3}\cdot3^{7}\cdot5}+\frac{2251\cdot100411\pi ^{8}}{2^{7}\cdot3^{11}\cdot5\cdot7\cdot19\cdot43}$\\\hline\hline\end{tabular}}\end{table}\addtocounter{table}{-1}\begin{table}[t]\caption{ \textit{(Continued)}}{\scriptsize\begin{tabular}{r@{{\,}={\;}}p{.8\textwidth}}\hline\hline
$\mu _{3,1,2,2}$&$\frac{7^{2}\cdot103\mu _{8}}{43}+\frac{2\cdot47\mu _{6,2}}{11}+\frac{1931\zeta _{5,3}}{2\cdot3^{4}\cdot5}-3^{2}\mu _{6,1,1}+\mu _{4,2,2}-\frac{173\cdot3253\pi  i  \zeta _{7}}{2\cdot3^{7}\cdot11}-\frac{2^{9}\cdot5\pi  i  \mu _{6,1}}{3^{3}\cdot19}-\frac{2^{2}\cdot5^{2} \pi  i  \mu _{5,2}}{3\cdot19}+\frac{2\pi  i  \mu _{4,1,2}}{3^{2}}+\frac{11\cdot47\pi ^{2} \mu _{6}}{2\cdot3^{4}\cdot19}-\frac{2\pi ^{2} \mu _{4,2}}{3^{3}}-\frac{2\cdot3217\zeta _{5} \zeta _{3}}{3^{4}\cdot11}-\frac{17\cdot6899\pi ^{3} i  \zeta _{5}}{2^{2}\cdot3^{7}\cdot11}+\frac{2^{3} \zeta _{3} \mu _{4,1}}{3^{2}}-\frac{2\pi ^{3} i  \mu _{4,1}}{3^{5}}-\frac{2\pi ^{2} \zeta _{3}^2}{3^{4}}-\frac{217157\pi ^{5} i  \zeta _{3}}{2^{3}\cdot3^{8}\cdot5\cdot11}-\frac{4073\cdot115921373\pi ^{8}}{2^{6}\cdot3^{11}\cdot5^{3}\cdot7\cdot11\cdot19\cdot43}$\\
$\mu _{3,1,2,1,1}$&$\frac{509\mu _{8}}{43}+\frac{2\cdot59\zeta _{5,3}}{3^{4}}+\mu _{6,1,1}-\frac{7\cdot23\cdot37\pi  i  \zeta _{7}}{2^{2}\cdot3^{6}}-\frac{13\cdot31\pi  i  \mu _{6,1}}{3^{2}\cdot19}-\frac{3\cdot5\pi  i  \mu _{5,2}}{19}+\frac{7\cdot179\pi ^{2} \mu _{6}}{2\cdot3^{3}\cdot11\cdot19}-\frac{ \pi ^{2} \mu _{4,2}}{2\cdot3^{2}}+\frac{5\cdot37\zeta _{5} \zeta _{3}}{3^{4}}-\frac{5\cdot113\pi ^{3} i  \zeta _{5}}{2^{2}\cdot3^{6}}+\frac{ \pi ^{3} i  \mu _{4,1}}{3^{5}}+\frac{\mu _{4}^2}{2}-\frac{ \pi  i  \zeta _{3} \mu _{4}}{3^{2}}-\frac{59\pi ^{4} \mu _{4}}{2^{3}\cdot3^{6}\cdot5}+\frac{ \pi ^{2} \zeta _{3}^2}{3^{4}}+\frac{311\pi ^{5} i  \zeta _{3}}{2^{3}\cdot3^{7}\cdot5}-\frac{59\cdot139\cdot402763\pi ^{8}}{2^{7}\cdot3^{11}\cdot5^{2}\cdot7\cdot19\cdot43}$\\
$\mu _{3,1,1,2,1}$&$-\frac{7\cdot11\mu _{8}}{43}-\frac{2\cdot251\zeta _{5,3}}{3^{4}\cdot5}-\frac{2^{3} \mu _{6,1,1}}{3}-\frac{1427\pi  i  \zeta _{7}}{2^{3}\cdot3^{5}}+\frac{353\pi  i  \mu _{6,1}}{3^{2}\cdot19}+\frac{2^{3} \pi  i  \mu _{5,2}}{19}+\frac{ \pi  i  \mu _{4,1,2}}{3}-\frac{2\cdot1511\pi ^{2} \mu _{6}}{3^{3}\cdot11\cdot19}-\frac{ \pi ^{2} \mu _{4,2}}{2\cdot3^{2}}-\frac{379\zeta _{5} \zeta _{3}}{3^{4}}+\frac{179\pi ^{3} i  \zeta _{5}}{2^{2}\cdot3^{5}}+\frac{2\zeta _{3} \mu _{4,1}}{3}-\frac{ \pi ^{3} i  \mu _{4,1}}{2\cdot3^{4}}-\frac{3\mu _{4}^2}{2}+\frac{7\pi  i  \zeta _{3} \mu _{4}}{3^{2}}+\frac{139\pi ^{4} \mu _{4}}{2^{3}\cdot3^{5}\cdot5}+\frac{ \pi ^{2} \zeta _{3}^2}{3^{4}}+\frac{7\pi ^{5} i  \zeta _{3}}{2^{3}\cdot3^{5}}+\frac{70603121\pi ^{8}}{2^{5}\cdot3^{9}\cdot5^{3}\cdot19\cdot43}$\\
$\mu _{3,1,1,1,2}$&$\frac{733\mu _{8}}{43}+2\mu _{6,2}+\frac{2^{2} \zeta _{5,3}}{3^{2}\cdot5}+\frac{5\mu _{6,1,1}}{3}+\frac{8369\pi  i  \zeta _{7}}{2^{3}\cdot3^{6}}-\frac{11\pi  i  \mu _{6,1}}{19}+\frac{2^{2} \pi  i  \mu _{5,2}}{19}-\frac{ \pi  i  \mu _{4,1,2}}{3}+\frac{449\pi ^{2} \mu _{6}}{2\cdot3^{3}\cdot19}+\frac{ \pi ^{2} \mu _{4,2}}{3^{2}}+\frac{41\zeta _{5} \zeta _{3}}{3^{3}}-\frac{7\cdot29\pi ^{3} i  \zeta _{5}}{2\cdot3^{5}}-\frac{2^{2} \zeta _{3} \mu _{4,1}}{3}+\frac{ \pi ^{3} i  \mu _{4,1}}{3^{4}}+\frac{3\mu _{4}^2}{2}-\frac{7\pi  i  \zeta _{3} \mu _{4}}{3^{2}}-\frac{139\pi ^{4} \mu _{4}}{2^{3}\cdot3^{5}\cdot5}-\frac{ \pi ^{2} \zeta _{3}^2}{3^{3}}-\frac{199\pi ^{5} i  \zeta _{3}}{2^{2}\cdot3^{7}}-\frac{29\cdot85587521\pi ^{8}}{2^{7}\cdot3^{9}\cdot5^{3}\cdot7\cdot19\cdot43}$\\
$\mu _{3,1,1,1,1,1}$&$3\mu _{8}+\frac{7\zeta _{5,3}}{2\cdot3^{3}\cdot5}-\frac{4481\pi  i  \zeta _{7}}{2^{4}\cdot3^{5}}+\frac{ \pi  i  \mu _{6,1}}{3}-\frac{29\zeta _{5} \zeta _{3}}{2\cdot3^{3}}+\frac{61\pi ^{3} i  \zeta _{5}}{2^{2}\cdot3^{6}}-\frac{ \pi ^{3} i  \mu _{4,1}}{2\cdot3^{4}}-\frac{\pi ^{4} \mu _{4}}{2^{3}\cdot3^{4}}-\frac{ \pi ^{2} \zeta _{3}^2}{2^{2}\cdot3^{3}}+\frac{11^{2} \pi ^{5} i  \zeta _{3}}{2^{3}\cdot3^{7}\cdot5}-\frac{9001\pi ^{8}}{2^{3}\cdot3^{9}\cdot5^{3}\cdot7}$\\
$\mu _{2,2,2,1,1}$&$11^{2}\mu _{8}+11\mu _{6,2}+\frac{47\zeta _{5,3}}{3^{4}}+\mu _{4,2,2}-\frac{127\cdot503\pi  i  \zeta _{7}}{2\cdot3^{7}}-\frac{13\cdot29\pi  i  \mu _{6,1}}{3^{3}\cdot19}-\frac{41\pi  i  \mu _{5,2}}{3\cdot19}+\frac{7\pi  i  \mu _{4,1,2}}{3^{2}}-11\mu _{6} \mu _{2}-\frac{2^{2}\cdot5^{2}\cdot7\cdot31\pi ^{2} \mu _{6}}{3^{4}\cdot11\cdot19}-\mu _{2} \mu _{4,2}-\frac{7\pi ^{2} \mu _{4,2}}{3^{3}}-\frac{5^{2}\cdot31\zeta _{3} \zeta _{5}}{3^{4}}+\frac{7\cdot11\pi  i  \zeta _{5} \mu _{2}}{3^{3}}-\frac{2^{2}\cdot199\pi ^{3} i  \zeta _{5}}{3^{6}}-\frac{7\pi  i  \mu _{2} \mu _{4,1}}{3^{2}}-\frac{2^{2} \pi ^{3} i  \mu _{4,1}}{3^{5}}+\frac{ \mu _{4} \mu _{2}^2}{2}+\frac{2\cdot11\pi  i  \zeta _{3} \mu _{4}}{3^{3}}-\frac{2^{2} \pi ^{2} \mu _{4} \mu _{2}}{3^{3}}-\frac{\pi ^{4} \mu _{4}}{3^{5}\cdot5}+\zeta _{3}^2\mu _{2}-\frac{ \pi  i  \zeta _{3} \mu _{2}^2}{2\cdot3^{2}}-\frac{5\pi ^{2} \zeta _{3}^2}{3^{3}}+\frac{2\cdot5\pi ^{3} i  \zeta _{3} \mu _{2}}{3^{5}}-\frac{7\cdot241\pi ^{5} i  \zeta _{3}}{2^{2}\cdot3^{7}\cdot5}-\frac{ \pi ^{2} \mu _{2}^3}{2^{2}\cdot3^{3}}-\frac{11\pi ^{4} \mu _{2}^2}{2^{4}\cdot3^{4}\cdot5}+\frac{17851\pi ^{6} \mu _{2}}{2^{4}\cdot3^{8}\cdot5\cdot7}-\frac{11\cdot31\cdot857957\pi ^{8}}{2^{7}\cdot3^{11}\cdot5^{2}\cdot7\cdot19}$\\
$\mu _{2,2,1,2,1}$&$-\frac{7\cdot19\cdot293\mu _{8}}{3\cdot43}-3\cdot7\mu _{6,2}-\frac{509\zeta _{5,3}}{3^{5}}+\frac{5^{2}\cdot7\mu _{6,1,1}}{3^{2}}-3\mu _{4,2,2}+\frac{11\cdot67\cdot691\pi  i  \zeta _{7}}{2^{3}\cdot3^{7}}+\frac{2251\pi  i  \mu _{6,1}}{3^{3}\cdot19}+\frac{2^{2}\cdot31\pi  i  \mu _{5,2}}{3\cdot19}-\frac{2\cdot5\pi  i  \mu _{4,1,2}}{3^{2}}+3\cdot7\mu _{6} \mu _{2}+\frac{23\cdot139\pi ^{2} \mu _{6}}{2\cdot3^{2}\cdot11\cdot19}+3\mu _{2} \mu _{4,2}+\frac{7\pi ^{2} \mu _{4,2}}{3^{3}}+\frac{11\cdot673\zeta _{3} \zeta _{5}}{3^{5}}-\frac{2\cdot17\pi  i  \zeta _{5} \mu _{2}}{3^{2}}+\frac{2^{2}\cdot1861\pi ^{3} i  \zeta _{5}}{3^{7}}-\frac{2\cdot7\zeta _{3} \mu _{4,1}}{3^{2}}+\frac{ \pi  i  \mu _{2} \mu _{4,1}}{3}+\frac{7\pi ^{3} i  \mu _{4,1}}{2\cdot3^{5}}-\frac{3\mu _{4} \mu _{2}^2}{2}-\frac{2^{2}\cdot5\pi  i  \zeta _{3} \mu _{4}}{3^{3}}+\frac{ \pi ^{2} \mu _{4} \mu _{2}}{3^{2}}+\frac{13\pi ^{4} \mu _{4}}{2\cdot3^{6}\cdot5}-\frac{7\zeta _{3}^2\mu _{2}}{3^{2}}+\frac{5\pi  i  \zeta _{3} \mu _{2}^2}{2\cdot3^{2}}+\frac{2\pi ^{2} \zeta _{3}^2}{3^{4}}-\frac{2\cdot23\pi ^{3} i  \zeta _{3} \mu _{2}}{3^{5}}+\frac{4937\pi ^{5} i  \zeta _{3}}{3^{8}\cdot5}+\frac{ \pi ^{2} \mu _{2}^3}{2^{2}\cdot3^{3}}+\frac{7\cdot17\pi ^{4} \mu _{2}^2}{2^{4}\cdot3^{5}\cdot5}-\frac{37591\pi ^{6} \mu _{2}}{2^{4}\cdot3^{8}\cdot5\cdot7}+\frac{4561\cdot1980073\pi ^{8}}{2^{5}\cdot3^{11}\cdot5^{2}\cdot7\cdot19\cdot43}$\\
$\mu _{2,2,1,1,1,1}$&$-2\cdot11\mu _{8}-\mu _{6,2}-\frac{2\zeta _{5,3}}{5}+\frac{97\cdot113\pi  i  \zeta _{7}}{2^{3}\cdot3^{5}}+\frac{ \pi  i  \mu _{5,2}}{3}+\mu _{6} \mu _{2}+\frac{\pi ^{2} \mu _{6}}{2\cdot3^{2}}+\frac{ \pi ^{2} \mu _{4,2}}{2\cdot3^{2}}+2\zeta _{3} \zeta _{5}-\frac{5^{2} \pi  i  \zeta _{5} \mu _{2}}{2\cdot3^{4}}+\frac{5\cdot17\pi ^{3} i  \zeta _{5}}{2\cdot3^{7}}+\frac{7\pi ^{3} i  \mu _{4,1}}{2\cdot3^{5}}-\frac{ \pi ^{2} \mu _{4} \mu _{2}}{2\cdot3^{2}}+\frac{11\pi ^{4} \mu _{4}}{2^{2}\cdot3^{6}}+\frac{ \pi ^{3} i  \zeta _{3} \mu _{2}}{2\cdot3^{5}}-\frac{2^{2} \pi ^{5} i  \zeta _{3}}{3^{7}\cdot5}+\frac{\pi ^{4} \mu _{2}^2}{2^{4}\cdot3^{5}}-\frac{17\cdot19\pi ^{6} \mu _{2}}{2^{4}\cdot3^{7}\cdot5\cdot7}+\frac{101\cdot173\cdot967\pi ^{8}}{2^{7}\cdot3^{11}\cdot5^{3}\cdot7}$\\
$\mu _{2,1,2,1,1,1}$&$2^{3}\cdot3^{2}\mu _{8}+5\mu _{6,2}+\frac{29^{2} \zeta _{5,3}}{2\cdot3^{4}\cdot5}-\frac{367\cdot409\pi  i  \zeta _{7}}{2^{4}\cdot3^{6}}-\frac{7\cdot47\pi  i  \mu _{6,1}}{3^{2}\cdot19}-\frac{2\cdot29\pi  i  \mu _{5,2}}{3\cdot19}-5\mu _{2} \mu _{6}-\frac{5\cdot23\pi ^{2} \mu _{6}}{2\cdot3^{3}\cdot19}-\frac{ \pi ^{2} \mu _{4,2}}{2\cdot3}-\frac{19\cdot41\zeta _{5} \zeta _{3}}{2\cdot3^{4}}+\frac{7\cdot11\pi  i  \zeta _{5} \mu _{2}}{3^{4}}-\frac{7^{2}\cdot13\pi ^{3} i  \zeta _{5}}{2\cdot3^{6}}-\frac{ \pi ^{3} i  \mu _{4,1}}{2\cdot3^{4}}+\frac{ \pi ^{2} \mu _{2} \mu _{4}}{2\cdot3}-\frac{\pi ^{4} \mu _{4}}{2\cdot3^{5}}+\frac{17\pi ^{2} \zeta _{3}^2}{2^{2}\cdot3^{4}}-\frac{ \pi ^{3} i  \zeta _{3} \mu _{2}}{3^{5}}-\frac{1609\pi ^{5} i  \zeta _{3}}{2^{3}\cdot3^{7}\cdot5}-\frac{\pi ^{4} \mu _{2}^2}{2^{4}\cdot3^{5}}+\frac{11\cdot199\pi ^{6} \mu _{2}}{2^{4}\cdot3^{7}\cdot5\cdot7}-\frac{587\cdot680881\pi ^{8}}{2^{7}\cdot3^{10}\cdot5^{3}\cdot7\cdot19}$\\
$\mu _{2,1,1,1,1,1,1}$&$\mu _{8}-\frac{7^{2}\cdot13\pi  i  \zeta _{7}}{2^{4}\cdot3^{5}}-\frac{\pi ^{2} \mu _{6}}{2\cdot3^{2}}+\frac{5^{2} \pi ^{3} i  \zeta _{5}}{2^{2}\cdot3^{7}}+\frac{\pi ^{4} \mu _{4}}{2^{3}\cdot3^{5}}-\frac{\pi ^{5} i  \zeta _{3}}{2^{3}\cdot3^{7}\cdot5}-\frac{\pi ^{6} \mu _{2}}{2^{4}\cdot3^{8}\cdot5}-\frac{11\cdot1061\pi ^{8}}{2^{7}\cdot3^{9}\cdot5^{2}\cdot7}$\\$\RE\mu _8$&$\frac{127\cdot1093\pi ^8}{2^{8}\cdot3^{10}\cdot5^{2}\cdot7}$\\
$\RE\mu _{6,1,1}$&$-\frac{7\zeta _{5,3}}{2^{2}\cdot3^{3}}-\frac{ \pi  \I\mu _{6,1}}{3}-\frac{179\zeta _{5} \zeta _{3}}{2^{2}\cdot3^{3}}+\frac{ \pi ^2\zeta _{3}^2}{2^{3}\cdot3^{2}}+\frac{65629\pi ^8}{2^{5}\cdot3^{10}\cdot5^{2}\cdot7}$\\
$\RE\mu _{4,2,2}$&$-\frac{11\RE\mu _{6,2}}{2}-\frac{2\cdot241\zeta _{5,3}}{3^{4}\cdot5}-\frac{2\cdot7\cdot47\pi  \I\mu _{6,1}}{3^{2}\cdot19}-\frac{2^{3}\cdot5\pi  \I\mu _{5,2}}{3\cdot19}+\frac{ \pi ^2\RE\mu _{4,2}}{2^{2}\cdot3^{2}}-\frac{2\cdot29\zeta _{3} \zeta _{5}}{3^{3}}+\I\mu _{4}{}^2-\frac{2^{2} \pi  \zeta _{3} \I\mu _{4}}{3^{2}}+\frac{ \pi ^2\zeta _{3}^2}{2\cdot3^{2}}+\frac{1471\cdot15919\pi ^8}{2^{9}\cdot3^{11}\cdot5^{3}\cdot7}$\\$\I\mu _{6,2}$&$-11\I\mu _{8}+\frac{659\pi  \zeta _{7}}{2^{3}\cdot3^{4}}+\frac{ \pi ^2\I\mu _{6}}{2^{2}\cdot3^{2}}+\frac{2\cdot5\pi ^3\zeta _{5}}{3^{4}}+\frac{17\pi ^5\zeta _{3}}{2^{2}\cdot3^{6}}$ \\\hline\hline \end{tabular}}\end{table}
\end{enumerate}
\end{proposition}\begin{proof}\begin{enumerate}[leftmargin=*,  label=(\alph*),ref=(\alph*),
widest=d, align=left] \item
At weight $7$, the linear system       $ \{\Lyn\fib_{z=\varrho} g_{\bm \alpha}(z)=0\}_{\bm \alpha\in\mathfrak L_{7}}$ has rank $12$. To deduce the top 15 entries in Table \ref{tab:MCV7}, we need the assistance from \begin{align}
\vfib_{z=\varrho}\left[ G(\boldsymbol0_{a},z,\boldsymbol0_{5-a},z;1)-G\left( \boldsymbol0_{a}, z,\boldsymbol0_{5-a},1-\frac{1}{z};1\right) \right]=0,
\end{align}  where $a\in\{2,3,4\}$. The remaining entries  are natural consequences of       $ \{\Lyn\fib_{z=\varrho} g_{\bm \alpha}^{R}(z)=0\}_{\bm \alpha\in\mathfrak L_7}$    and      $ \{\Lyn\fib_{z=\varrho} g_{\bm \alpha}^{I}(z)=0\}_{\bm \alpha\in\mathfrak L_7}$. 
\item At weight $8$, the linear system       $ \{\Lyn\fib_{z=\varrho} g_{\bm \alpha}(z)=0\}_{\bm \alpha\in\mathfrak L_{8}}$ has rank $20$.  To arrive at the top $ 26$ entries of Table \ref{tab:MCV8}, we need  to enlist the help from \begin{align}
\vfib_{z=\varrho}\left[ G(\boldsymbol0_{a},z,\boldsymbol0_{6-a},z;1)-G\left( \boldsymbol0_{a}, z,\boldsymbol0_{6-a},1-\frac{1}{z};1\right) \right]=0,
\end{align} where $a\in\{2,3,4,5\}$, as well as \begin{align}
\vfib_{z=\varrho}\left[ G(0,0,z,0,z,0,0,z;1)-G\left( 0,0, z,0,z,0,0,1-\frac{1}{z};1\right) \right]=0
\end{align}and \begin{align}
\vfib_{z=\varrho}\left[ G(0,0,z,0,z,0,0,z;1)-G\left( 0,0, z,0,1-\frac{1}{z},0,0,1-\frac{1}{z};1\right) \right]=0.
\end{align}The rest follow from       $ \{\Lyn\fib_{z=\varrho} g_{\bm \alpha}^{R}(z)=0\}_{\bm \alpha\in\mathfrak L_8}$    and      $ \{\Lyn\fib_{z=\varrho} g_{\bm \alpha}^{I}(z)=0\}_{\bm \alpha\in\mathfrak L_8}$.\qedhere\end{enumerate}
\end{proof}\begin{remark}Thus far, we may verify all the entries of Tables \ref{tab:A1n} and \ref{tab:Asn}, completely reducing the right-hand sides of \eqref{eq:A1n_Z(6)} and \eqref{eq:Asn_Z(6)} via the algebraic relations proved in Propositions \ref{prop:Lyn56} and \ref{prop:Lyn78}. Via the same approach (cf.\  the more cost-effective method in \S\ref{subsubsec:desc1}), we may also confirm Table  \ref{tab:A2nA4n} up to weight $8$. As a by-product,  Table \ref{tab:Au1.5b} grows out of Lyndon decompositions for the MCV summands in  $ \mathsf A_{w,n}$ [see \eqref{eq:Au1.5b}] and the reductions of Lyndon MCVs in Tables \ref{tab:MCV2345}--\ref{tab:MCV8}.
\eor\end{remark}

\begin{remark}In \cite[Theorem 3.3]{BorweinBroadhurstKamnitzer2001}, Borwein--Broadhurst--Kamnitzer stated \begin{align}
\mathscr A_8=-\frac{2\cdot7\zeta _{5,3}}{3\cdot5}-2^{2}\pi  \I\mu _{6,1}-\frac{2\cdot19\zeta _{5} \zeta _{3}}{3}+\frac{ \pi ^2\zeta _{3}^2}{3^{2}}+\frac{691\cdot5011\pi ^8}{2^{7}\cdot3^{9}\cdot5^{3}\cdot7},\label{eq:A8BBK}
\end{align}without elaborating on its analytic proof. Evaluating the right-hand side of \eqref{eq:AsLs} by the  
\texttt{logsine} package of Borwein--Straub  \cite{BorweinStraub2011ISSAC}, we get\begin{align}
\mathscr A_8=2^{2}\cdot3^{2}\RE\mu _{7,1}-2^{2}\pi  \I\mu _{6,1}+2^{2}\zeta _{5} \zeta _{3}+\frac{ \pi ^2\zeta _{3}^2}{3^{2}}-\frac{266677\pi ^8}{2^{7}\cdot3^{9}\cdot5^{2}\cdot7}.\label{eq:A8BS}
\end{align} In view of  $\RE\mu _8=\frac{127\cdot1093\pi ^8}{2^{8}\cdot3^{10}\cdot5^{2}\cdot7} $, the real part of the top entry in  Table \ref{tab:MCV8} simplifies to \begin{align}
\RE\mu _{7,1}={}&-\frac{7 \zeta _{5,3}}{2\cdot3^{3}\cdot5}-\frac{5^{2} \zeta_{5} \zeta_{3}}{2\cdot3^{3}}+\frac{13\cdot61487 \pi ^8}{2^{8}\cdot3^{10}\cdot5^{3}\cdot7},\label{eq:mu71}
\end{align}which reveals the equivalence between \eqref{eq:A8BBK} and \eqref{eq:A8BS}.  We are unable to reconcile the right-hand sides  of \eqref{eq:A8BBK} and \eqref{eq:A8BS} solely on the basis  of \cite[Theorems 4.2 and 4.4]{BorweinBroadhurstKamnitzer2001}. Neither have we located a definitive proof of the reduction for $ \mu_{7,1}$  in previous literature.    \eor\end{remark}
\begin{remark}Broadhurst's  numerical experiments on   certain members in $ \mathfrak Z_w(6)$ (up to  $ w=36$) led him to conjecture that  $ \RE\mu_{w-1,1}\in\mathfrak Z_w(1)$, when $w$ is even \cite[Conjecture 5]{Broadhurst2014MDV}. As the MCVs $ \RE\mu_{w-1,1}\in\mathfrak Z_w(6)$ are (numerically)  MZVs in disguise, Broadhurst referred to them as ``honorary MZVs''. Note that Broadhurst's  (empirical) formula for $ \RE\mu _{7,1}$ (cf.\ \cite[(30)]{Broadhurst2014MDV} or \cite[p.~13]{Broadhurst2015MDV}) is equivalent to \eqref{eq:mu71}.  
\eor\end{remark}

\subsection{Clausen descents\label{subsec:descClausen}}The purpose of this subsection is two-fold. In \S\ref{subsubsec:desc3}, we will upgrade the weaker forms of Theorems \ref{thm:A1n}--\ref{thm:Asn} (proved in \S\ref{subsec:Z(6)}) to their full-fledged forms. In  \S\ref{subsubsec:desc1}, we  will prove Theorem \ref{thm:A2nA4n}, after taking a fresh look at    \eqref{eq:A2nMZV}---a result of Hou--Sun \cite[Theorem 1.3]{HouSun2026}. Both tasks share the same theme: reduction of the levels for ($ \mathbb Q$-linear combinations of) MCVs.\subsubsection{Descents to level 3\label{subsubsec:desc3}}At weight $1$, one has   $ \mu_1=\frac{\pi i}{3}\in \mathfrak Z_1(3)$.  For \ $ 2\leq a+n\leq5$, we have  $ \mu_{a,\mathbf 1_n}\in\mathfrak Z_{a+n}(3)$, in view of the following identities with a third root of unity $ \omega\colonequals  e^{2\pi i/3}$:{\allowdisplaybreaks
\begin{align}
&\left\{\begin{array}{@{}r@{{}={}}l}
\mu_{2,\mathbf 1_0} &\frac{3i\I\Li_2(\omega)}{2}+\frac{\pi^2}{2^{2}\cdot3^2},\\[2pt]
\mu_{1,1_1}&-\frac{\pi^2}{2\cdot3^{2}},
\end{array}\right.\label{eq:mu2}\\&\left\{\begin{array}{@{}r@{{}={}}l}
\mu_{3,\mathbf 1_0} &\frac{\zeta _{3}}{3}+\frac{5\pi ^3 i }{2\cdot3^{4}},\\[2pt]
\mu_{2,\mathbf 1_1}&\frac{2 \zeta _{3}}{3}-\frac{\pi\I\Li_2(\omega) }{2}+\frac{ \pi ^3 i}{2^{2}\cdot3^{4}},\\[2pt]\mu_{1,\mathbf 1_2}&-\frac{\pi ^3 i}{2\cdot3^{4}} ,
\end{array}\right.\\&\left\{\begin{array}{@{}r@{{}={}}l}
\mu_{4,\mathbf 1_0} &\frac{3^{2}i\I\Li_4(\omega)}{2^{3}}+\frac{7\cdot13 \pi ^4}{2^{4}\cdot3^{5}\cdot5},\\
[2pt]\mu_{3,\mathbf 1_1}&\frac{3^{2}i\I\Li_4(\omega)}{2^{3}}-\frac{2\pi i\zeta_{3}}{3^{2}}-\frac{23 \pi ^4}{2^{4}\cdot3^{5}\cdot5},\\[2pt]\mu_{2,\mathbf 1_2}&\frac{3^{2}i\I\Li_4(\omega)}{2^{3}}-\frac{\pi i\zeta_{3}}{3^{2}}-\frac{\pi ^2 i\I\Li_2(\omega)}{2^{2}\cdot3} -\frac{\pi ^4}{2^{4}\cdot3^{5}},\\[2pt]\mu_{1,\mathbf 1_3} &\frac{\pi ^4}{2^{3}\cdot3^{5}},
\end{array}\right.\\&\left\{\begin{array}{@{}r@{{}={}}l}
\mu_{5,\mathbf 1_0} &\frac{5^{2} \zeta _{5}}{2\cdot3^{3}}+\frac{17  \pi ^5i}{2^{3}\cdot3^{6}},\\
[2pt]\mu_{4,\mathbf 1_0}&\frac{43 \zeta_{5}}{2^{2}\cdot3^2}+\frac{3^{3}i\I[3\Li_{4,1}(\omega,1)+\Li_{3,2}(\omega,1)]}{2\cdot13}-\frac{3\pi\I\Li_4(\omega)}{2^{3}}-\frac{\pi ^2 \zeta _{3}}{2^{2}\cdot3^2}+\frac{151 \pi ^5i}{2^{4}\cdot3^{6}\cdot13},\\[2pt]\mu_{3,\mathbf 1_2}& \frac{29 \zeta_{5}}{2^{2}\cdot3^2}+\frac{3^{3}i\I[3\Li_{4,1}(\omega,1)+\Li_{3,2}(\omega,1)]}{2\cdot13}-\frac{3\pi\I\Li_4(\omega)}{2^{3}}+\frac{\pi ^2 \zeta _{3}}{2^{2}\cdot3^3}+\frac{599  \pi ^5i}{2^{4}\cdot3^{6}\cdot5\cdot13},\\[2pt]\mu_{2,\mathbf 1_3} &\frac{29 \zeta _{5}}{2\cdot3^{3}}-\frac{3\pi\I\Li_4(\omega)}{2^{3}}+\frac{\pi ^2 \zeta _{3}}{2\cdot3^{3}}+\frac{\pi^{3}\I\Li_2(\omega)}{2^{2}\cdot3^3}-\frac{ \pi ^5i}{2^{4}\cdot3^{6}\cdot5},\\[2pt]\mu_{1,\mathbf 1_4}&\frac{\pi ^5i}{2^{3}\cdot3^{6}\cdot5},
\end{array}\right.\label{eq:mu5}
\end{align}
}which are all verifiable by Au's \texttt{MultipleZetaValues} package \cite{Au2025a,Au2022a}. Here, for an arbitrary positive integer  $a$, one may sharpen $ \mu_a\in \mathfrak Z_{a}(6)$ into $ \mu_a\in \mathfrak Z_{a}(3)$  via \begin{align}
\I \mu_a={}&\left( 1+\frac{1}{2^{a-1}} \right)\I\Li_a(\omega)=\left( 1+\frac{1}{2^{a-1}} \right)\frac{\sqrt{3}L(\chi_{-3},a)}{2}\in i \mathfrak Z_a(3)\label{eq:Imu_a}
\intertext{and}
\RE \mu_a={}&-\left( 1-\frac{1}{2^{a-1}} \right)\RE\Li_a(\omega)=\left( 1-\frac{1}{2^{a-1}} \right)\left( 1-\frac{1}{3^{a-1}} \right)\frac{\zeta_a}{2}\in\mathfrak Z_a(1),\label{eq:Rmu_a}  
\end{align}where
\begin{align}
L(\chi_{-3},s)\colonequals \sum_{n=0}^\infty\left[\frac{1}{(3n+1)^s}-\frac{1}{(3n+2)^s}\right]
\end{align}
is a special Dirichlet $L$-function. In particular, when $a=1$, both  $ \RE\mu_a$ and $1-\frac{1}{2^{a-1}} $ vanish.

More generally, each MCV  descends  to a $ \mathbb Z$-linear combination of level-$3$ CMZVs,
as demonstrated in the proposition below.\begin{proposition}\label{prop:Cd}
For any positive  integers $ a_1,\dots,a_n$, we have \begin{align}
\begin{split}\mu_{a_1,\dots,a_n}={}&(-1)^{n}G\big(\boldsymbol0_{a_{1}-1},1,\boldsymbol0_{a_{2}-1},1,\dots,\boldsymbol0_{a_{n}-1},1;e^{\pi i/3}\big)\\\in{}&\Span_\mathbb Z\left\{ G(\alpha_{1},\dots,\alpha_{m};1)G(\beta_{1},\dots,\beta_{w-m};\omega)\middle|\begin{smallmatrix}w=a_1+\dots+a_n,m\in\mathbb Z\cap[0,w]\\\alpha_1,\dots,\alpha_m\in\left\{\omega,\omega^2\right\}\\\beta_{1},\dots,\beta_{w-m}\in\{0,1\}\end{smallmatrix}\right\}\\\subseteq{}&\mathfrak Z_{a_1+\dots+a_n}(3)
,
\end{split}\label{eq:MCVinCMZV3}\end{align} where $ \omega\colonequals  e^{2\pi i/3}$. 
\end{proposition}\begin{proof}
We write  an MCV as\begin{align}
{\mu_{a_1,\dots,a_n}}={}&\left.\!\!(-1)^{n}G\bigg(\boldsymbol0_{a_{1}-1},1,\boldsymbol0_{a_{2}-1},1,\dots,\boldsymbol0_{a_{n}-1},1;\frac{z-\omega^{2}}{z-\omega}\bigg)\right|_{z=1}.
\end{align} Noting that \begin{align}
-\frac{\partial}{\partial z}G\bigg( 0,\beta_2,\dots,\beta_w;\frac{z-\omega^{2}}{z-\omega} \bigg)={}&G\bigg( \beta_2,\dots,\beta_w;\frac{z-\omega^{2}}{z-\omega} \bigg)\left[\frac{1}{z-\omega}-\frac{1}{z-\omega^{2}}\right],\\-\frac{\partial}{\partial z}G\bigg( 1,\beta_2,\dots,\beta _{w};\frac{z-\omega^{2}}{z-\omega} \bigg)={}&G\bigg( \beta_2,\dots,\beta_w;\frac{z-\omega^{2}}{z-\omega} \bigg)\frac{1}{z-\omega} ,
\end{align} we may build  \begin{align}\begin{split}&
\left.\!\!G\bigg(\boldsymbol0_{a_{1}-1},1,\boldsymbol0_{a_{2}-1},1,\dots,\boldsymbol0_{a_{n}-1},1;\frac{z-\omega^{2}}{z-\omega}\bigg)\right|_{z=1}
\\\in {}&\Span_\mathbb Z\left\{ G(\alpha_{1},\dots,\alpha_{m};1)G(\beta_{1},\dots,\beta_{w-m};\omega)\middle|\begin{smallmatrix}w=a_1+\dots+a_n,m\in\mathbb Z\cap[0,w]\\\alpha_1,\dots,\alpha_m\in\left\{\omega,\omega^2\right\}\\\beta_{1},\dots,\beta_{w-m}\in\{0,1\}\end{smallmatrix}\right\}\end{split}\label{eq:MCVdescent_prep}\end{align}inductively on the GPL recursion in \eqref{eq:GPL_rec} and the initial condition\begin{align}
G\left( 1;\frac{z-\omega^{2}}{z-\omega} \right)=G(1;\omega)-G(\omega;z).
\end{align} Therefore, each MCV is characterized by \eqref{eq:MCVinCMZV3}.
\end{proof}

At this stage, we have embedded every MCV of weight $w$ into $ \mathfrak Z_w(3)$,\footnote{This explicit embedding also helps us understand  conceptually the (relatively painless) generations of algebraic relations for MCVs (see Tables \ref{tab:MCV2345}--\ref{tab:MCV8} and Propositions \ref{prop:Lyn56}--\ref{prop:Lyn78}) in an automated manner.  The algebraic relations for  CMZVs of level $N$ are ``standard'' (a notion introduced by Deligne in a letter to Goncharov and Racinet \cite[Footnote 1]{Zhao2010}), when  $N$ is equal to $ 1$, $2$, $3$, or an integer power of a prime $ p\geq5$. For other levels, ``non-standard'' $ \mathbb Q$-linear relations may lurk somewhere, which require improvised contour integrals (cf.\ \cite{Zhao2008} and \cite[\S5.1]{Au2022a}) to prove them. During our proofs of  Propositions \ref{prop:Lyn56} and \ref{prop:Lyn78}, we did not need to worry about ``non-standard'' relations of level $6$,  because all our targets actually resided at a ``standard'' level  $N=3$. } thereby establishing both Theorems \ref{thm:A1n} and \ref{thm:Asn} in full strength, for generic positive integers $s$. As we will  see very soon, the subtle symmetries for $s=2$ and $s=4$ provide us with further descents of the deformed Ap\'ery-like series into the $ \mathbb Q$-vector spaces of  MZVs.

\subsubsection{Descents to level 1\label{subsubsec:desc1}}We start by revisiting  \eqref{eq:A2nMZV} from the perspective of GPL fibrations, independent of the WZ-based approach of   Hou--Sun \cite[Theorem 1.3]{HouSun2026}.\begin{proposition}\label{prop:A2nMZV}For $ \I z\neq0$ and $ n\in\mathbb Z_{>0}$, we have \begin{align}
\begin{split}\mathscr G_{2,n}(z)\colonequals {}&G(z,0,\boldsymbol 1_n;1-z(1-z))-G(1,0,\boldsymbol 1_n;1-z(1-z))\\{}&{}+\sum_{\alpha_i\in\{0,1\}}\left[ 2G(0,\varrho,\alpha_1,\dots,\alpha_n ;z )+2G\bigg(0,\frac{1}{\varrho},\alpha_1,\dots,\alpha_n ;z \bigg)\right.\\&\left.{}-G(1,\varrho,\alpha_1,\dots,\alpha_n ;z )-G\bigg(1,\frac{1}{\varrho},\alpha_1,\dots,\alpha_n ;z \bigg) \right]\\\in{}&\Span_{\mathbb Q}\left\{ \left(\frac{\pi i \I z}{|\I z|}\right)^\ell Z_m G(\beta_{1},\dots,\beta_{n+2-\ell-m};z)\middle|\begin{smallmatrix}\ell,m,n+2-\ell-m\in\mathbb Z_{\geq0}\\Z_m\in\mathfrak Z_m(1)\\\beta_1,\dots,\beta_{n+2-\ell-m}\in\{0,1\}\end{smallmatrix}\right\}\equalscolon\mathfrak f_{n+2}^{(z)}.\label{eq:G2nF2n}
\end{split}
\end{align}   Consequently, we have $ \mathscr A_{2,n}\in\mathfrak Z_{n+2}(1)$, as declared in \eqref{eq:A2nMZV}.
\end{proposition}\begin{proof}By a variation on the arguments behind \eqref{eq:Asn_int_sum}, we get \begin{align}
\begin{split}\mathscr G_{2,n}(z)= {}&G(z,0,\boldsymbol 1_n;1-z(1-z))-G(1,0,\boldsymbol 1_n;1-z(1-z))\\&{}+\int_0^z \left( \frac{2}{t} -\frac{1}{t-1}\right)G(1,\boldsymbol0_{n};t(1-t))\D t
\end{split}
\label{eq:G2n_int}\end{align}and \begin{align}
\mathscr A_{2,n}=-\frac{n!}{3}\int_0^1 \left( \frac{2}{t} -\frac{1}{t-1}\right)G(1,\boldsymbol0_{n};t(1-t))\D t.
\end{align}It is clear that
\begin{align}
\lim_{z\to0}[\mathscr G_{2,n}(z)+G(1,0,\boldsymbol 1_n;1-z(1-z))]={}&G(0,0,\boldsymbol 1_n;1)\in\mathfrak Z_{n+2}(1).
\end{align}Moreover, similar to our derivation  of \eqref{eq:MCVdescent_prep}, we may build \begin{align}
\begin{split}&G(1,0,\boldsymbol 1_n;1-z(1-z))\\\in{}& \Span_\mathbb Z\left\{ Z_{m}G(\beta_{1},\dots,\beta_{n+2-m};z(1-z))\middle|\begin{smallmatrix}m,n+2-m\in\mathbb Z_{\geq0}\\Z_{m}\in\mathfrak Z_m(1)\\\beta_{1},\dots,\beta_{n+2-m}\in\{0,1\}\end{smallmatrix}\right\}
\end{split}
\end{align}inductively on GPL recursions.
 Therefore,
one can explicitly construct a function  $\mathring{\mathscr G}_{2,n}(z)\in \mathfrak f_{n+2}^{(z)} $ such that $ \lim_{z\to0}[\mathscr G_{2,n}(z)-\mathring{\mathscr G}_{2,n}(z)]=0$.
Next, we will prove $\mathscr G_{2,n}(z)\in\mathfrak  f_{n+2}^{(z)} $ by GPL recursions and induction on $n$.

In view of \eqref{eq:GPL_diff_form} and \eqref{eq:G2n_int}, we may put down\begin{align}\begin{split}
\frac{\partial\mathscr G_{2,n}(z)}{\partial z}={}&\left( \frac{2}{z} -\frac{1}{z-1}\right)[G(1,\boldsymbol0_{n};z(1-z))-G(0,\boldsymbol1_{n};1-z(1-z))]\\{}&+\frac{G(z,\boldsymbol1_{n};1-z(1-z))}{z}.\label{eq:dG2n(z)}
\end{split}
\end{align}Meanwhile, we have \begin{align}
\begin{split}\frac{\partial}{\partial u}[G(0,\boldsymbol1_{n};1-u)-G(1,\boldsymbol0_{n};u)]\xlongequal{\text{\eqref{eq:GPL_rec}}}{}&\frac{G(\boldsymbol1_{n};1-u)-G(\boldsymbol0_{n};u)}{u-1}\\\xlongequal{\text{\eqref{eq:GPL0}}}{}&\frac{G(\boldsymbol1_{n};1-u)-\frac{\log^n u}{n!}}{u-1}\\\xlongequal{\text{\eqref{eq:ShGPL}}}{}&\frac{\frac{[G(1;1-u)]^{n}}{n!}-\frac{\log^n u}{n!}}{u-1}\xlongequal{\text{\eqref{eq:G(a;t)}}}0,
\end{split}
\end{align}so \begin{align} \begin{split}G(0,\boldsymbol1_{n};1-u)-G(1,\boldsymbol0_{n};u)={}&G(0,\boldsymbol1_{n};1)-G(1,\boldsymbol0_{n};0)\xlongequal{\text{\eqref{eq:ShGPL}}}G(0,\boldsymbol1_{n};1)\\\xlongequal{\text{\cite[\S7]{BorweinBradleyBroadhurstLisonek2001}}}{}&(-1)^{n+1}G(\boldsymbol0_{n},1;1)=(-1)^{n+1}\zeta_{n+1}\label{eq:G011G110}
\end{split}\end{align}simplifies \eqref{eq:dG2n(z)} to \begin{align}
\frac{\partial\mathscr G_{2,n}(z)}{\partial z}=(-1)^{n}\left( \frac{2}{z} -\frac{1}{z-1}\right)\zeta_{n+1}+\frac{G(z,\boldsymbol1_{n};1-z(1-z))}{z}.
\end{align}Here, we may reduce the right-hand side of\begin{align}
\begin{split}G(z,\boldsymbol1_{n};1-z(1-z))=G(0,\boldsymbol1_{n};1)+\int_{0}^z\frac{\partial G(w,\boldsymbol1_{n};1-w(1-w))}{\partial w}\D w
\end{split}
\end{align}by  \eqref{eq:GPL_diff_form} and \eqref{eq:G011G110}, so as to arrive at\begin{align}
\begin{split}&G(z,\boldsymbol1_{n};1-z(1-z))\\={}&(-1)^{n+1}\zeta_{n+1}+\int_{0}^z\frac{ G(\boldsymbol1_{n};1-w(1-w))+ G(w,\boldsymbol1_{n-1};1-w(1-w))}{w-1}\D w\\\xlongequal{\text{\eqref{eq:G(a;t)}}}{}&(-1)^{n+1}\zeta_{n+1}+\int_0^z\frac{ \frac{\log^{n}(w(1-w))}{n!}+ G(w,\boldsymbol1_{n-1};1-w(1-w))}{w-1}\D w.
\end{split}
\end{align} It then follows by induction that $ G(z,\boldsymbol1_{n};1-z(1-z))\in \mathfrak f_{n+1}^{(z)}$ and  $\mathscr G_{2,n}(z)-\mathring{\mathscr G}_{2,n}(z)\in\mathfrak  f_{n+2}^{(z)} $.

Now that \eqref{eq:G2nF2n} is proven, we may study\begin{align}\begin{split}
\mathscr G_{2,n}(z)={}&G(z,0,\boldsymbol 1_n;1-z(1-z))-G(1,0,\boldsymbol 1_n;1-z(1-z))\\{}&-\frac{3\mathscr A_{2,n}}{n!}+\int_1^z \left( \frac{2}{t} -\frac{1}{t-1}\right)G(1,\boldsymbol0_{n};t(1-t))\D t
\end{split}
\end{align}for $ z\to1\pm i0^+$.  By the shuffle relation \eqref{eq:ShGPL}, we have \begin{align}
\begin{split}&G(z,0,\boldsymbol 1_n;1-z(1-z))-G(1,0,\boldsymbol 1_n;1-z(1-z))\\={}&[G(z;1-z(1-z))-G(1;1-z(1-z))]G(0,\boldsymbol 1_n;1-z(1-z))+O(z-1),
\end{split}
\end{align}where \begin{align}
G(z;1-z(1-z))-G(1;1-z(1-z))\xlongequal{\text{\eqref{eq:G(a;t)}}} \frac{\pi i \I z}{|\I z|}+G(1;z)-2G(0;z)
\end{align}We take regularized limits (as defined in \cite[(2.8)]{Panzer2015})  that suppress all the positive integer powers of   logarithmic divergences $ G(1;z)$ as  $ z\to1\pm i0^+$, so that\footnote{In practice,  the MZV representation for $ \Reg_{z\to1+i0^+}f(z)$ is implemented in Panzer's  \texttt{HyperInt}  package \cite{Panzer2015} for \texttt{Maple}, via the following code:\begin{quotation}\texttt{fibrationBasis(eval(regHlog(eval(fibrationBasis(f(z), [z]), [z=1])), [Hlog(1, [1])=0, delta[1]=1]));}\end{quotation}As for $   \Reg_{z\to1-i0^+}f(z)$, simply replace \texttt{delta[1]=1} by \texttt{delta[1]=-1}. } \begin{align}
\mathscr A_{2,n}=-\frac{n!}{3}\RE\Reg_{z\to1\pm i0^+}\mathscr G_{2,n}(z)\equiv-\frac{n!}{6}\left[ \Reg_{z\to1+ i0^+}\mathscr G_{2,n}(z) +\Reg_{z\to1- i0^+}\mathscr G_{2,n}(z)\right]
\end{align}is explicitly a $ \mathbb Q$-linear combination of MZVs at weight $n+2$.
\end{proof} In the proposition above, the descent of each $ \mathscr A_{2,n}$ to the MZV space $ \mathfrak Z_{n+2}(1)$ was facilitated by  
a corresponding counter-term\begin{align}\mathscr C_{2,n}(z)\colonequals 
G(z,0,\boldsymbol 1_n;1-z(1-z))-G(1,0,\boldsymbol 1_n;1-z(1-z)).
\end{align}In the regularized limits, such a  counter-term serves as zero-padding:\begin{align}
\RE\Reg_{z\to1\pm i0^+}\mathscr C_{2,n}(z)=0.
\end{align}In the next proposition, we will need more sophisticated counter-terms for $ s=4$.
\begin{proposition}
\label{prop:A4nMZV}For each  $ n\in\mathbb Z_{>0}$,  define a counter-term\begin{align}
\begin{split}\mathscr C_{4,n}(z)={}&\frac13[7G(z,1,1,0,\boldsymbol 1_n;1-z(1-z))-2G(z,z,1,0,\boldsymbol 1_n;1-z(1-z))\\{}&-G(z,1,z,0,\boldsymbol 1_n;1-z(1-z))-G(1,z,1,0,\boldsymbol 1_n;1-z(1-z))\\{}&-3G(1,1,1,0,\boldsymbol 1_n;1-z(1-z))].
\end{split}
\end{align}
For $ \I z\neq0$ and $ n\in\mathbb Z_{>0}$, we have \begin{align}
\begin{split}\mathscr G_{4,n}(z)\colonequals {}&\mathscr C_{4,n}(z)+\sum_{\alpha_i\in\{0,1\}}\left[ G(0,\alpha_{1},\alpha_{2},\varrho,\alpha_3,\dots,\alpha_{n+2} ;z )+G\bigg(0,\alpha_{1},\alpha_{2},\frac{1}{\varrho},\alpha_3,\dots,\alpha_{n+2} ;z \bigg)\right.\\&\left.{}-2G(1,\alpha_{1},\alpha_{2},\varrho,\alpha_3,\dots,\alpha_{n+2} ;z )-2G\bigg(1,\alpha_{1},\alpha_{2},\frac{1}{\varrho},\alpha_3,\dots,\alpha_{n+2} ;z \bigg) \right]\\\in{}&\Span_{\mathbb Q}\left\{ \left(\frac{\pi i \I z}{|\I z|}\right)^\ell Z_m G(\beta_{1},\dots,\beta_{n+4-\ell-m};z)\middle|\begin{smallmatrix}\ell,m,n+4-\ell-m\in\mathbb Z_{\geq0}\\Z_m\in\mathfrak Z_m(1)\\\beta_1,\dots,\beta_{n+4-\ell-m}\in\{0,1\}\end{smallmatrix}\right\}\equalscolon\mathfrak f_{n+4}^{(z)}\label{eq:G4nF4n}
\end{split}
\end{align} and \begin{align}
\RE\Reg_{z\to1\pm i0^+}\mathscr C_{4,n}(z)\in\mathfrak Z_{n+4}(1).\label{eq:C4nReg}
\end{align}Consequently, we have $ \mathscr A_{4,n}\in\mathfrak Z_{n+4}(1)$, as declared in \eqref{eq:A4nMZV}.
\end{proposition}\begin{proof}Akin to the beginning of our proof for the last proposition, we avail ourselves of\begin{align}\mathscr G_{4,n}(z)\colonequals {}&\mathscr C_{4,n}(z)+\int_0^z \left( \frac{1}{t} -\frac{2}{t-1}\right)G(0,0,1,\boldsymbol0_{n};t(1-t))\D t\end{align}   and\begin{align}
\mathscr A_{4,n}=-\frac{n!}{3}\int_0^1 \left( \frac{1}{t} -\frac{2}{t-1}\right)G(0,0,1,\boldsymbol0_{n};t(1-t))\D t.
\end{align}One can also construct a function  $\mathring{\mathscr G}_{4,n}(z)\in \mathfrak f_{n+4}^{(z)} $ such that $ \lim_{z\to0}[\mathscr G_{4,n}(z)-\mathring{\mathscr G}_{4,n}(z)]=0$.
However, the inductive proof of  $\mathscr G_{4,n}(z)\in\mathfrak  f_{n+4}^{(z)} $ will require more explanations.

Thanks to \eqref{eq:GPL_diff_form}, we have\begin{align}
\begin{split}\frac{\partial\mathscr G_{4,n}(z)}{\partial z}={}&\left( \frac{1}{z} -\frac{2}{z-1}\right)[G(0,0,1,\boldsymbol0_{n};z(1-z))-G(1,1,0,\boldsymbol1_{n};1-z(1-z))]\\{}&-\frac{G(z,1,z,\boldsymbol1_{n};1-z(1-z))}{3z}\\&{}+\frac{2}{3(z-1)}[2G(z,1,0,\boldsymbol1_{n};1-z(1-z))-G(z,z,0,\boldsymbol1_{n};1-z(1-z))\\&{}-G(1,z,0,\boldsymbol1_{n};1-z(1-z))-G(1,1,0,\boldsymbol1_{n};1-z(1-z))].\label{eq:dG4n}
\end{split}
\end{align}Here, we can show that \begin{align}
G(\boldsymbol0_{m},1,\boldsymbol0_{n};u)-G(\boldsymbol1_{m},0,\boldsymbol1_{n};1-u)\in\Span_\mathbb Q\left\{ Z_{m+n+1-\ell}G(\boldsymbol0_{m-\ell};u)\middle|\begin{smallmatrix}\ell ,m,n,{m-\ell}\in\mathbb Z_{\geq0}\\Z_{m+n+1-\ell}\in\mathfrak Z_{m+n+1-\ell}(1)\\\beta_{1},\dots,\beta_{n+2-m}\in\{0,1\}\end{smallmatrix}\right\},
\end{align}after multiplying both sides of \begin{align}
G(0,\boldsymbol1_{n};1-u)=G(1,\boldsymbol0_{n};u)+(-1)^{n+1}\zeta_{n+1}\tag{\ref{eq:G011G110}$'$}
\end{align} by integer powers of $ G(1;1-u)=G(0;u)=\log u$ and appealing to GPL shuffles \eqref{eq:ShGPL} repeatedly.
Thus, we can identify the first line on the right-hand side of \eqref{eq:dG4n} with the derivative of a certain member in $ \mathfrak f_{n+4}^{(z)}$. To move onto the second line, we note that our previous demonstration for  $ G(z,\boldsymbol1_{n};1-z(1-z))\in \mathfrak f_{n+1}^{(z)}$  generalizes to\begin{align}
G(\alpha_{1},\dots,\alpha_m;1-z(1-z))\in \mathfrak f_m^{(z)}
\end{align} whenever $ \alpha_1,\dots,\alpha_m\in\{z,1\}$. This observation also helps us resolve the remaining terms on the right-hand side  of  \eqref{eq:dG4n}, which yield \begin{align}
\begin{split}&\frac{\partial}{\partial z}[2G(z,1,0,\boldsymbol1_{n};1-z(1-z))-G(z,z,0,\boldsymbol1_{n};1-z(1-z))\\&{}-G(1,z,0,\boldsymbol1_{n};1-z(1-z))-G(1,1,0,\boldsymbol1_{n};1-z(1-z))]\\={}&-\frac{G(z,z,\boldsymbol1_{n};1-z(1-z))+G(1,z,\boldsymbol1_{n};1-z(1-z))}{z}.
\end{split}
\end{align} Therefore,  the relation in \eqref{eq:G4nF4n} holds true.

Our next goal is to verify \eqref{eq:C4nReg}, so that\begin{align}
\mathscr A_{4,n}=-\frac{n!}{3}\RE\left[\Reg_{z\to1\pm i0^+}\mathscr G_{4,n}(z)-\Reg_{z\to1\pm i0^+}\mathscr C_{4,n}(z)\right]\in\mathfrak Z_{n+4}(1).
\end{align}Towards this end, we will first show how to evaluate  $ \Reg_{z\to1\pm i0^+}\mathscr C_{4,n}(z)$ in  Panzer's  \texttt{HyperInt}  package \cite{Panzer2015}, and then justify the algorithm in the framework of   MZV characterizations. Pick a variation on the counter-term $ \mathscr C_{4,n}(z)$ as follows:\begin{align}
\begin{split}f(u,z)\colonequals{}&\frac13[7G(z,1,1,0,\boldsymbol 1_n;1-u)-2G(z,z,1,0,\boldsymbol 1_n;1-u)\\{}&-G(z,1,z,0,\boldsymbol 1_n;1-u)-G(1,z,1,0,\boldsymbol 1_n;1-u)-3G(1,1,1,0,\boldsymbol 1_n;1-u)].
\end{split} 
\end{align}We claim that   $ \Reg_{z\to1+i0^+}\mathscr C_{4,n}(z)$  is equal to the output of the following \texttt{Maple} code:\begin{quote}\texttt{eval(regHlog(eval(fibrationBasis(eval(regHlog(eval(fibrationBasis(f(u, z), [u, z]), [u=1-z])), [Hlog(1-z, [1-z])=Hlog(z, [1])]), [z]), [z=1])), [Hlog(1, [1])=0, delta[1]=1]);}\end{quote}while $ \Reg_{z\to1-i0^+}\mathscr C_{4,n}(z)$ 
is generated after one  replaces \texttt{delta[1]=1} by \texttt{delta[1]=-1} in the code above. Here is a step-by-step justification of this algorithm:\begin{itemize}
\item 
The first step produces an explicit  fibration\begin{align}
\begin{split}&f(u,z)=\fib_{u,z}f(u,z)\\\in{}&\Span_{\mathbb Q}\left\{ \left(\frac{\pi i \I z}{|\I z|}\right)^\ell Z_m G(\beta_{1},\dots,\beta_{n+4-\ell-m};u)\middle|\begin{smallmatrix}\ell,m,n-\ell-m\in\mathbb Z_{\geq0}\\Z_m\in\mathfrak Z_m(1)\\\beta_1,\dots,\beta_{n+4-\ell-m}\in\{0,1,1-z\}\end{smallmatrix}\right\}
\end{split}
\end{align}through the operation \texttt{fibrationBasis(f(u, z), [u, z])}, where $ \frac{\I z}{|\I z|}$ is represented by \texttt{delta[z]} in the \texttt{Maple} code. Here, the fibration provides an identity involving bivariate GPLs, in which we may set $ u=z(1-z)$ to recover the original counter-term $\mathscr C_{4,n}(z)$.\item The second step gets us ready  for the regularized expansion 
with respect to $ u=z(1-z)=1-z+O((1-z)^2)$. By the code \texttt{eval(regHlog(eval(..., [u=1-z])),
[Hlog(1-z, [1-z])=Hlog(z, [1])])}, we trick  the program into  unshuffling  the GPLs, so that all the positive integer powers of \begin{align}
G(1-z;u)=G(1-z;z(1-z))=G(1;z)
\end{align} are highlighted, in the guise of manipulating ``\texttt{Hlog(1-z, [1-z])}'' [which is\ a  divergent GPL ``$ G(1-z;1-z)$''].\item Hitting \texttt{eval(regHlog(eval(fibrationBasis(..., [z]), [z=1])), [Hlog(1, [1])=0, delta[1]=1])} on the outcome of the last step, we eventually extract the regularized limit  $ \Reg_{z\to1-i0^+}\mathscr C_{4,n}(z)$. Since the fibration with respect to $z$ leaves us GPLs  whose parameters belong to  $ \{0,1\}$, we  end up with a $ \mathbb Q$-linear combination of MZVs.
\end{itemize}This concludes our proof.
\end{proof}With the algorithms outlined in the last two propositions, we may verify all the MZV expressions\footnote{The choices of  representative MZVs up to weight 10 are identical across the works of  Au \cite{Au2025a,Au2022a}, Broadhurst--Kreimer \cite{BroadhurstKreimer1997}, 
Panzer \cite{Panzer2015}, and Schnetz \cite{Schnetz2010}. Opinions differ at higher weights. At 
weight 11, we  follow the convention of Au--Panzer, instead of Broadhurst--Kreimer--Schnetz, picking $ \zeta_{5,3,3}$ rather than $ \zeta_{3,5,3}=\frac{13\cdot23\zeta _{11}}{2}-2\zeta _{5,3,3}-3\cdot5\pi ^{2} \zeta _{9}+ \zeta _{5,3}\zeta _{3}-\frac{\pi ^{4} \zeta _{7}}{3\cdot5}+\frac{\pi ^{6} \zeta _{5}}{3^{3}\cdot7}$, so as to avoid a non-monotone sequence of integers in the subscript. 
At weight 12, we adhere to the preference of Broadhurst--Kreimer--Schnetz (which favored $ \zeta_{4,4,2,2}$), instead of  Au--Panzer \Big(which used $ \zeta_{6,4,1,1}=-\frac{199\zeta _{9,3}}{2^{4}\cdot3}+\frac{3\zeta _{4,4,2,2}}{2^{3}}+\frac{3\pi ^{2} \zeta _{7,3}}{7}-\frac{1439\zeta _{9} \zeta _{3}}{2^{4}\cdot3^{2}}+\frac{11\pi ^{4} \zeta _{5,3}}{2^{4}\cdot3\cdot5}+\frac{3^{3}\cdot5\zeta _{7}\zeta _{5} }{2^{3}}-\frac{ \pi ^{2} \zeta _{7} \zeta _{3}}{2^{2}\cdot3}-\frac{43\pi ^{2} \zeta _{5}^2}{2^{3}\cdot3\cdot7}+\frac{101\pi ^{4} \zeta _{5} \zeta _{3}}{2^{4}\cdot3^{2}\cdot5}-\frac{\zeta _{3}^4}{2^{3}\cdot3}-\frac{ \pi ^{6} \zeta _{3}^2}{2^{2}\cdot3^{3}\cdot7}-\frac{4562213\pi ^{12}}{2^{7}\cdot3^{7}\cdot5^{3}\cdot7^{2}\cdot11\cdot13}$\Big), so that our entry for $ \mathscr A_{4,8}$ could have more integer coefficients. Unlike our choices of 
representative MCVs 
in Tables \ref{tab:MCV2345}--\ref{tab:MCV8}, the subscripts of  representative MZVs  in all these cited works are not prioritized by their lexicographic orders.  Broadhurst--Kreimer \cite{BroadhurstKreimer1997}  chose $ \zeta_{5,3}$ and $ \zeta_{7,3}$ over their lexicographic precedents $ \zeta_{6,2}=-\frac{2\zeta_{5,3}}{5}+2\zeta _{5} \zeta_{3}-\frac{7\pi ^8}{2^{3}\cdot3^{3}\cdot5^{3}}$ and $ \zeta_{8,2}=-\frac{2\zeta_{7,3}}{7}+2\zeta _{7} \zeta_{3}+\frac{2^{3} \zeta _{5}^2}{7}-\frac{17\pi ^{10}}{2\cdot3^{4}\cdot5\cdot7^{2}\cdot11}$, probably due to the associations of these constants with $ (4,3)$ and $(5,3)$ torus knots, or due to the appearances of $ \zeta _{5} \zeta_{3}$ and $ \zeta _{7} \zeta_{3}$ in related formulae.

 } in Table \ref{tab:A2nA4n}, oblivious to the MCV reductions described in Tables \ref{tab:MCV2345}--\ref{tab:MCV8}. All the  computations of MZVs up to weight 12 are supported natively by Panzer's   \texttt{HyperInt}  package \cite{Panzer2015}.
\section{Discussions and outlook\label{sec:discussion}}\subsection{Log-sine integrals revisited}  Zucker \cite{Zucker1985}  and Borwein--Broadhurst--Kamnitzer \cite{BorweinBroadhurstKamnitzer2001}
derived \eqref{eq:AsLs} through iterated applications of integration by parts.
 Here, we describe a non-iterative approach, in a similar spirit to \cite[\S3]{CGZ2024}.
Starting from \begin{align}
\begin{split}\mathscr A_{s}={}&\frac{1}{2}\int_0^1\frac{\Li_{s-1}(t(1-t))}{t(1-t)}\D t\xlongequal{t=\frac{1}{1+u}}\frac{1}{2}\int_0^\infty\Li_{s-1}\left( \frac{u}{(1+u)^2} \right)\frac{\D u}{u}\\={}&\frac{1}{4\pi i}\left( \int_{ \infty +i0^+}^0+\int_0^{\infty-i0^+} \right)\Li_{s-1}\left( \frac{u}{(1+u)^2} \right)\frac{\log(-u)\D u}{u}
\end{split}
\end{align}and the jump discontinuity\begin{align}
\Li_{s-1}(x+i0^+)-\Li_{s-1}(x-i0^+)=\frac{2\pi i}{(s-2)!}\log^{s-2}x
\end{align}for $ x>1$, we may deform the contour into a tight loop wrapping around the circular arc $ \{e^{i\phi}|2\pi/3\leq\phi\leq4\pi/3\}$ \big[on which one has $ \frac{(1+u)^2}{u}=4\cos^2\frac{\phi}{2}$\big], thereby proving \eqref{eq:AsLs}. 

The log-sine integrals\begin{align}
\Ls_{n}^{(k)}(\sigma)\colonequals -\int_0^\sigma\theta^k\log^{n-1-k}\left\vert 2\sin\frac{\theta}{2} \right\vert\D\theta
\end{align}have been studied extensively by Lewin \cite{Lewin1958,Lewin1981} and Borwein--Straub \cite{BorweinStraub2011ISSAC,BorweinStraub2012Mahler,BorweinStraub2015Snp}.
The reductions of $\Ls_{n}^{(k)}(\sigma) $ to Nielsen-type polylogarithms are automated by  the 
\texttt{logsine} package  \cite{BorweinStraub2011ISSAC}. Thus, one obtains  explicit closed-form evaluations like $
\mathscr A_{16}=-\frac{(-2)^{14}}{14!}\Ls_{16}^{(1)}\left( \frac{\pi}{3} \right)\in\Span_{\mathbb Q}\{\RE\mu_{15,1},$ $\RE \mu_{13,1,1,1},$ $\pi \I\mu_{14,1}, $ $\pi\I \mu_{12,1,1,1},$ $\pi^2\RE\mu_{13,1},$ $\pi^3\I\mu_{12,1}\}+\Span_{\mathbb Q}\left\{\prod_{j=1}^n\zeta_{a_j}\middle|\begin{smallmatrix}a_j\in\mathbb Z\cap[2,13]\\\sum_{j=1}^n a_j=16\end{smallmatrix}\right\}
$, without going through the trouble of Lyndon word decompositions or Nielsen reductions of MCVs associated with non-monotone Lyndon words.

The techniques presented in the last two paragraphs will not work for $  
\mathscr A_{s,n}$ when $n$ is a positive integer.  In fact, Table \ref{tab:Asn} tells us that the CMZV structures of the deformed series $ \mathscr A_{s,n}$ are richer than the Ap\'ery-like series $ \mathscr A_s$, where $ s\in\{3\}\cup\mathbb Z_{>4}$. If we assume that the Deligne--Goncharov bound $ \dim_{\mathbb Q}\mathfrak Z_4(3)\leq16$ (or  $ \dim_{\mathbb Q}\mathfrak Z_4(6)\leq  55$) \cite[Corollaire 5.25(iii)]{DeligneGoncharov2005} sharpens to an equality, then  $ \mathscr A_{3,1}=-2(\I\mu_{2})^{2}$ is not a $ \mathbb Q$-linear combination of the following log-sine integrals:{\allowdisplaybreaks
\begin{align}
\Ls_4^{(0)}\left( \frac{\pi}{3} \right)={}&\frac{9\I\mu_4}{2}+\frac{\pi\zeta_3}{2},\\\Ls_4^{(1)}\left( \frac{\pi}{3} \right)={}&-\frac{17 \pi ^4}{6480},\\\Ls_{4}^{(2)}\left( \frac{\pi}{3} \right)={}&-2\I\mu_4+\frac{2\pi\zeta_3}{9}+\frac{\pi^2\I\mu_2}{9},\\\Ls_4^{(3)}\left( \frac{\pi}{3} \right)={}&-\frac{\pi ^4}{324}.
\end{align}
}Barring any miraculous reductions\footnote{The number $ \RE\mu_{4,2}$ has appeared in Laporta's $ V_{6b}$ term \cite[(15)]{Laporta:2017okg} for the 4-loop contribution to electron's anomalous magnetic moment. Had this number been reducible to well-known constants of simpler shapes, it would have   been detected (with a high likelihood) in Laporta's extensive studies. It is also worth noting that in both Laporta's $ V_{6b}$  \cite[(15)]{Laporta:2017okg} and our $ \mathscr A_{3,3}$ (see Table \ref{tab:Asn}), the number  $ \RE\mu_{4,2}$ appears in a combination $ \RE\mu_{4,2} +\I\mu_4\I\mu_2$. } of $ \RE\mu_{4,2}$ into a polynomial of MCVs in simpler shapes, the number $ \mathscr A_{3,3}$ does not appear to live in an algebra generated by log-sine integrals evaluated at $ \frac{\pi}{3}$.

\subsection{Broadhurst's conjectures on MCVs\label{subsec:conjMDV}}

The  MCVs\footnote{Let $ \mathfrak C_w\colonequals \Span_{\mathbb Q}\{\mu_{a_1,\dots,a_n}|a_1+\dots+a_n=w\}$ be the $ \mathbb Q$-vector space spanned by MCVs of weight $w$. It has been conjectured (see \cite[Conjecture 5.1]{BorweinBroadhurstKamnitzer2001} or \cite[Conjecture 1]{Broadhurst2014MDV}) that $ \dim_{\mathbb Q}\mathfrak C_w\overset?=\dim_{\mathbb Q}\mathfrak Z_w(2)\overset?=F_{w+1}$, where $ F_{w+1}$ is the $ (w+1)$-st Fibonacci number. By comparison, there are provable upper bounds  $ \dim_{\mathbb Q}\mathfrak Z_w(3)\leq 2^{w}$  and $ \dim_{\mathbb Q}\mathfrak Z_w(6)\leq F_{2w+2}$ \cite[Corollaire 5.25(iii)]{DeligneGoncharov2005} that is probably sharp (according to numerical evidence for small $w$).} (and more generally, the CMZVs of level 6) feature prominently in high energy physics \cite{Broadhurst1999,HennSmirnovSmirnov2017,Laporta:2017okg,Schnetz2018}.
 Numerical experiments revealed many algebraic relations among these CMZVs, most of which are heretofore unproven. 

Broadhurst  \cite[Conjecture 4]{Broadhurst2014MDV} predicted that for each odd weight $ w>1$, there exists a unique (up to an overall scaling) $ \mathbb Z$-linear combination of \begin{align}
\I G(\underset{w-2k-1}{\underbrace{0,\dots,0}},\varrho,\underset{2k-1}{\underbrace{0,\dots,0}},1;1),\quad \text{where }1\leq k\leq \frac{w-1}{2},
\end{align}that evaluates to a rational multiple of $ \pi^w$. Broadhurst deemed this conjecture as  ``[lying]  at the heart of'' his paper \cite{Broadhurst2014MDV}. For weights 3 and  5, Au's    \texttt{MultipleZetaValues} package \cite{Au2025a,Au2022a}  leaves us \begin{align}
\I G(\varrho,0,1;1)=-\frac{7 \pi ^3}{2^{2}\cdot3^{4}}
\end{align}and \begin{align}
\left\{\begin{array}{@{}r@{{}={}}l}
\I G(0,0,\varrho,0,1;1) & -\frac{2\I\mu_{4,1}}{3}-\frac{2\zeta_{3}\I\mu_2}{3}+\frac{5 \pi ^5}{3^{7}}, \\
\I G(\varrho,0,0,0,1;1) & \I\mu_{4,1}+\zeta_{3}\I\mu_2-\frac{173 \pi ^5}{2^{3}\cdot3^{6}\cdot5}, \\
\end{array}\right.
\end{align}which add up to \begin{align}
3\I G(0,0,\varrho,0,1;1)+2\I G(\varrho,0,0,0,1;1)=-\frac{73 \pi ^5}{2^{2}\cdot3^{6}\cdot5}.
\end{align}At weight 7,  with the  packages  \texttt{HPL} \cite{Maitre2005,Maitre2012} and \texttt{HyperInt} \cite{Panzer2015}, one may implement the procedures in \S\ref{sec:evalClGl} [up to  Proposition \ref{prop:Lyn78}(a)], and verify that \begin{align}
\left\{\begin{array}{@{}r@{{}={}}l}
\I G(0,0,0,0,\varrho,0,1;1) & -5 \I\mu _{6,1}-\I\mu _{5,2}-\frac{29 \zeta _{5} \I\mu _{2}}{2\cdot3^{3}}-3 \zeta _{3} \I\mu _{4}+\frac{7\cdot13 \pi ^7}{2^{5}\cdot3^{7}}, \\
\I G(0,0,\varrho,0,0,0,1;1) & \frac{241\text{Im$\mu $}_{6,1}}{3\cdot19}+\frac{2\cdot3^{2} \I\mu _{5,2}}{19}+\frac{7\zeta _{3} \I\mu _{4}}{3}-\frac{2^{4}\cdot17\pi ^{7}}{3^{10}\cdot5}, \\\I G(\varrho,0,0,0,0,0,1;1) & \I\mu _{6,1}+\zeta _{5} \I\mu _{2}+\zeta _{3} \I\mu _{4}-\frac{4919 \pi ^7}{2^{3}\cdot3^{9}\cdot5\cdot7},
\end{array}\right.
\end{align}so one gets\begin{align}
\begin{split}&2\cdot3^{3}\I G(0,0,0,0,\varrho,0,1;1)+3\cdot19\I G(0,0,\varrho,0,0,0,1;1)+29\I G(\varrho,0,0,0,0,0,1;1)\\={}&-\frac{90163 \pi ^7}{2^{4}\cdot3^{9}\cdot5\cdot7}.
\end{split}
\end{align}Thus far, we have given a partial answer to Broadhurst's question on the existence of $ \mathbb Z$-linear combinations. The corresponding uniqueness part still hinges on some open questions regarding the $ \mathbb Q$-linear independence among  MCVs.

Both   \cite[Conjecture 4]{Broadhurst2014MDV} and \cite[Conjecture 5]{Broadhurst2014MDV} in Broadhurst's work concern Clausen descents similar to the patterns in \eqref{eq:A4nMZV} and \eqref{eq:A2nMZV}. For weights up to $8$, they are consistent with explicit MCV reductions (Propositions \ref{prop:Lyn56} and \ref{prop:Lyn78}). It is perhaps appropriate to ask whether one can prove these conjectures for all the possible weights by introducing counter-terms,
in the spirit  of Propositions \ref{prop:A2nMZV} and \ref{prop:A4nMZV}. 

In \cite[Conjecture 2]{Broadhurst2014MDV} (see also \cite[Conjecture 5.2]{BorweinBroadhurstKamnitzer2001}), Broadhurst proposed an empirical formula \begin{align}
\prod_{w=2}^\infty\prod_{d=1}^\infty\big(1-x^w y^d\big)^{N_{w,d}}\overset?=1-\frac{x^2 y}{1-x}
\end{align}for the number $ N_{w,d}$ of primitive MCVs at any given weight $ w$ and depth $ d$. One may recast this empirical formula  into   (cf.\ \cite[(13)]{Broadhurst2014MDV}) \begin{align}
N_{w,d}\overset?=\frac{1}{w}\sum_{n|w,n|d}\mu(n)\binom{\frac{w}{n}}{\frac{d}{n}},
\end{align}where $w>2d $, and $\mu(n)$ is the M\"obius function. In particular, inspired by  the expansion $ 1-\frac{x^2 y}{1-x}=\big(1-x^2 y\big) \big(1-x^3 y\big) \big(1-x^4 y\big) \big(1-x^5 y\big) \big(1-x^5 y^2\big) \big(1-x^6 y\big) \big(1-x^6 y^2\big) \big(1-x^7 y\big) \big(1-x^7 y^2\big)^2 \big(1-x^7 y^3\big) \big(1-x^8 y\big) \big(1-x^8 y^2\big)^2 \big(1-x^8 y^3\big)^2\big[1+O\big(x^9y\big)\big]$, one may enumerate  primitive MCVs up to level 8, in Table \ref{tab:primMCVs}. The choice of either real or imaginary part in each individual entry follows that of \cite[Conjecture 3]{Broadhurst2014MDV}. Here, we have suppressed the odd zeta values $ \zeta_3$, $ \zeta_5$, $ \zeta_7$ (resp.\ the MZV $ \zeta_{5,3}$) in favor of the real parts of MCVs $\RE\mu_3 $, $ \RE\mu_5$, $\RE\mu_7$ (resp.\  $ \RE\mu_{7,1}$). \begin{table}[h]\caption{An empirical list of primitive MCVs up to weight $8$\label{tab:primMCVs}}{\footnotesize\begin{align*}\begin{array}{c|lll}\hline\hline
w & d=1 &d=2&d=3\\\hline
2 & \I\mu_{2} \\
3 & \RE\mu_{3} \\
4 & \I\mu_{4} \\
5 & \RE\mu_{5}&\I\mu_{4,1}\ \\
6 & \I\mu_{6}&\RE\mu_{4,2}\ \\
7 & \RE\mu _7&\I\mu _{6,1},\I\mu _{5,2}&\RE\mu _{4,1,2} \\
8 & \I\mu_8&\RE\mu_{7,1},\RE\mu_{6,2}&\I\mu _{6,1,1},\I\mu _{4,2,2} \\\hline\hline
\end{array}\end{align*}}\end{table}  

The primitive MCVs listed in Table \ref{tab:primMCVs} are compatible with our choices of irreducible MCVs in Tables \ref{tab:MCV2345}--\ref{tab:MCV8}
(see Footnote \ref{fn:irred_prim} for the differences between ``irreducible'' and ``primitive''). In \S\ref{subsec:algLyndonMCV}, we did not study algebraic relations among MCVs above weight $8$, due to the  current lack of automated Lyndon decompositions at higher weights. For future work, one may wish to check whether the procedures in \S\ref{subsec:algLyndonMCV}
remain effective for Broadhurst's MDV datamine \cite{Broadhurst2014MDV} beyond weight $8$.

\subsection{A related class of infinite series involving $ \binom{4k}{2k}$}
The method developed in this article draws on the coalescence of several software packages. In this subsection, we will illustrate this methodology further with the computations of \eqref{eq:4k2k} for $n\in\{0,1,2,3,4\}$. 

While it is not hard to identify the infinite series \eqref{eq:4k2k} with \begin{align}
\mathscr I_n\colonequals \left.\!\frac{\partial^n}{\partial x^n}\frac{10x-1}{\binom{4x}{2x}}\right|_{x=0}+\int_{0}^1\sum_{k=1}^\infty\left.\!\frac{\partial^n}{\partial x^n}\left[x(10x-1)t^{2x-1}(1-t)^{2x-1}\right]\right|_{x=k}\D t
\end{align}and compute \begin{align}
\mathscr I_0=-1+\int_0^1\frac{t (1-t) \left(11 t^4-22 t^3+11 t^2+9\right)\D t}{\left[1-t^2 (1-t)^2\right]^3}=\frac{2^{2} \pi }{3^{2} \sqrt{3}}\in i\sqrt{3}\mathfrak Z_1(3),
\end{align} it is a challenging task to  evaluate $
\mathscr I_n$ for $ n\in\{1,2,3,4\}$ if one has only a single platform at disposal. Concretely speaking, we face the following obstacles:\begin{itemize}
\item 
Au's \texttt{MultipleZetaValues} package \cite{Au2025a,Au2022a} refuses to process $\mathscr I_n$ for $ n\in\{1,2,3,4\}$    
through \texttt{MZIntegrate}. This hiccup is actually understandable: the roots of the polynomial $ 1-t^2 (1-t)^2$ are $ t=\frac{1\pm\sqrt{3}i}{2}$ and $ t=\frac{1\pm\sqrt{5}}{2}$, which call for CMZVs of levels $ \geq30$, beyond the scope of Au's
package.\footnote{Partial fractions may decompose $ \mathscr I_n$ into a finite sum of level-$5$ and level-$6$ CMZVs, but do not remove the obstacle completely: the  \texttt{MultipleZetaValues} package \cite{Au2025a,Au2022a} supports reductions in $ \mathfrak Z_w(5)$ only for $ w\in\{1,2,3,4\}$.}\item The native integrator of \texttt{Mathematica} struggles  with the evaluations of 
definite integrals involved in $ \mathscr I_n$ (where $n\in\mathbb Z_{>0}$), even though the integrands are just elementary functions in the ring $ \mathbb Q(t)[\log t,\log(1-t)]$.\item The native integrator of \texttt{Maple} (which relies on \cite{Frellesvig2018Maple,Frellesvig2016} for GPL manipulations)
handles $ \mathscr I_1$ very quickly, but  gives an incorrect answer \begin{align}
\frac{32}{3}+\frac{8\sqrt{3}i}{9}\Li_2\left( \frac{1-\sqrt{3}i}{2} \right)-\frac{8\sqrt{3}i}{9}\Li_2\left( \frac{1+\sqrt{3}i}{2} \right).
\end{align} The correct one should be\begin{align}
\frac{32}{3}-\frac{8\sqrt{3}i}{9}\Li_2\left( \frac{1-\sqrt{3}i}{2} \right)+\frac{8\sqrt{3}i}{9}\Li_2\left( \frac{1+\sqrt{3}i}{2} \right).
\end{align}  
\end{itemize}
We overcome these difficulties by the following makeshift procedures:\begin{itemize}
\item In Panzer's \texttt{HyperInt}  package \cite{Panzer2015} for \texttt{Maple}, evaluate  $\mathscr I_n$ for $ n\in\{1,2,3,4\}$    over the splitting field $ \mathbb Q(i,\sqrt{3},\sqrt{5})$. The results are colossal expressions of the following structure:\begin{align}
\Span_{\mathbb Q(i,\sqrt{3},\sqrt{5})}\left\{ G(\alpha_1,\dots,\alpha_n;1)[G(0;2)]^m \middle|\begin{smallmatrix}m\in\mathbb Z_{\geq0}, n\in\mathbb Z_{>0}\\\alpha_1,\dots,\alpha_n\in\left\{0,\varrho,\omega,-1,\frac{1}{\omega},\frac{1}{\varrho}\right\}\end{smallmatrix}\right\},
\end{align}
where $ \varrho\colonequals e^{\pi i/3}$ and $ \omega\in e^{2\pi i/3}$.  \item Export the aforementioned results to \texttt{Mathematica}, and simplify the GPLs with \texttt{MZExpand} in Au's \texttt{MultipleZetaValues} package \cite{Au2025a,Au2022a}. 
\end{itemize}
The net outputs are satisfyingly concise:{\allowdisplaybreaks
\begin{align}\mathscr I_1
={}&-\frac{2^{3}\I\Li_2(\omega)}{\sqrt{3}}+\frac{2^{5}}{3} \in i\sqrt{3}\mathfrak Z_2(3)+\mathbb Q,\\
\mathscr I_2={}&\frac{2^{6} \pi ^3}{3^{4} \sqrt{3}}\in i\sqrt{3}\mathfrak Z_3(3),\\\mathscr I_3
={}&-\frac{2^{4}\cdot3^{3}\I\Li_4(\omega)}{\sqrt{3}}+\frac{2^{5}\pi^{2}\I\Li_2(\omega)}{\sqrt{3}}-\frac{2^{7} \pi ^2}{3}\in i\sqrt{3}\mathfrak Z_4(3)+\mathfrak Z_2(1),\\\begin{split}\mathscr I_4
={}&-\frac{2^{10}\cdot3^{2}\sqrt{3}\I[\Li_{3,2}(\omega,1)+3\Li_{4,1}(\omega,1)]}{13} -\frac{2^{10}\cdot3\zeta_{3}\I\Li_2(\omega)}{\sqrt{3}}+\frac{2^{11}\cdot5 \pi ^5}{3^{4}\cdot13 \sqrt{3}}\\&{}+2^{12} \zeta_{3}\in i\sqrt{3}\mathfrak Z_5(3)+\mathfrak Z_3(1),
\end{split}
\end{align}
}where $ \omega=e^{2\pi i/3}$. In view of \eqref{eq:mu2}--\eqref{eq:mu5}, one can also reexpress these four results through MCVs. We close by pointing out that the CMZV structure of \eqref{eq:4k2k} for $n>4$ is  an uncharted territory, where MCVs may or may not suffice.



\end{document}